\documentclass{amsart}

\usepackage{allneeded}

\begin{document}

\title{The Lagrangian Conley Conjecture}%

\author{Marco Mazzucchelli}%

\address{Dipartimento di Matematica, Universit\`a di Pisa, 56127 Pisa, Italy}%
\email{mazzucchelli@mail.dm.unipi.it}%


\date{October 12, 2008. \emph{Last update:} November 22, 2010}%

\maketitle
\begin{abstract}
We prove a Lagrangian analogue of the Conley conjecture: given a 1-periodic Tonelli Lagrangian with global flow on a closed configuration space, the associated Euler-Lagrange system has infinitely many  periodic solutions. More precisely, we show that there exist infinitely many contractible integer periodic solutions with a priori bounded mean action and either infinitely many of them are 1-periodic or they have unbounded period.
\end{abstract}

\begin{quote}
\begin{footnotesize}
\tableofcontents
\end{footnotesize}
\end{quote}

\section{Introduction}
An old problem in classical mechanics is the existence and the number of periodic orbits of a   mechanical system. As it is well known, there are two dual formulations of classical mechanics, namely the Lagrangian one and the Hamiltonian one. If a system is constrained on a manifold, say $M$, then it is described by a (possibly time-dependent) Lagrangian defined on the tangent bundle of $M$ or, dually, by a Hamiltonian defined on the cotangent bundle of $M$.  A classical assumption is that a Lagrangian function $\Lagr:\R\times \Tan M\to\R$ is Tonelli with global flow. This means that $\Lagr$ has fiberwise superlinear growth,  positive definite fiberwise Hessian, and every maximal integral curve of its associated Euler-Lagrange vector field has all of $\R$ as its domain of definition. The Tonelli class is particularly important in Lagrangian dynamics: in fact, whenever a Lagrangian function is fiberwise convex, its Legendre transform defines a diffeomorphism between the tangent and cotangent bundles of the configuration space if and only if the Lagrangian function belongs to the Tonelli class. In other words, the Tonelli Lagrangians constitute the broadest family of fiberwise convex Lagrangian functions for which the Lagrangian-Hamiltonian duality, given by the Legendre transform, occurs. Furthermore,  the Tonelli assumptions imply existence and regularity results for action minimizing orbits joining two given points on the configuration space. For a comprehensive reference on Tonelli Lagrangians, we refer the reader to the forthcoming book by Fathi \cite{b:Fa}.

The problem of the existence of infinitely many periodic orbits, for  Hamiltonian systems on the cotangent bundle of  closed manifolds, is a  non-compact version of the celebrated Conley conjecture \cite{b:Co}, which goes back to the eighties. In its original form, the conjecture states that every Hamiltonian diffeomorphism on the standard symplectic torus $\mathds T^{2n}$ has infinitely many periodic points.
Under a non-degeneracy assumption on the periodic orbits, the conjecture was soon confirmed by Conley and Zehnder \cite{b:CZ2}, and then extended to aspherical closed symplectic manifolds  in 1992 by Salamon and Zehnder \cite{b:SZ}. The full conjecture was  established in 2004 by Hingston \cite{b:Hi}  for the torus case, and in 2006 by Ginzburg \cite{b:Gi} for the aspherical closed case. Other related results are contained in \cite{b:FS, b:GG, b:Gu, b:HZ, b:Sch2, b:Vi2}.

In the Lagrangian formulation, the periodic orbits are extremal points of the Lagrangian action functional, and several sophisticated methods from Morse theory may be applied in order to assert their existence. However, these methods only work for classes of Lagrangians that are significantly smaller than the Tonelli one. Along this line, in 2000  Long \cite{b:Lo} proved the existence of infinitely many periodic orbits for Lagrangian systems associated to fiberwise quadratic Lagrangians on the torus $\T^n$. More precisely, he proved the result for $C^3$ Lagrangian functions $\Lagr:\R/\Z\times\Tan\T^n\to\R$ of the form
\[ \Lagr(t,q,v)=\abra{ A(q)v,v } + V(t,q),\s\s\forall (t,q,v)\in\R/\Z\times\Tan\T^n, \]
where $A:\T^n\to\mathrm{GL}(n)$ takes values in the space of positive definite  symmetric matrices, $\abra{\cdot,\cdot}$ is the standard flat Riemannian metric on $\T^n$ and $V:\R/\Z\times\T^n\to\R$ is a $C^3$  function.  Recently, Lu \cite{b:Lu} extended Long's proof in many directions, in particular to the case of a general closed configuration  space. Other related results concerning autonomous Lagrangian systems are contained in  \cite{b:CT, b:LL, b:LW}.

In 2007, Abbondandolo and Figalli \cite{b:AbFig} showed how to apply some techniques from critical point theory  in the Tonelli setting.  They proved that, on each closed configuration space with finite fundamental group, every  Lagrangian system associated to a 1-periodic Tonelli Lagrangian with global flow admits an infinite sequence of 1-periodic solutions with diverging action. In this paper, inspired by their work and by Long's one, we address the problem of the existence of infinitely many periodic solutions of a 1-periodic Tonelli Lagrangian system on a general closed configuration space. 
   
\subsection{Main result}  Our main result   is the following.

\begin{thm}\label{t:Conley}
Let $M$ be a smooth   closed manifold,  $\Lagr:\R/\Z\times \Tan M\to \R$ a smooth 1-periodic  Tonelli  Lagrangian with global flow and    $\ACT\in\R$ a  constant greater than
\begin{align}\label{e:ACTconstant} 
\max_{q\in M} \gbra{ \int_0^1 \Lagr(t,q,0)\, \diff t}.
\end{align}
Assume that only finitely many contractible  1-periodic  solutions of the Euler-Lagrange system of $\Lagr$ have action less than $\ACT$.  Then, for each prime $p\in\N$, the Euler-Lagrange system of $\Lagr$ admits       infinitely many contractible periodic solutions with period that is a power of $p$ and mean action less than $\ACT$. 
\end{thm}

Here it is worthwhile to point out that the  infinitely many periodic orbits that we find are geometrically distinct in the phase-space  $\R/\Z\times\Tan M$ of our system, and the mean action of an orbit is defined as the usual Lagrangian action divided by the period of the orbit.

Theorem~\ref{t:Conley} can be equivalently stated in the Hamiltonian formulation. In fact, it is well known that the Legendre duality sets up a one-to-one correspondence
\[
\gbra{
\begin{tabular}{c}
$\Lagr:\R/\Z\times \Tan M\to \R$\\
Tonelli
\end{tabular}
}
\longleftrightarrow
\gbra{
\begin{tabular}{c}
$\Ham:\R/\Z\times \Tan^* M\to \R$\\
Tonelli
\end{tabular}
},
\]
where two correspondent functions $\Lagr$ and $\Ham$ satisfy  the Fenchel relations
\begin{equation}\label{e:Fenchel}
\begin{split}
\Ham(t,q,p)&=\max \gbra{p(v) - \Lagr(t,q,v)\,|\,v\in\Tan_qM},\s\s\forall (t,q,p)\in\Tan^*M,\\
\Lagr(t,q,v)&=\max \gbra{p(v) - \Ham(t,q,p)\,|\,p\in\Tan_q^*M},\s\s\forall (t,q,v)\in\Tan M.
\end{split}
\end{equation}  
Here, the definition of Tonelli Hamiltonian is the cotangent bundle analogue of the one of Tonelli Lagrangian: the Hamiltonian $\Ham:\R/\Z\times\Tan^*M\to\R$ is Tonelli when it has fiberwise superlinear growth and positive definite fiberwise Hessian. The Legendre duality also sets up  a one-to-one correspondence between the (contractible) integer-periodic solutions of the Euler-Lagrange system of $\Lagr$ and the (contractible) integer-periodic orbits of the Hamilton system of the dual $\Ham$. Therefore theorem~\ref{t:Conley} can be translated into the following.
\begin{thm}[Hamiltonian formulation]\label{t:ConleyTonHam}
Let $M$ be a smooth   closed manifold,  $\Ham:\R/\Z\times \Tan^* M\to \R$ a smooth 1-periodic  Tonelli  Hamiltonian with global flow and    $\ACT\in\R$ a  constant greater than
\begin{align*}
-\min_{q\in M} \gbra{ \int_0^1  \min_{p\in\Tan^*_qM} \gbra{ \Ham(t,q,p) }  \, \diff t}.
\end{align*}
Assume that only finitely many contractible  1-periodic  solutions of the Hamilton system of $\Ham$ have (Hamiltonian) action less than $\ACT$.  Then, for each prime $p\in\N$, the Hamilton system of $\Ham$ admits   infinitely many contractible periodic solutions with period that is a power of $p$ and (Hamiltonian) mean action less than $\ACT$. 
\end{thm}

Let $\Phi_{\Ham}^t$ be the Hamiltonian flow of $\Ham$, i.e.\ if $\Gamma:\R\to\Tan^*M$ is a Hamiltonian curve then $\Phi_{\Ham}^t(\Gamma(0)) = \Gamma(t)$ for each $t\in\R$. A Hamiltonian curve $\Gamma$ is $\tau$-periodic, for some integer $\tau$, if and only if its starting point $\Gamma(0)$ is a $\tau$-periodic point of the Hamiltonian diffeomorphism $\Phi_{\Ham}=\Phi_{\Ham}^1$, i.e.
\[
\underbrace
{\Phi_{\Ham}\circ...\circ \Phi_{\Ham}}
_{\tau\mbox{ }\mathrm{times}} 
(\Gamma(0)) =\Gamma(0).
\]
 Therefore, theorem~\ref{t:ConleyTonHam} readily implies the Conley conjecture for Tonelli Hamiltonian systems on the cotangent bundle of a  closed manifold.
\begin{cor}\label{c:ConleyTon}
Let $M$ be a smooth  closed manifold and  $\Ham:\R/\Z\times\Tan^*M\to\R$ a 1-periodic  Tonelli Hamiltonian  with global flow $\Phi_{\Ham}^t$. Then the Hamiltonian diffeomorphism $\Phi_{\Ham}^1$ has infinitely many periodic points.
\hfill$\qed$
\end{cor}

We shall prove theorem~\ref{t:Conley} by a Morse theoretic argument, inspired by a work of Long \cite{b:Lo}. The rough idea is the following: assuming by contradiction that the Euler-Lagrange system of a Tonelli Lagrangian $\Lagr$ admits only finitely many integer periodic solutions as in the statement, then it is possible to find a solution whose local homology persists under iteration, in contradiction with a homological vanishing property  (analog to the one proved by Bangert and Klingenberg \cite{b:BK} for the geodesics action functional).

Under the Tonelli assumptions we need to deal with several problems while carrying out the above scheme of the proof. These problems are mainly due to the fact that  a functional setting in which the Tonelli action functional is both regular (say $C^1$) and satisfies the Palais-Smale condition, the minimum requirements to perform Morse theory, is not known. To deal with these lacks, we apply the machinery of convex quadratic modifications introduced by Abbondandolo and Figalli  \cite{b:AbFig}: the idea consists in modifying the involved Tonelli Lagrangian outside a sufficiently big neighborhood of the zero section of $\Tan M$, making it fiberwise quadratic there. If we fix a period $\tau\in\N$ and an action bound, a suitable a priori estimate on the $\tau$-periodic  orbits with bounded action  allows to prove that these orbits must lie in the region where the Lagrangian is not modified. Some work is needed in order  to apply this argument in our proof, since the mentioned  a priori estimate holds only in a fixed period $\tau$, while we look for  orbits with arbitrarily high period. 

Moreover, we have to deal with some regularity issues in applying the machinery of critical point theory to the action functionals of the modified Lagrangians. In fact, the natural ambient to perform Morse theory with these functionals is the $W^{1,2}$ loop space, over which they are $C^1$ and satisfy the Palais-Smale condition. However, all the arguments that involve the Morse lemma (such as the vanishing of the local homology groups in certain degrees) are valid only for $C^2$ functionals. We will show that these arguments are still valid, developing an analog of the classical broken geodesic approximation of the loop space: we shall prove that the action sublevels deformation retract onto finite dimensional submanifolds of the loop space, over which the action functionals are $C^2$.

\subsection{Organization of the paper} 
In section~\ref{s:Prel} we set up the  notation and we give most of the preliminary definitions and results.
In the subsequent three sections we deal with Lagrangian functions that are convex quadratic-growth (the precise definition is given in section~\ref{s:Lagr_settings}). In section~\ref{s:discretizations} we introduce a discretization technique for the   action functional. In section~\ref{s:LocHomEmb} we prove an abstract Morse-theoretic result, which will be applied to the   action functional. In section~\ref{s:HomVanSec}, we prove a  vanishing result for elements of the relative homology groups of pairs of action sublevels under the iteration map. The three sections \ref{s:discretizations},  \ref{s:LocHomEmb} and \ref{s:HomVanSec} can be read independently from one another, with the only exception of  subsection~\ref{s:ApplMorseLagr}, which  requires section~\ref{s:discretizations}.  In section~\ref{s:conv_quad_modifications} we introduce the machinery of convex quadratic modifications of Tonelli Lagrangians, and we apply it to build suitable local homology groups and to prove a homological vanishing result for the Tonelli action. Finally, in section~\ref{s:proof} we  prove theorem~\ref{t:Conley}.\\

\begin{quote}
\footnotesize
\textbf{Acknowledgments.} I am indebted to  Alberto Abbondandolo for many fruitful conversations. This work was largely written when I was a visitor at Stanford University. I wish to thank Yakov Eliashberg for his kind hospitality.
\end{quote}

\section{Preliminaries}\label{s:Prel}

Throughout the paper, $M$ will be a  smooth $N$-dimensional closed  manifold, the configuration space of a Lagrangian  system, over which we will consider an arbitrary Riemannian metric $\abra{\cdot,\cdot}_\cdot$.

\subsection{The free loop space}\label{s:free_loop_space} 
We recall that a \textbf{free loop space} of $M$ is, loosely speaking, a set of maps from the circle $\T:=\R/\Z$ to $M$. Common examples are the spaces $C(\T;M)$ or $C^\infty(\T;M)$. For our purposes, a suitable free loop space will be $\W(\T;M)$, which is the space of absolutely continuous loops in $M$ with square-integrable weak derivative. It is well known that this space admits an infinite dimensional Hilbert manifold structure, for the reader's convenience we briefly sketch this argument here (see \cite[chapter~1]{b:Kl} for more details). First of all, we recall that that  the inclusions 
\[C^\infty(\T;M)\subseteq\W(\T;M) \subseteq C(\T;M)\] are dense homotopy equivalences. For each $\gamma\in \W(\T;M)$  we denote by $\W(\gamma^*\Tan M)$ the separable Hilbert space of  $W^{1,2}$-sections of the  pull-back vector bundle $\gamma^*\Tan M$. Its inner product, that we denote by $\aabra{\cdot,\cdot}_\gamma$, is given by 
\[\aabra{\xi,\zeta}_\gamma:=\int_0^1 \qbra{\abra{\xi(t),\zeta(t)}_{\gamma(t)} + \abra{\nabla_{t}\xi,\nabla_{t}\zeta}_{\gamma(t)} } \diff t,
\s\s
\forall \xi,\zeta\in\W(\gamma^*\Tan M),\]
where $\nabla_t$ denotes the covariant derivative with respect to the Levi-Civita connection on the Riemannian manifold  $(M,\abra{\cdot,\cdot}_\cdot)$. Now,  let $\epsilon>0$ be a constant smaller than the injectivity radius of $(M,\abra{\cdot,\cdot}_\cdot)$, and  $U_\epsilon:=\gbra{v\in\Tan_q M\,|\, q\in M,\  |v|_q<\epsilon}$. 
We define a bijective map 
\[
   \exp_\gamma:
   \W(\gamma^*U_\epsilon)
   \to 
   \catU_\gamma
   \subseteq\W(\T;M)
\] 
as
\[
(\exp_\gamma\xi)(t):=\exp(\xi(t)),\s\s\forall \xi\in \W(\gamma^*U_\epsilon),\ t\in\T,
\]
where $\W(\gamma^*U_\epsilon)\subseteq\W(\gamma^*\Tan M)$ is the open set of sections   that take values inside $U_\epsilon$. Then, the above mentioned differentiable structure on $\W(\T;M)$ is induced  by the atlas $\gbra{\exp_\gamma^{-1}:\catU_\gamma\to \W(\gamma^*U_\epsilon)\,|\,\gamma\in C^\infty(\T;M)}$. The tangent space of $\W(\T;M)$ at a loop $\gamma$ is given by $\W(\gamma^*\Tan M)$, and the above defined $\aabra{\cdot,\cdot}_\cdot$ is a Hilbert-Riemannian metric on $\W(\T;M)$. By means of this metric, $\W(\T;M)$ turns out to be a \textbf{complete} Hilbert-Riemannian manifold.

\begin{rem}
Notice that, whenever a smooth loop $\gamma$ is contractible, by means of a trivialization of $\gamma^*\Tan M$ we can identify  $\W(\gamma^*U_\epsilon)$ with an open neighborhood of $\bm0$ in the Hilbert space $\W(\T;\R^N)$. 
\end{rem}

For each  $\tau\in\N$, let $\TT\tau:=\R/\tau\Z$. Extending the definition given before, we introduce the \textbf{$\tau$-periodic free loop space} $\W(\TT\tau;M)$,  which is a complete Hilbert-Riemannian manifold with respect to the metric $\aabra{\cdot,\cdot}_\cdot$ given by 
\begin{multline*}
\aabra{\xi,\zeta}_\gamma:=\frac 1\tau \int_0^\tau 
\qbra{ \abra{\xi(t),\zeta(t)}_{\gamma(t)} + 
\abra{\nabla_t\xi,\nabla_t\zeta}_{\gamma(t)} }
\diff t, \\ 
\forall \gamma\in\W(\TT \tau;M),\ \xi,\zeta\in\W(\gamma^*\Tan M).
\end{multline*}
For each  $n\in\N$, we define the \textbf{$n^{\mathrm{th}}$-iteration map} \[\itmap n:\W(\TT\tau;M)\hookrightarrow\W(\TT{n\tau};M)\] by $\itmap n (\gamma):=\gamma\iter n$ for each $\gamma\in\W(\TT\tau;M)$, where $\gamma\iter n$ is given by the composition of $\gamma$ with the $n$-fold covering map of the circle $\TT\tau$. Analogously, for each $\gamma\in \W(\TT\tau;M)$, we define the $n^{\mathrm{th}}$-iteration map \[\Itmap n:\W(\gamma^*\Tan M)\hookrightarrow\W(\gamma\iter n{}^*\Tan M),\] which is a linear isometric embedding. It is straightforward to verify that \[\exp_{\gamma\iter n}\circ\Itmap n=\itmap n\circ\exp_\gamma\] and \[\diff\itmap n(\gamma)=\Itmap n,\] which implies that $\itmap n$ is a smooth isometric embedding.

\subsection{Lagrangian settings} \label{s:Lagr_settings}
The elements of the tangent bundle $\Tan M$ will be denoted by $(q,v)$, where $q\in M $ and $v\in \Tan_qM$. 
Let $\Lagr:\T\times\Tan M\to\R$ be  a smooth 1-periodic Lagrangian. We will be interested in integer periodic solutions $\gamma:\R\to M$ of the \textbf{Euler-Lagrange system} of $\Lagr$, which   can be written in local coordinates as
\begin{align} \label{e:EuleroLagrange}
    &\der{}{t} \pder{\Lagr}{v_j} (t,\gamma(t),\dot\gamma(t)) - \pder{\Lagr}{q_j} (t,\gamma(t),\dot\gamma(t)) = 0,
    &  j=1,...,N.
\end{align}
We denote by $\displaystyle\Phi_{\Lagr}^t:\Tan M\toup^\simeq\Tan M$ the associated \textbf{Euler-Lagrange flow}, i.e.\  
\[\Phi_{\Lagr}^t(\gamma(0),\dot\gamma(0))=(\gamma(t), \dot\gamma(t)),\] where $\gamma:[0,t]\to M$ is a solution of~\eqref{e:EuleroLagrange}.

In this paper we will consider two classes of   1-periodic  Lagrangian functions. A smooth Lagrangian  $\Lagr:\T\times\Tan M\to\R$ is called \textbf{Tonelli} when:
\begin{itemize}
  \item[\textbf{(T1)}] the fiberwise Hessian of $\Lagr$ is positive definite, i.e.\
\[ \sum_{i,j=1}^N \frac{\partial^2\Lagr}{\partial v_i\,\partial v_j}(t,q,v)\,w_iw_j > 0,\]
for all $(t,q,v)\in\T\times \Tan M$ and  $w=\sum_{i=1}^N w_i \pder{}{q_i}\in\Tan_qM$ with $w\neq0$;
  \item[\textbf{(T2)}] $\Lagr$ is fiberwise superlinear, i.e.
  \[\lim_{|v|_q\to\infty}  \frac{\Lagr(t,q,v)}{|v|_q}=\infty, \] 
for all $(t,q)\in\T\times  M$.
\end{itemize}
Moreover, we will always require that each Tonelli Lagrangian $\Lagr$  further satisfies:
\begin{itemize}
  \item[\textbf{(T3)}] the Euler-Lagrange flow of $\Lagr$ is global, i.e.\ $\Phi_{\Lagr}:\R\times \Tan M\to \Tan M$.
\end{itemize}

The second class of Lagrangians that we will be interested in consists of  smooth functions  $\Lagr:\T\times\Tan M\to\R$ satisfying:
\begin{itemize}
  \item[\textbf{(Q1)}]  there is a positive constant  $\ell_0$ such that 
\[ \sum_{i,j=1}^N \frac{\partial^2\Lagr}{\partial v_i\,\partial v_j}(t,q,v)\,w_iw_j \geq \ell_0 |w|_q^2,\]
for all $(t,q,v)\in\T\times \Tan M$ and $w=\sum_{i=1}^N w_i \pder{}{q_i}\in\Tan_qM$;
   \item[\textbf{(Q2)}] there is a positive constant $\ell_1$ such that    
\begin{align*}
   \abs{ \frac{\partial^2\Lagr}{\partial v_i\,\partial v_j}  (t,q,v)} & \leq\ell_1,\\
   \abs{ \frac{\partial^2\Lagr}{\partial q_i\,\partial v_j} (t,q,v)} & \leq\ell_1(1+\abs{v}_q),\\
   \abs{\frac{\partial^2\Lagr}{\partial q_i\,\partial q_j}(t,q,v)} & \leq\ell_1(1+\abs{v}_q^2),
\end{align*}
  for all $(t,q,v)\in \T\times\Tan M$ and $i,j=1,...,N$.  
\end{itemize}
In the following we will informally refer   to this latter class   as the class of \textbf{convex quadratic-growth} Lagrangians. Notice that, up to changing the constants $\ell_0$ and $\ell_1$, the above conditions \textbf{(Q1)} and \textbf{(Q2)} are independent of the choice  of the Riemannian metric and of the system of local coordinates used to express them. Moreover,  assumption \textbf{(Q1)} implies that $\Lagr$ is a Tonelli Lagrangian, hence this second class is contained in the first.

For each $\tau\in\N$, we define   the \textbf{mean   action functional} $\act\tau:\W(\TT\tau;M)\to\R$ by
\[\act\tau(\gamma)=\frac 1\tau\int_0^\tau \Lagr(t,\gamma(t),\dot\gamma(t))\, \diff t.\]
In the following we will simply call $\act\tau$  the \textbf{mean action} or just the \textbf{action}, and in period $1$ we will omit the superscript, i.e. $\Act:=\act1$. Since for all $n\in\N$ we have $\act{n\tau}\circ\itmap n=\act\tau$, if we see $\W(\TT\tau;M)$ as a submanifold of $\W(\TT{n\tau};M)$ via the embedding $\itmap n$, then $\act\tau$ is  the restriction of $\act{n\tau}$ to $\W(\TT\tau;M)$. It is well known that the $\tau$-periodic solutions of the Euler-Lagrange system~\eqref{e:EuleroLagrange}  are precisely the extremals of $\act\tau$. These extremals  turn out to be smooth, as it is guaranteed by the Tonelli assumptions. If the involved Lagrangian $\Lagr$ is  convex quadratic-growth, the associated action functional $\act\tau$ has good properties:  the fact that $\Lagr$ grows \emph{at most} quadratically  guarantees that $\act\tau$ is $C^1$ and twice Gateaux-differentiable, while the fact that $\Lagr$ grows \emph{at least} quadratically   implies that $\act\tau$ satisfies the Palais-Smale condition (see \cite{b:Be} or \cite[propositions 2.2 and  2.5]{b:AS2}). However, $\act\tau$ is $C^2$ if and only if  the restriction of $\Lagr$ to each fiber of $\Tan M$ is a polynomial of degree at most  $2$ (see \cite[proposition  2.3]{b:AS2}).

\begin{rem}
For simplicity, in this paper, all the Lagrangian functions are assumed to be smooth, i.e.\ $C^\infty$. This assumption can be easily weakened, but then one would have to care about technical issues due to the fact that the  solutions of the Euler-Lagrange system would not be $C^\infty$ anymore (they are  $C^r$ whenever the Lagrangian is $C^r$).
\end{rem}

\subsection{The Conley-Zehnder-Long index pair}\label{s:CZLetc}
Let  $\Lagr:\T\times \Tan M\to\R$   be a  1-periodic    Tonelli  Lagrangian. We denote by $\partial_v \Lagr(t,q,v)\in\Tan_q^*M$ the fiberwise derivative of $\Lagr$ at $(t,q,v)$, which is given in local coordinates by 
\[ \partial_v \Lagr(t,q,v)=\sum_{j=1}^N \pder{\Lagr}{v_j}(t,q,v)\, \diff q_j. \]
Under the Tonelli assumptions it is well known that the Legendre transform $\Leg:\T\times\Tan M\to\T\times \Tan^*M$, given by
\begin{align*}
\Leg(t,q,v)=(t,q,\partial_v \Lagr(t,q,v)),\s\s\forall(t,q,v)\in\T\times\Tan M,
\end{align*}
is a diffeomorphism (see \cite[theorem 3.4.2]{b:Fa}). This diffeomorphism allows to define a  Hamiltonian $\Ham:\T\times\Tan^*M\to\R$ by
\[\Ham(t, \Leg(t,q,v) )=\partial_v \Lagr(t,q,v) v-\Lagr(t,q,v),\s\s\forall(t,q,v)\in\T\times\Tan M.\]
The functions $\Lagr$ and $\Ham$ are said to be \textbf{Legendre-dual}, and they fulfill the Fenchel relations~\eqref{e:Fenchel}.

The cotangent bundle $\Tan^*M$, whose elements will be denoted by $(q,p)$, has a canonical symplectic form $\omega$ given in local coordinates by
\[ \omega = \sum_{j=1}^N \diff q_j \wedge \diff p_j. \]
The (time-dependent) \textbf{Hamiltonian vector field} $X_{\Ham}$ is defined as usual by $X_{\Ham}\lrcorner\,\omega=\diff \Ham$, and its flow $\Phi_{\Ham}^t$ is called the \textbf{Hamiltonian flow} of $\Ham$. It is well known that this latter is conjugated to the Euler-Lagrange flow $\Phi_{\Lagr}^t$ by the Legendre transform. In other words,  a curve $\gamma:[0,\tau]\to M$ is a solution of the Euler-Lagrange system of $\Lagr$ if and only if the curve $(\gamma,\rho):[0,\tau]\to\Tan^*M$, where $\rho(t):=\partial_v \Lagr(t,\gamma(t),\dot\gamma(t))$, is an integral curve of the Hamiltonian vector field $X_{\Ham}$. In particular, there is a one-to-one  correspondence between the $\tau$-periodic Euler-Lagrange orbits of $\Lagr$ and the $\tau$-periodic  Hamiltonian orbits of $\Ham$.

Let $\Tan^{\mathrm{ver}}\Tan^*M$ denote the vertical subbundle  of $\Tan\Tan^*M$, i.e.\ 
\[\Tan^{\mathrm{ver}}_{(q,p)}\Tan^*M=\ker(\diff\tau^*(q,p)),\s\s\forall(q,p)\in\Tan^*M,\] 
where  $\tau^*:\Tan^*M\to M$ is the projection of the cotangent bundle onto the base manifold.  Consider a $\tau$-periodic solution $\gamma$ of the Euler-Lagrange system of $\Lagr$ and its Hamiltonian correspondent $\Gamma=(\gamma,\partial_v \Lagr( \cdot, \gamma, \dot\gamma))$. If $\gamma$ is contractible, $\Gamma$ is contractible as well, and there exists a symplectic trivialization 
\[\phi:\TT\tau\times\R^{2N}\toup^\simeq \Gamma^*\Tan\Tan^*M\]
that maps  the vertical Lagrangian subspace  $\V^N:=\gbra{\bm0}\times\R^N\subset\R^{2N}$   to the vertical sub-bundle $\Gamma^*\Tan^{\mathrm{ver}}\Tan^*M$, more precisely
\begin{align}\label{e:trivVert}
\phi(\TT\tau\times\V^N)=\Gamma^*\Tan^{\mathrm{ver}}\Tan^*M,
\end{align}
see  \cite[lemma 1.2]{b:AbSc} for a proof. By means of this trivialization, the differential of the Hamiltonian flow along $\Gamma$ defines a path $\Gamma_\phi:[0,\tau]\to\Sp(2N)$ in the symplectic group, given by
\[ \Gamma_\phi(t):= \phi(t,\cdot)^{-1}\circ\diff\Phi_{\Ham}^t(\Gamma(0))\circ\phi(0,\cdot),\s\s\forall t\in[0,\tau]. \]
Notice that $\Gamma_\phi(0)$ is the identity matrix, hence $\Gamma_\phi$ has a well defined \textbf{Conley-Zehnder index} $\iota(\Gamma_\phi)\in\Z$.  We  denote by $\nu(\Gamma_\phi)\in\N\cup\gbra0$  the geometric multiplicity of 1 as an eigenvalue of $\Gamma_\phi$ (in particular, we set $\nu(\Gamma_\phi)=0$ if $1$ is not an eigenvalue of $\Gamma_\phi$).

\begin{rem}
Here, we are using the generalized notion of Conley-Zehnder index that is due to Long, see \cite{b:LoMaslov}. In case $\nu(\Gamma_\phi)=0$, the  index $\iota(\Gamma_\phi)$ coincides with the usual  Conley-Zehnder index, see \cite{b:CZ, b:SZ}.  
\end{rem}

The pair $(\iota(\Gamma_\phi),\nu(\Gamma_\phi))$ does not depend on the chosen symplectic trivialization $\phi$, as long as this latter satisfies \eqref{e:trivVert}, see \cite[lemma 1.3]{b:AbSc}. Hence, we can define the \textbf{Conley-Zehnder-Long index pair}   of the periodic orbit $\gamma$ as  $(\iota(\gamma),\nu(\gamma)):=(\iota(\Gamma_\phi),\nu(\Gamma_\phi))$. This  pair satisfies the following \textbf{iteration inequalities}
\begin{equation}
\label{e:indexiteration}
\begin{split}
  n\, \aiota(\gamma) - N \leq &\, \iota(\gamma\iter n) , \\
    & \, \iota(\gamma\iter n)+\nu(\gamma\iter n) \leq n\, \aiota(\gamma) + N,
\end{split}
\s\s\forall n\in\N
\end{equation}
where $N=\dim(M)$ as before, and  $\aiota(\gamma)\in\R$ is \textbf{mean Conley-Zehnder index}   of $\gamma$, given by
\[ \aiota(\gamma) = \lim_{n\to\infty} \frac{\iota(\gamma\iter n)}{n}.  \] 
We refer the reader to \cite[theorem~1]{b:LL1} or \cite[theorem~1.1]{b:LL2} for more details on these inequalities.

If the Lagrangian $\Lagr$ also happens to be convex quadratic-growth, as we have already remarked in the previous section, its action functional $\act\tau$ is $C^1$ and twice Gateaux differentiable over the loop space $\W(\TT\tau;M)$. In this case,  $\iota(\gamma)$ is equal to  the Morse index of $\act\tau$ at $\gamma$, while $\nu(\gamma)$ is equal to the nullity of $\act\tau$ at $\gamma$, i.e.\ to the dimension of the null-space of the Gateaux-Hessian of $\act\tau$ at $\gamma$, see  \cite{b:Vi}, \cite{b:LoAn} or \cite{b:AbMaslov} for a proof.

\section{Discretizations for convex quadratic-growth Lagrangians}\label{s:discretizations}

Throughout this section, $\Lagr:\T\times\Tan M\to\R$ will be a 1-periodic  convex quadratic-growth Lagrangian, with associated  mean action $\act\tau$, $\tau\in\N$. In order to simplify the notation, we will work in period $\tau=1$, but everything goes through in every integer period.

The $\W$ functional setting for the action functional $\Act$ presents several drawbacks. First of all, the regularity that we can expect for $\Act$ is only $C^{1,1}$, at least if we assume to deal with a general convex quadratic-growth Lagrangian. This prevents the applicability of all those abstract results that require more smoothness, for instance the Morse lemma from critical point theory. Moreover, the $\W$ topology is sometimes  uncomfortable to work with. In fact, in several occasions it may be desirable to deal with a topology that is as strong as the $C^1$ topology, or at least as the $W^{1,\infty}$ topology. This would guarantee that the restriction of the action functional $\Act$ to a small neighborhood of a loop $\gamma$ only depends on the values that the Lagrangian assumes on a small neighborhood of the support of the lifted loop $(\gamma,\dot\gamma)$ in $\Tan M$. In the $W^{1,\infty}$ functional setting, the action functional $\Act$ is smooth, but unfortunately its sublevels do not satisfy any compactness condition (such as the Palais-Smale condition), which makes that functional setting inadequate for Morse theory. In order to overcome these difficulties, in this section  we develop a discretization technique that is a generalization to Lagrangian systems of the broken geodesics approximation of the path space (see \cite[section~16]{b:MiMo} or \cite[section~A.1]{b:Kl} for the Riemannian case, and \cite{b:Ra} for the Finsler case).

\subsection{Uniqueness of the action minimizers}
Given an interval $[t_0,t_1]\subset\R$, we say that an absolutely continuous curve $\gamma:[t_0,t_1]\to M$ is an \textbf{action minimizer} with respect to the Lagrangian $\Lagr$ when  every other absolutely continuous curve $\zeta:[t_0,t_1]\to M$ with the same endpoints of $\gamma$ satisfies
\[ \int_{t_0}^{t_1} \Lagr(t,\gamma(t),{\dot\gamma(t)})\,\diff t \leq  \int_{t_0}^{t_1} \Lagr(t,\zeta(t),{\dot\zeta(t)})\,\diff t .\]
It is well known that the action minimizers are  smooth solutions of the Euler-Lagrange system~\eqref{e:EuleroLagrange}. The existence of an action minimizer  joining two given points of $M$ is a known result that holds even for Tonelli Lagrangians, and it is essentially due to Tonelli (see  e.g.\ \cite{b:BGH} or \cite[page 98]{b:Fa}  for a modern treatment). A more ancient result, that goes back to Weierstrass, states that every sufficiently short action minimizer is unique, meaning that it is the only curve between its given endpoints that minimizes the action (see \cite[page 106]{b:Fa} or \cite[page 175]{b:Ma}). However, in this paper, we shall need the following stronger result that holds only for convex quadratic-growth Lagrangians.

\begin{prop}\label{p:fathi}
Let $\Lagr:\T\times\Tan M\to\R$ be a convex quadratic-growth Lagrangian. There exist  $\epsilon_0=\epsilon_0(\Lagr)>0$ and $\rho_0=\rho_0(\Lagr)>0$ such that, for each real interval $[t_0,t_1]\subset\R$ with $0<t_1-t_0\leq\epsilon_0$ and for all  $q_0,q_1\in M$ with $\dist(q_0,q_1)< \rho_0$, there is a unique action minimizer (with respect to $\Lagr$)  $\gamma_{q_0,q_1}:[t_0,t_1]\to M$ with  $\gamma_{q_0,q_1}(t_0)=q_0$ and $\gamma_{q_0,q_1}(t_1)=q_1$. 
\begin{proof}
For each absolutely continuous curve $\zeta:[t_0,t_1]\to M$, we denote by $\acp(\zeta)$ its action (with respect to $\Lagr$), i.e.  
\[ \acp(\zeta)=\int_{t_0}^{t_1} \Lagr(t,\zeta(t), \dot\zeta(t))\diff t \in\R\cup\gbra{+\infty}.\]
Up to summing a positive constant to $\Lagr$,  we can assume that there exist two positive constants $\underline\ell<\overline\ell$ such that 
\begin{align}\label{e:incastr}
\underline\ell\abs v_q^2 
\leq
\Lagr(t,q,v)
\leq
\overline\ell( \abs v_q^2+1 ),\s\s \forall q\in M, v\in\Tan_qM.
\end{align}
Consider two points $q_0,q_1\in M$ and two real numbers $t_0<t_1$. We put \[\rho:=\dist(q_0,q_1),\s \s \epsilon:=t_1-t_0.\] 
Since   $\W([t_0,t_1];M)$ is dense in the space of absolutely continuous maps from $[t_0,t_1]$ to $M$  and since the action minimizers are smooth,  a curve $\gamma_{q_0,q_1}$ as in the statement  is an action  minimizer  if and only if it is a global minimum of $\acp$ over the space 
\[\catW_{q_0,q_1}^{t_0,t_1} = \gbra{ \zeta\in\W([t_0,t_1];M)\,|\, \zeta(t_0)=q_0,\ \zeta(t_1)=q_1 }.\]
Therefore, all we have to do in order to prove the statement is to show that, for $\rho$ and $\epsilon$ sufficiently small, the functional $\acp|_{\catW_{q_0,q_1}^{t_0,t_1}}$ admits a unique global minimum.

Consider an arbitrary real constant $\mu>1$. By compactness, the manifold $M$ admits a finite atlas $\mathfrak U=\gbra{\phi_\alpha:U_\alpha\to\R^N\,|\,\alpha=1,...,u}$ such that, for all $\alpha\in\gbra{1,...,u}$, $q,q'\in U_\alpha$ and $v\in\Tan_qM$, we have
\begin{gather}
\label{e:mu1}  \mu^{-1} \abs{\phi_\alpha(q)-\phi_\alpha(q')} \leq  \dist(q,q')  \leq \mu \abs{\phi_\alpha(q)-\phi_\alpha(q')},\\
\label{e:mu2}  \mu^{-1} \abs{\diff\phi_\alpha(q)v} \leq  \abs v_q  \leq \mu \abs{ \diff\phi_\alpha(q)v },
\end{gather}
where we denote by $|{\cdot}|$ the standard norm in $\R^N$ and by $|{\cdot}|_q$ the Riemannian norm in $\Tan_qM$ as usual. Moreover, we can further assume that the image $\phi_\alpha(U_\alpha)$ of every chart is a convex subset of $\R^N$ (e.g.\ a ball). 
Let $\Leb(\mathfrak U)$ denote the Lebesgue number\footnote{We recall that, for every  open cover $\mathfrak U$ of a compact metric space, there exists a positive number $\Leb(\mathfrak U)>0$, the \textbf{Lebesgue number} of $\mathfrak U$, such that  every subset of the metric space of diameter less than $\Leb(\mathfrak U)$ is contained in some member of the cover $\mathfrak U$.} of the atlas $\mathfrak U$ and consider the two points $q_0,q_1\in M$ of the beginning with $\dist(q_0,q_1)=\rho$.  By   definition of Lebesgue number, the Riemannian closed ball 
\[ \overline{ B (q_0,\Leb(\mathfrak U)/2)}=\gbra{q\in M\,|\,\dist(q,q_0)\leq \Leb(\mathfrak U)/2} \]
is contained in a coordinate open set $U_\alpha$ for some $\alpha\in\gbra{1,...,u}$. Therefore, if we require  that $\rho \leq \Leb(\mathfrak U)/2$, the points $q_0$ and $q_1$  lie in the same open set $U_\alpha$.

Let $r:[t_0,t_1]\to U_\alpha$ be  the segment from $q_0$ to $q_1$ given by
\[
r(t)= \phi_\alpha^{-1} \cbra{
\frac{t_1- t}{\epsilon}
\, \phi_\alpha(q_0) 
+ 
\frac{t-t_0}{\epsilon}\, \phi_\alpha(q_1)
},\s\s\forall t\in[t_0,t_1]. \]
By \eqref{e:incastr}, \eqref{e:mu1} and \eqref{e:mu2} we obtain the following upper bound for the action of the curve $r$ 
\begin{align*}
\acp(r)  
& \leq   
\overline{\ell} \cbra{ \int_{t_0}^{t_1} |\dot r(t)|^2_{r(t)} \diff t + \epsilon }
\leq
\overline{\ell} \cbra{ \epsilon \max_{t\in[t_0,t_1]} \gbra{ |\dot r(t)|^2_{r(t)} } + \epsilon}
\\
& \leq 
\overline{\ell} \cbra{ \mu^2 \frac{ \abs{\phi_\alpha(q_1) - \phi_\alpha(q_0)}^2 }{  \epsilon} + \epsilon }
\leq  
\overline{\ell} \cbra{ \mu^4   \frac{ \dist(q_0,q_1)^2 }{  \epsilon} + \epsilon }
\\
&\leq
\overline{\ell} \mu^4  \cbra{ \frac{ \rho^2 }{  \epsilon} + \epsilon }=C \cbra{ \frac{ \rho^2 }{  \epsilon} + \epsilon },
\end{align*}
where the positive constant $C=\overline{\ell}   \mu^{4}$ does not depend  on   $q_0,q_1$ and $[t_0,t_1]$. This estimate, in turn, gives as an upper bound for the action of the minima, i.e.
\begin{align*}
\min_{\zeta\in\catW_{q_0,q_1}^{t_0,t_1}}\gbra{\acp(\zeta)}\leq C\cbra{ \frac{ \rho^2 }{  \epsilon} + \epsilon },
\end{align*}
therefore the   action sublevel
\begin{align}
\label{e:close_sublevel} 
 \catU_{q_0,q_1}^{t_0,t_1}=\catU_{q_0,q_1}^{t_0,t_1}(\rho,\epsilon)=
 \gbra{\zeta\in\catW_{q_0,q_1}^{t_0,t_1}\,|\, \acp(\zeta)\leq C \cbra{\frac{\rho ^2}\epsilon + \epsilon}} 
\end{align}
is not empty and it must contain a global  minimum $\gamma_{q_0,q_1}$ of the action (the existence of a minimum is a well known fact that holds even for Tonelli Lagrangians, see \cite[page 98]{b:Fa}). All we have to do in order to conclude is to show that, for $\rho $ and $\epsilon$ sufficiently small, the sublevel $\catU_{q_0,q_1}^{t_0,t_1}=\catU_{q_0,q_1}^{t_0,t_1}(\rho,\epsilon)$ cannot  contain other minima of the action.

By the first inequality in \eqref{e:incastr} we   have
\begin{align*}
  \int_{t_0}^{t_1}  |\dot\zeta(t)|^2_{\zeta(t)}  \diff t
  \leq {\underline{\ell}}^{-1} \acp(\zeta),\s\s\forall \zeta\in\catW_{q_0,q_1}^{t_0,t_1},
\end{align*}
and this, in turn, gives  the following bound for all $\zeta\in\catU_{q_0,q_1}^{t_0,t_1}$
\begin{align*}
  \max_{t\in[t_0,t_1]} \dist(\zeta(t_0),\zeta(t))^2 &
\leq
\cbra{\int_{t_0}^{t_1}  |\dot\zeta(t)|_{\zeta(t)}  \diff t}^2
\leq
\epsilon 
\int_{t_0}^{t_1}  |\dot\zeta(t)|^2_{\zeta(t)}  \diff t 
\\
&\leq
\epsilon {\underline{\ell}}^{-1}\acp(\zeta)
\leq
 C{\underline{\ell}}^{-1} (\rho ^2+\epsilon^2).
\end{align*}
Therefore all the curves $\zeta\in\catU_{q_0,q_1}(\rho ,\epsilon)$ have image inside the coordinate open set $U_\alpha\subseteq M$ provided $\rho $ and $\epsilon$ are sufficiently small, more precisely for 
\begin{align}\label{e:inside_chart}
\rho ^2 + \epsilon^2 \leq   \frac{\underline{\ell}}{4C} \Leb(\mathfrak U)^2  .
\end{align}
This allows us to  restrict our attention  to the open set $U_\alpha$. From now on we will briefly identify  $U_\alpha$ with $\phi_\alpha(U_\alpha)\subseteq\R^N$, so that
\[ q_0\equiv\phi_\alpha(q_0)\in\R^N,\s q_1\equiv\phi_\alpha(q_1)\in\R^N. \]
Without loss of generality we can also assume that $q_0\equiv \phi_\alpha(q_0)=\bm0\in\R^N$.
On the set $U_\alpha\equiv\phi_\alpha(U_\alpha)$ we will consider the standard flat norm $|\cdot|$ of $\R^N$, and the norms $\|\cdot\|_{L^1}$, $\|\cdot\|_{L^2}$ and $\|\cdot\|_{L^\infty}$ will be computed using this norm. We will also consider $\Lagr$ as a convex quadratic-growth Lagrangian of the form 
\[ \Lagr:\T\times \phi_\alpha(U_\alpha) \times\R^N\to\R\] by means of the identification
\[ \Lagr(t,q,v)\equiv \Lagr(t,\phi_\alpha^{-1}(q), \diff\phi_\alpha^{-1}(\phi_\alpha(q))v ). \]

Now, consider the following close convex subset of $\W([t_0,t_1];\R^N)$
\begin{multline}
\label{e:close_convex} 
    \catC_{q_0,q_1}^{t_0,t_1}=\catC_{q_0,q_1}^{t_0,t_1}(\rho,\epsilon)= 
    \biggl\{ \zeta\in \W([t_0,t_1];\R^N)\,\biggl|\, \\
    \zeta(t_0)=q_0=\bm0, \ \zeta(t_1)=q_1, \ 
    \| \dot\zeta \|_{L^2}^2\leq {\mu C}{\underline{\ell}}^{-1} \cbra{\frac{\rho^2}{\epsilon}+\epsilon }
    \biggr \}.  
\end{multline}
Since $\|\zeta\|_{L^\infty}^2\leq \epsilon \|\dot\zeta\|_{L^2}^2$, for $\rho$ and $\epsilon$ sufficiently small all the curves $\zeta\in\catC_{q_0,q_1}^{t_0,t_1}$ have support inside the open set $U_\alpha$. Moreover, by \eqref{e:incastr}, \eqref{e:mu2} and \eqref{e:close_sublevel},  we have 
\[  \| \dot\zeta \|_{L^2}^2  \leq 
 \mu \int_{t_0}^{t_1}  |\dot\zeta(t)|^2_{\zeta(t)}  \diff t
  \leq   {\mu } {\underline{\ell}}^{-1} \acp(\zeta)\leq {\mu C}{\underline{\ell}}^{-1}\cbra{\frac{\rho^2}\epsilon+\epsilon},\s\s\forall \zeta\in\catU_{q_0,q_1}^{t_0,t_1},
   \]
that implies $\catU_{q_0,q_1}^{t_0,t_1}\subseteq\catC_{q_0,q_1}^{t_0,t_1}$. Now, since we know that a minimum $\gamma_{q_0,q_1}$ of the action exists and all the minima lie in the closed convex subset $\catC_{q_0,q_1}^{t_0,t_1}\subseteq\W([t_0,t_1];\R^N)$, in order to conclude that $\gamma_{q_0,q_1}$ is the unique  minimum  we only need to show that the Hessian of the action is positive definite on $\catC_{q_0,q_1}^{t_0,t_1}$ provided $\rho$ and $\epsilon$ are sufficiently small, i.e.\ we need to show that there exist $\rho _0>0$ and $\epsilon_0>0$ such that, for all $\rho \in(0,\rho _0)$ and $\epsilon\in(0,\epsilon_0]$, we have
\begin{equation}\label{e:pos_def_C}
\begin{split}
&\Hess\acp(\zeta)[\sigma,\sigma]>0,\\
&\s\s\s\forall  \zeta\in \catC_{q_0,q_1}^{t_0,t_1}=\catC_{q_0,q_1}^{t_0,t_1}(\rho,\epsilon),\ \sigma\in \W_0([t_0,t_1];\R^N).
\end{split}
\end{equation}
Notice that the above Hessian  is well defined, since $\acp$ is $C^1$ and twice Gateaux differentiable. In \eqref{e:pos_def_C}, we have denoted by $\W_0([t_0,t_1];\R^N)$ the tangent space of $\catC_{q_0,q_1}^{t_0,t_1}$ at $\zeta$, i.e.
\[\W_0([t_0,t_1];\R^N)=\gbra{\sigma\in\W([t_0,t_1];\R^N)\,|\,\sigma(t_0)=\sigma(t_1)=\bm0}.\]
Consider arbitrary  $\zeta\in\catC_{q_0,q_1}^{t_0,t_1}$ and   $\sigma\in\W_0([t_0,t_1];\R^N)$. Then, we have
\begin{align*}
&\Hess\acp(\zeta)[\sigma,\sigma]\\
&\s=
\int_{t_0}^{t_1} 
\cbra{
\langle{ \partial^2_{vv}\Lagr(t,\zeta,\dot\zeta)\dot\sigma,\dot\sigma }\rangle 
+2
\langle{ \partial^2_{vq}\Lagr(t,\zeta,\dot\zeta)\sigma,\dot\sigma }\rangle 
+
\langle{ \partial^2_{qq}\Lagr(t,\zeta,\dot\zeta)\sigma,\sigma }\rangle 
}\diff t
\\
&\s\geq 
\int_{t_0}^{t_1}
\ell_0\abs{\dot\sigma}^2 \diff t
-
\underbrace{
\int_{t_0}^{t_1}
2 \ell_1(1+\mu |\dot\zeta|)\abs{\sigma}\abs{\dot\sigma} \diff t
}_{\mbox{$=:I_1$}}
-
\underbrace{
\int_{t_0}^{t_1}
 \ell_1 (1+\mu^2 |{\dot\zeta}|^2)\abs{\sigma}^2 \diff t
}_{\mbox{$=:I_2$}} 
 ,
\end{align*}
where $\ell_0$ and $\ell_1$ are the positive constants that appear in \textbf{(Q1)} and \textbf{(Q2)} with respect to the atlas $\mathfrak U$. Now, the quantities $I_1$ and $I_2$ can be estimated from above as follows
\begin{align*}
I_1 & 
\leq  2\ell_1\mu \|\sigma\|_{L^\infty} \cbra{ \|\dot\sigma\|_{L^1} + \norm{ |\dot\zeta|\cdot|\dot\sigma| }_{L^1} } \\
& \leq 2\ell_1\mu \sqrt\epsilon  \|\dot\sigma\|_{L^2} \cbra{\sqrt\epsilon \|\dot\sigma\|_{L^2} +   \|\dot\zeta\|_{L^2} \|\dot\sigma\|_{L^2} } \\ 
& =   2\ell_1\mu   \|\dot\sigma\|_{L^2}^2 \cbra{\epsilon + \sqrt\epsilon \|\dot\zeta\|_{L^2}  }, \\
I_2 & \leq  \ell_1 \mu^2 \cbra{ \|\sigma\|_{L^2}^2 + \|\sigma\|_{L^\infty}^2 \|\dot\zeta\|_{L^2}^2 }
\leq \ell_1 \mu^2 \|\dot\sigma\|_{L^2}^2 \cbra{ \epsilon^2 + \epsilon \|\dot\zeta\|_{L^2}^2 },
\end{align*}
and, since by \eqref{e:close_convex}   we have
\[ \| \dot\zeta \|_{L^2}^2\leq {\mu C}{\underline{\ell}}^{-1} \cbra{\frac{\rho^2}{\epsilon}+\epsilon }, \]
we conclude
\begin{align*}
&\Hess\acp(\zeta)[\sigma,\sigma] 
\\ 
&\s\geq 
\ell_0 \|\dot\sigma\|_{L^2}^2 - I_1 - I_2 
  \\
&\s\geq 
\|\dot\sigma\|^2_{L^2}
\underbrace{\cbra{
\ell_0  - 
2\ell_1\mu \cbra{ \sqrt{{\mu C}{\underline{\ell}}^{-1}} + 1 }(\rho+\epsilon) -
\ell_1\mu^2 \cbra{  {{\mu C}{\underline{\ell}}^{-1}} + 1 }(\rho^2+\epsilon^2)
}}_{\mbox{$=: F(\rho ,\epsilon)$}}. 
\end{align*}
Notice that the quantity $F(\rho ,\epsilon)$ is independent of the specific choice of the points $q_0,q_1$ and of the interval $[t_0,t_1]$, but depends only on   $\rho=\dist(q_0,q_1) $  and  $\epsilon=t_1-t_0$. Moreover, there exist $\rho _0>0$ and $\epsilon_0>0$   small enough so that  for all $\rho \in(0,\rho _0)$ and $\epsilon\in(0,\epsilon_0]$ the quantity $F(\rho ,\epsilon)$ is positive.  This proves \eqref{e:pos_def_C}.
\end{proof}
\end{prop}

Now, we want to remark that the short action minimizers $\gamma_{q_0,q_1}$, given by  proposition~\ref{p:fathi}, depend smoothly on their endpoints $q_0$ and $q_1$. If $\rho_0$ is the constant given by proposition~\ref{p:fathi}, we denote by $\Delta_{\rho_0}$ the open neighborhood of the diagonal submanifold of $M\times M$ given by 
\[\Delta_{\rho_0}=\gbra{(q_0,q_1)\in M\times M\,|\,\dist(q_0,q_1)<\rho_0}.\]

\begin{prop}\label{p:fathi2}
With the notation of proposition~\ref{p:fathi}, for each real interval $[t_0,t_1]\subset\R$ with $0<t_1-t_0\leq\epsilon_0$ the assignment 
\begin{align}\label{e:minimizer_map}
(q_0,q_1)\mapsto\gamma_{q_0,q_1}:[t_0,t_1]\to M\end{align}
defines a smooth map $\Delta_{\rho_0}\to C^\infty([t_0,t_1];M)$.
\begin{proof}
Since the action minimizers are smooth, \eqref{e:minimizer_map} defines a map \[\Delta_{\rho_0}\to C^\infty([t_0,t_1];M),\] and we just need to show that the dependence of $\gamma_{q_0,q_1}$ from $(q_0,q_1)$ is smooth. If $t_1-t_0\in(0,\epsilon_0]$  and $(q_0,q_1)\in\Delta_{\rho_0}$, in the proof of proposition~\ref{p:fathi} we have already shown that the minimizer $\gamma_{q_0,q_1}:[t_0,t_1]\to M$ has image contained  in a coordinate neighborhood $U_\alpha\subseteq M$ that we can identify with an open set of $\R^N$. The curve $\gamma_{q_0,q_1}$ is a smooth solution of the Euler-Lagrange system of $\Lagr$, therefore
\[\Phi_{\Lagr}^t\circ(\Phi_{\Lagr}^{t_0})^{-1}(q_0,v_0)= (\gamma_{q_0,q_1}(t),\dot\gamma_{q_0,q_1}(t)),\s\s\forall t\in[t_0,t_1],\]
where $v_0=\dot\gamma_{q_0,q_1}(t_0)$ and $\Phi_{\Lagr}^t$ is the Euler-Lagrange flow associated to $\Lagr$ (see section~\ref{s:Lagr_settings}). 
We define
\[  Q^t:=\pi\circ\Phi_{\Lagr}^t\circ(\Phi_{\Lagr}^{t_0})^{-1}: U_\alpha' \times \R^N\to U_\alpha,\s\s\forall t\in[t_0,t_1], \]
where $U_\alpha'\subset U_\alpha$ is a small neighborhood of $q_0$, and   $\pi:\R^N\times\R^N\to\R^N$ is the projection onto  the first $N$ components, i.e. $\pi(q,v)=q$ for all $(q,v)\in\R^N\times\R^N$. We claim that 
\begin{align}\label{e:non_deg_gen_funct}
\diff Q^{t_1}(q_0,v_0)(\gbra{\bm0}\times \R^N)=\R^N.
\end{align}
In fact, assume by contradiction that~\eqref{e:non_deg_gen_funct} does not hold. Then, there exists a nonzero vector $v\in\R^N$ such that 
\[ \left.\der{}{s}\right|_{s=0} Q^{t_1}(q_0,v_0+s v)=\bm0.  \]
If we define the curve $\sigma:[t_0,t_1]\to\R^N$ by 
\[ \sigma(t):= \left.\der{}{s}\right|_{s=0} Q^{t}(q_0,v_0+s v), \]
then  $\sigma(t_0)=\sigma(t_1)=\bm0$, and $\sigma$ is a solution of the linearized Euler-Lagrange system
\begin{multline*} 
\der{}{t} \cbra{ \partial^2_{vv} \Lagr(t,\gamma_{q_0,q_1},\dot\gamma_{q_0,q_1})\dot\sigma + \partial^2_{vq} \Lagr(t,\gamma_{q_0,q_1},\dot\gamma_{q_0,q_1})\sigma  } \\
-  \partial^2_{qv} \Lagr(t,\gamma_{q_0,q_1},\dot\gamma_{q_0,q_1})\dot\sigma 
- \partial^2_{qq} \Lagr(t,\gamma_{q_0,q_1},\dot\gamma_{q_0,q_1})\sigma=0.
\end{multline*}
This implies that $\Hess\acp(\gamma_{q_0,q_1})[\sigma,\sigma]=0$, which contradicts the  positive definitiveness of $\Hess\acp(\gamma_{q_0,q_1})$ (see \eqref{e:pos_def_C} in the proof of proposition~\ref{p:fathi}). Therefore, \eqref{e:non_deg_gen_funct} must hold.

By the implicit function theorem we obtain a  neighborhood $U_{q_0,q_1}\subset\R^N\times\R^N$ of $(q_0,q_1)$, a neighborhood $U_{v_0}\subset\R^N$ of $v_0$ and a smooth map $V_0:U_{q_0,q_1}\to U_{v_0}$  such that, for each $(q_0',q_1',v_0')\in U_{q_0,q_1}\times U_{v_0}$, we have
$Q^{t_1}(q_0',v_0')=q_1'$ if and only if $v_0'=V_0(q_0',q_1')$. Then, we can define a smooth map from $U_{q_0,q_1}$ to $C^\infty([t_0,t_1];U_\alpha)$ given by 
\begin{align}\label{e:min_map2}
(q_0',q_1')\mapsto \zeta_{q_0',q_1'},
\end{align}
 where for each $t\in[t_0,t_1]$ we have
 \[\zeta_{q_0',q_1'}(t)=Q^{t}(q_0',V_0(q_0',q_1')).\]
  In order to conclude we only have to show that the  map in \eqref{e:min_map2} coincides with the one in \eqref{e:minimizer_map} on  $U_{q_0,q_1}$ provided this latter neighborhood is sufficiently small, i.e.\ we have to show that $\zeta_{q_0',q_1'}$ is the unique action  minimizer joining $q_0'$ and $q_1'$, for each $(q_0',q_1')$ in a sufficiently small neighborhood $U_{q_0,q_1}$ of $(q_0,q_1)$.
This is easily seen as follows. By construction, the curves $\zeta_{q_0',q_1'}$ are critical points of the action $\acp$ over the space $\catW_{q_0',q_1'}^{t_0,t_1}$, being solutions of the Euler-Lagrange system of $\Lagr$. By the arguments in the proof of proposition~\ref{p:fathi}, each of these curves $\zeta_{q_0',q_1'}$ is the unique action minimizer joining its endpoints if and only if it lies in the convex set $\catC_{q_0',q_1'}^{t_0,t_1}$ defined in \eqref{e:close_convex}.
We already know that $\zeta_{q_0,q_1}=\gamma_{q_0,q_1} \in\catC_{q_0,q_1}^{t_0,t_1}$.
Since the map in \eqref{e:min_map2} is smooth, for $(q_0',q_1')$ close to $(q_0,q_1)$ we obtain that the curve $\zeta_{q_0',q_1'}$ is $C^1$-close to $\zeta_{q_0,q_1}=\gamma_{q_0,q_1}$, and therefore  $\zeta_{q_0',q_1'}\in\catC_{q_0',q_1'}^{t_0,t_1}$.
\end{proof}
\end{prop}

\subsection{The discrete action functional}\label{s:discr_act_funct}

Let $\epsilon_0=\epsilon_0(\Lagr)$ and $\rho_0=\rho_0(\Lagr)$ be the positive constants given by proposition~\ref{p:fathi},  and let $k\in\N$ be such that $1/k \leq \epsilon_0$. We define the  \textbf{$k$-broken Euler-Lagrange loop space} as the subspace $\Lambda_k=\Lambda_{k,\Lagr}\subset\W(\T;M)$ consisting of those loops $\gamma:\T\to M$ such that $\dist(\gamma(\frac ik),\gamma(\frac {i+1}k))<\rho_0$ and $\gamma|_{[i/k,(i+1)/k]}$ is an action  minimizer for each $i\in\gbra{0,...,k-1}$. Notice that,  by propositions~\ref{p:fathi} and~\ref{p:fathi2},  the correspondence 
\begin{align*} 
\textstyle \gamma\mapsto \cbra{ \gamma(0),\gamma(\frac 1k),..., \gamma(\frac{k-1}k)} 
\end{align*}
defines a diffeomorphism between $\Lambda_k$ and an open subset of the $k$-fold product $M\times...\times M$. Thus, $\Lambda_k$ is a finite dimensional submanifold of the $W(\T;M)$, which implies that the  $W^{1,2}$ and $W^{1,\infty}$ topologies coincide on it.

We define the \textbf{discrete action functional} $\Act_k$ as the restriction of $\Act$ to $\Lambda_k$, i.e.
\[ \Act_k:=\Act|_{\Lambda_k}.\]
Notice that $\Act_k$ is smooth, since  $\Act$ is smooth on $W^{1,\infty}(\T;M)$ and $\Lambda_k\subset W^{1,\infty}(\T;M)$. Moreover, the next proposition implies that, for each action value $c\in\R$, there is a sufficiently big discretization pass $k\in\N$ such that $\Act_k$ satisfies the Palais-Smale condition in the $c$-sublevel.

\begin{prop}
For each $c\in\R$ there exists $\bar k=\bar k(c)\in\N$ such that, for each $k\geq\bar k$, the closed sublevel  $\Act_k^{-1}(-\infty,c]$ is compact.
\begin{proof}
Consider the compact subset of $\Lambda_k$ defined by
\begin{align*}
C_k:= 
\gbra{
\gamma\in\Lambda_k\,|\,
\dist(\gamma(\sfrac ik),\gamma(\sfrac{i+1}k))\leq\rho_0/2,\  \forall i\in\gbra{1,...,k-1}
}.
\end{align*}
In order to prove the statement, we just need to show that 
\begin{align*}
\lim_{k\to\infty} \min\gbra{\Act_k(\gamma)\,|\,\gamma\in \partial C_k}=+\infty.
\end{align*}
Up to summing a positive constant to $\Lagr$, we can assume that there exists a  constant  $\underline\ell>0$ such that $\Lagr(t,q,v)\geq\underline\ell\, |v|_q^2$ for every  $(t,q,v)\in\T\times\Tan M$. Then, consider an arbitrary $\gamma\in\partial C_k$. For some $i\in\gbra{0,...,k-1}$ we have that \[\dist(\gamma(\sfrac ik),\gamma(\sfrac {i+1}k))=\rho_0/2,\]  and therefore we obtain the desired estimate
\begin{align*}
\Act_k(\gamma)
&\geq
\int_{i/k}^{(i+1)/k} \Lagr(t,\gamma(t),\dot\gamma(t))\,\diff t
\geq
\int_{i/k}^{(i+1)/k} \underline\ell\,|\dot\gamma(t)|_{\gamma(t)}^2\,\diff t\\
&\geq
k\,\underline\ell
\cbra{
\int_{i/k}^{(i+1)/k} \,|\dot\gamma(t)|_{\gamma(t)}\,\diff t
}^2
\geq
k\,\underline\ell\,\dist(\gamma(\sfrac ik),\gamma(\sfrac {i+1}k))^2\\
&\geq
k\,\underline\ell\,(\rho_0/2)^2.
\qedhere
\end{align*}
\end{proof}
\end{prop}

Each  critical point $\gamma$ of the action functional $\Act$ belongs to the $k$-broken  Euler-Lagrange loop space $\Lambda_k$, up to choosing a sufficiently big $k$, and in particular it is a critical point of the discrete action $\Act_k$. It is easy to verify that the converse is also true, namely that the critical points of the discrete action $\Act_k$ are smooth solutions of the Euler-Lagrange system of $\Lagr$. The next statement discusses the invariance of the Morse index and nullity under discretization.

\begin{prop}\label{p:SameMorseNullPair_total}
Consider a contractible  $\gamma:\T\to M$ that is a smooth solution of the Euler-Lagrange system of $\Lagr$. Then,  for each sufficiently big $k\in\N$, the Morse index and nullity pair of $\Act$ and $\Act_k$ at $\gamma$ are the same.
\end{prop}

We will split the proof of this proposition in several lemmas, which will take the remaining of this subsection. First of all, since the statement is of a local nature, let us adopt suitable local coordinates in the loop space. Being $\gamma$ a smooth contractible loop, we can consider the chart of $\W(\T;M)$ given by
\[\exp_\gamma^{-1}: \catU_\gamma \to \W(\gamma^*\Tan M)\simeq W^{1,2}(\T;\R^N),\]   
see section~\ref{s:free_loop_space} for the notation.  Let $U$ be a small neighborhood of the origin  in  $\R^N$ and let $\pi_{\gamma}:\T \times U \times \R^N  \hookrightarrow \T\times \Tan M$ be the embedding defined by
\begin{align*}
\pi_{\gamma}(t,q,v)=\cbra{ t, \exp_{{\gamma}(t)}(q), \diff ({\exp_{{\gamma}(t)}})(q)v +
 \der{}{t} \exp_{{\gamma}(t)}(q) },\s\s
\forall(t,q,v)\in\T\times U\times\R^N.
\end{align*}
The pulled-back Lagrangian $\Lagr\circ \pi_{\gamma} : \T \times U \times \R^N \to \R$ is again convex quadratic-growth, since conditions \textbf{(Q1)} and \textbf{(Q2)} are invariant with respect to coordinate  transformations of the form $\pi_{\gamma}$ (up to changing the constants $\ell_0$ and $\ell_1$ in the definition). Moreover, the  pulled-back functional $\Act\circ \exp_{\gamma}$ is the Lagrangian action functional associated to the convex quadratic-growth  Lagrangian $\Lagr\circ \pi_{\gamma}$, i.e.  
\begin{align*}
\Act\circ \exp_{\gamma}( \xi ) =  \int_0^1 \Lagr\circ\pi_{\gamma}(t,\xi(t),\dot\xi(t))\, \diff t.
\end{align*}
From now on, we will simply write $\Act$ and $\Lagr$ for  $\Act\circ \exp_{\gamma}$ and $\Lagr\circ\pi_{\gamma}$ respectively, so that $\gamma$ will  be identified with  the point $\bm0$ in the Hilbert space $W^{1,2}(\T;\R^N)$.  Moreover, for each $k\in\N$ such that $\gamma$ belongs to $\Lambda_k$, we identify (an open neighborhood of $\gamma$ in) the $k$-broken Euler-Lagrange loop space $\Lambda_k$ with a finite dimensional submanifold of $\W(\T;\R^N)$ containing $\bm0$.

We introduce the   quadratic Lagrangian $\lagr:\T\times\R^N\times\R^N\to\R$ given by
\begin{multline}\label{e:linLagr}
\s\s\lagr(t,q,v)=\frac12 \abra{a(t)v,v} + \abra{b(t)q,v} + \frac12 \abra{c(t)q, q},\\
\forall (t,q,v)\in\T\times\R^N\times\R^N,
\end{multline}
where, for each $t\in\T$,  $a(t)$, $b(t)$ and $c(t)$ are the $N\times N$ matrices defined by
\[ a_{ij}(t):=\frac{\partial^2 \Lagr}{\partial v_i\,\partial v_j} (t,0,0),\s 
b_{ij}(t):=\frac{\partial^2 \Lagr}{\partial v_i\,\partial q_j} (t,0,0),\s
c_{ij}(t):=\frac{\partial^2 \Lagr}{\partial q_i\,\partial q_j} (t,0,0). \]
A straightforward computation shows that the Euler-Lagrange system associated to $\lagr$ is given by the following linear system of ordinary differential equations for curves $\sigma$ in $\R^N$
\begin{align}\label{e:EulLagr_lin}
a\,\ddot\sigma + (b + \dot a - b^T )\, \dot\sigma + (\dot b - c)\, \sigma =0. 
\end{align}
This is precisely the linearization of the Euler-Lagrange system of $\Lagr$ along the periodic solution $\gamma\equiv \bm0$. The 1-periodic solutions $\sigma:\T\to\R^N$ of~\eqref{e:EulLagr_lin} are precisely the critical points of the action functional $\A:W^{1,2}(\T;\R^N)\to\R$ associated to $\lagr$,  given as usual by
\[ \A(\xi)= \int_0^1 \lagr(t,\xi(t),\dot\xi(t))\,\diff t,\s\s\forall \xi\in\W(\T;\R^N). \]

The following lemma characterizes the elements of the tangent space of the $k$-broken Euler-Lagrange loop space $\Lambda_k$ at $\gamma\equiv\bm0$.
\begin{lem}\label{l:TLambda}
The tangent space $\Tan_{\bm0}\Lambda_k$ is the space of continuous and piecewise smooth loops $\sigma:\T\to\R^N$ such that, for each $h\in\gbra{0,...,k-1}$, the restriction $\sigma|_{[h/k,(h+1)/k]}$ is a  solution of the Euler-Lagrange system~\eqref{e:EulLagr_lin}.
\begin{proof}
By definition of tangent space, every $\sigma\in\Tan_{\bm0}\Lambda_k$ is a continuous loop $\sigma:\T\to\R^N$  given by 
\begin{align}\label{e:der_of_var} 
\sigma(t)=\left.\pder{}{s}\right|_{s=0}\Sigma(s,t), 
\end{align}
for some continuous $\Sigma:(-\epsilon,\epsilon)\times\T\to\R^N$ such that, for all $h\in\gbra{0,...,k-1}$, the restriction $\Sigma|_{(-\epsilon,\epsilon)\times [h/k,(h+1)/k]}$ is smooth, $\Sigma(s,\cdot)\in\Lambda_k$ for all $s\in(-\epsilon,\epsilon)$ and $\Sigma(0,\cdot)\equiv\bm 0$. Namely $\Sigma$ is a   piecewise smooth variation of the constant loop $\bm0$ such that  the loops $\Sigma_s=\Sigma(s,\cdot)$  satisfy the Euler-Lagrange system of $\Lagr$ in the intervals $[h/k,(h+1)/k]$ for all $h\in\gbra{0,...,k-1}$, i.e.
\begin{align*}
    \partial^2_{vv} \Lagr  (t,\Sigma_s,\dot\Sigma_s)\, \ddot\Sigma_s  + 
{\partial^2_{vq} \Lagr}  (t,\Sigma_s,\dot\Sigma_s)\, \dot\Sigma_s  +
{\partial^2_{vt} \Lagr}  (t,\Sigma_s,\dot\Sigma_s) -
{\partial_q \Lagr} (t,\Sigma_s,\dot\Sigma_s)=0
\end{align*}
By differentiating the above equation with respect to $s$ in $s=0$, we obtain the Euler-Lagrange system~\eqref{e:EulLagr_lin} for the loop $\sigma$ (as before, satisfied on the  intervals $[h/k,(h+1)/k]$ for all $h\in\gbra{0,...,k-1}$).  Vice versa,   a continuous loop $\sigma:\T\to\R^N$ whose restrictions $\sigma|_{[h/k,(h+1)/k]}$  satisfy \eqref{e:EulLagr_lin} is of the form \eqref{e:der_of_var} for some $\Sigma$ as above, and therefore it is an element of $\Tan_{\bm0}\Lambda_k$.
\end{proof}
\end{lem}

The null-space of the Hessian of $\Act$ at $\bm0$ can be characterized as follows.

\begin{lem}\label{l:NullHess}
The null-space of $\Hess\Act(\bm0)$ consists of those smooth loops $\sigma:\T\to\R^N$ that are  solutions of the Euler-Lagrange system~\eqref{e:EulLagr_lin}.
\begin{proof}
For every $\sigma,\xi\in\W(\T;\R^N)$ we have
\begin{align*} 
\Hess\Act(\bm0)[\sigma,\xi]=&
  \int_0^1 
\cbra{ \langle a\,\dot\sigma,\dot\xi \rangle + 
\langle b\,\sigma,\dot\xi\rangle + 
\langle b^T\,\dot\sigma,\xi\rangle + 
\langle c\,\sigma,\xi\rangle }  
\diff t
= \diff \A(\sigma)\xi.
\end{align*}
Therefore $\sigma$ is in the null-space of $\Hess\Act(\bm0)$ if and only of it is a critical point of $\A$, that is if and only if it is a (smooth) solution of the Euler-Lagrange system~\eqref{e:EulLagr_lin}.
\end{proof}
\end{lem}

As a consequence of   lemmas~\ref{l:TLambda} and \ref{l:NullHess}, the null-space of $\Hess\Act(\bm0)$ is contained in $\Tan_{\bm0}\Lambda_k$, and therefore it is  contained in the null-space of the Hessian of the discrete action $\Hess\Act_k(\bm0)$. This inclusion is actually an equality, as shown by the following.
\begin{lem}\label{l:same_Null}
$\Hess\Act(\bm0)$ and $\Hess\Act_k(\bm0)$ have the same null-space, and in particular $\Act$ and $\Act_k$ have the same nullity at $\bm0$.
\begin{proof}
We only need to show that any curve $\sigma\in\Tan_{\bm0}\Lambda_k$ that is not everywhere smooth cannot be   in the null-space of $\Hess\Act_k(\bm0)$. In fact, since $\sigma$ is always smooth outside the points $\sfrac hk$ (for $h\in\gbra{0,...,k-1}$), for each $\xi\in\Tan_{\bm0}\Lambda_k$ we have
\begin{align}
\nonumber \Hess\Act_k(\bm0)[\sigma,\xi]=&
 \sum_{h=0}^{k-1}
  \int_{h/k}^{(h+1)/k} 
\cbra{ \langle a\,\dot\sigma,\dot\xi \rangle + 
\langle b\,\sigma,\dot\xi\rangle + 
\langle b^T\,\dot\sigma,\xi\rangle + 
\langle c\,\sigma,\xi\rangle }  
\diff t\\
\nonumber
=&
 \sum_{h=0}^{k-1}
  \int_{h/k}^{(h+1)/k} 
\langle 
\underbrace
{-a\,\ddot\sigma - b\,\dot\sigma - \dot a\,\dot\sigma - \dot b\, \sigma+ b^T\,\dot\sigma + c\,\sigma}_{=0}
, 
\xi\rangle\, \diff t
\\
\nonumber
&+\sum_{h=0}^{k-1} 
 \langle a\,\dot\sigma+b\,\sigma,\xi\rangle
\Bigr|_{(h/k)^+}^{((h+1)/k)^-} 
\\
\label{e:hess_restr}
=&
\sum_{h=0}^{k-1} 
 \langle {
 a(\sfrac hk) [ \dot\sigma(\sfrac hk\sm-) 
 -
  \dot\sigma(\sfrac hk\sm+) ]
  ,\xi(\sfrac hk)  
 }\rangle .
\end{align}
By assumption, we have that  $\dot\sigma(\sfrac {h}k \sm+)\neq\dot\sigma(\sfrac {h}k\sm-)$  for some $h\in\gbra{0,...,k-1}$, and therefore 
\[a(\sfrac {h}k)[\dot\sigma(\sfrac {h}k\sm+)-\dot\sigma(\sfrac {h}k\sm-)]\neq\bm0.\]
Here, we have used the fact that the matrix $a(\sfrac {h}k)$ is invertible, since the Lagrangian $\Lagr$ satisfies \textbf{(Q1)} (see section~\ref{s:Lagr_settings}). Now, consider $\xi\in\Tan_{\bm0}\Lambda_k$ given by 
\[  
\xi({\sfrac lk})
=
\left\{
  \begin{array}{lccl}
    a(\sfrac {h}k)[\dot\sigma(\sfrac {h}k\sm+)-\dot\sigma(\sfrac {h}k\sm-)], & & & l=h, \\ 
    \bm0, & & & l\in\gbra{0,...,k-1},\ l\neq h. \\ 
  \end{array}
\right.
\]
By~\eqref{e:hess_restr}, we have
\[
\Hess\Act_k(\bm0)[\sigma,\xi]= 
\abs{ 
a(\sfrac {h}k)[\dot\sigma(\sfrac {h}k\sm+)-\dot\sigma(\sfrac {h}k\sm-)]
 }^2 \neq0,
\]
and therefore we conclude that $\sigma$ is not in the null-space of $\Hess\Act_k(\bm0)$.
\end{proof}
\end{lem}

In order to conclude the proof of proposition~\ref{p:SameMorseNullPair_total} we only need to prove the invariance of the Morse index under discretization.

\begin{lem}\label{l:same_Morse}
For all $k\in\N$ sufficiently big, the functionals $\Act$ and $\Act_k$ have the same Morse index at $\bm0$.
\begin{proof}
For every  $\sigma\in\W(\T;\R^N)$ we have 
\begin{equation}\label{e:HessXXact}
\begin{split}
\Hess\Act(\bm0)[\sigma,\sigma]=&
  \int_0^1 
\cbra{ \langle a\,\dot\sigma,\dot\sigma \rangle + 
\langle b\,\sigma,\dot\sigma\rangle + 
\langle b^T\,\dot\sigma,\sigma\rangle + 
\langle c\,\sigma,\sigma\rangle }  
\diff t
= \\
=&
2\int_0^1 
\lagr(t,\sigma(t),\dot\sigma(t))\,  
\diff t = 2\,\A(\sigma).
\end{split}
\end{equation}
Let $\iota(\bm0)$ be the Morse index of $\Act$ at $\bm0$. By definition, there exists a $\iota(\bm0)$-dimensional vector subspace $V\subseteq\W(\T;\R^N)$ over which $\Hess\Act(\bm0)$ is negative definite, i.e.\
\begin{align}\label{e:Vneg}
\Hess\Act(\bm0)[\sigma,\sigma] < 0,\s\forall \sigma\in V\setminus\gbra{\bm0}. 
\end{align}
By density, for each $k\in\N$ sufficiently big we can choose $V$ to be composed of $k$-piecewise affine curves. Namely  we can choose $V$ such that, for each $\sigma\in V$, we have
\begin{align}
\label{e:pw_aff}
\sigma({\sfrac{h+t}k})=(1-t)\,\sigma({\sfrac{h}k})+t\,\sigma({\sfrac{h+1}k}),
&& \forall  t\in[0,1],\ h\in\gbra{0,...,k-1}.
\end{align}
Now, let us define a linear map $K:V\to\Tan_{\bm0}\Lambda_k$ as $K(\sigma)=\tilde\sigma$, where $\tilde\sigma$ is the unique element in $\Tan_{\bm0}\Lambda_k$ such that $\sigma(\frac hk)=\tilde\sigma(\frac hk)$ for each $h\in\gbra{0,...,k-1}$. 
Notice that $K$ is injective. In fact, if $K(\sigma)=\bm0$, we have $\sigma(\frac hk)=\bm0$ for each $h\in\gbra{0,...,k-1}$ and, by~\eqref{e:pw_aff}, we conclude $\sigma=\bm0$. Hence $\tilde V=K(V)$ is a $\iota(\bm0)$-dimensional vector subspace of $\Tan_{\bm0}\Lambda^k$.

In order to conclude we just have to show that $\Hess\Act_k(\bm0)$ is negative definite over the vector space $\tilde V$. To this aim, consider an arbitrary $\tilde\sigma\in\tilde V\setminus\gbra{\bm0}$ and put $\sigma=K^{-1}(\tilde\sigma)\in V\setminus\gbra{\bm0}$. For each $h\in\gbra{0,...,k-1}$ the curve $\tilde\sigma|_{[h/k,(h+1)/k]}$ is an action minimizer with respect to the Lagrangian $\lagr$, and therefore $\A(\tilde\sigma)\leq\A(\sigma)$. By~\eqref{e:HessXXact} and~\eqref{e:Vneg} we conclude
\[\Hess\Act_k(\bm0)[\tilde\sigma,\tilde\sigma]= 
2\,\A( \tilde\sigma) \leq 
2\,\A( \sigma)  =
\Hess\Act(\bm0)[\sigma,\sigma] <0.
\qedhere
\]
\end{proof}
\end{lem}

\subsection{Homotopic approximation  of the action sublevels}\label{s:hom_sub} 
We want to show that the sublevels of the action $\Act$ deformation retract onto the corresponding sublevels of the discrete action $\Act_k$, for all the sufficiently big $k\in\N$. To begin with, we need the following.

\begin{prop}\label{p:estimate_sub}
For each $\rho>0$ and  $c\in\R$ there exists  $\bar\epsilon=\bar\epsilon(\Lagr,\rho,c)>0$ such that, for each $\zeta\in\W(\T;M)$ with $\Act(\zeta)<c$ and for each interval $[t_0,t_1]\subset\R$ with $0<t_1-t_0\leq\bar\epsilon$, we have
$\dist\cbra{\zeta(t_0), \zeta(t_1)}<\rho$.
\begin{proof}
Up to summing a positive constant to the convex quadratic-growth Lagrangian $\Lagr$,  we can always assume  that  there exists   $\underline\ell>0$ such that 
\[\Lagr(t,q,v) \geq \underline\ell\,\abs v_q^2,\s\s \forall (t,q,v)\in\T\times\Tan M.\]
Then, let us consider an arbitrary  $\zeta\in\W(\T;M)$ such that $\Act(\zeta)<c$. For each interval $[t_0,t_1]\subset\R$ with $0<t_1-t_0\leq 1$ we have  
\begin{align*}
\dist\cbra{\zeta(t_0), \zeta(t_1)}^2
& \leq 
\cbra{
		\int_{t_0}^{t_1}  |{\dot\zeta(t)}|_{\zeta(t)}
		\diff t  
}^2
\leq
(t_1-t_0) 
\int_{t_0}^{t_1}   
|\dot\zeta(t)|_{\zeta(t)}^2 \,\diff t \\
& \leq  
(t_1-t_0) 
\int_{t_0}^{t_1}  
\underline\ell^{-1}  
\Lagr(t,\zeta(t),\dot\zeta(t)) \,\diff t
\leq
  (t_1-t_0) \underline\ell^{-1} \Act(\zeta)  \\   
& < 
  (t_1-t_0) \underline\ell^{-1}c.    
\end{align*}
Hence,  for $\bar\epsilon=\bar\epsilon(\Lagr,\rho,c):= \rho^2\, \underline\ell\, c^{-1}$, we obtain the claim.
\end{proof}
\end{prop}

From now on, we  will briefly denote the \textbf{open sublevels} of the action $\Act$ and of the discrete action $\Act_k$ by
\[ (\Act)_c:=\Act^{-1}(-\infty,c),\s (\Act_k)_c:=\Act_k^{-1}(-\infty,c),\s\s\forall c\in\R. \]
Let $\rho_0=\rho_0(\Lagr)$, $\epsilon_0=\epsilon_0(\Lagr)$ and $\bar\epsilon=\bar\epsilon(\Lagr,\rho_0,c)$ be the positive constants given by propositions~\ref{p:fathi} and~\ref{p:estimate_sub}, and consider the integer
\begin{align*}
\bar k = \bar k(\Lagr,c) := \left\lceil\max\gbra{\frac1{\epsilon_0},\frac1{\bar\epsilon}}\right\rceil \in\N.
\end{align*}
We fix an integer $k\geq\bar k$ and we define a retraction $r:(\Act)_c\to(\Act_k)_c$ in the following way: for each $\zeta\in(\Act)_c$, the image $r(\zeta)$ is the unique $k$-broken Euler-Lagrange loop such that $r(\zeta)(\sfrac ik)=\zeta(\sfrac ik)$   for each $i\in\gbra{0,...,k-1}$, see figure~\ref{f:shortening}(a). Then, we define a homotopy $R:[0,1]\times(\Act)_c\to(\Act)_c$ as follows: for each  $i\in\gbra{0,...,k-1}$, $s\in [\frac i k,\frac{i+1}k]$ and $\zeta\in(\Act)_c$,  the loop  $R(s,\zeta)$ is defined as $R(s,\zeta)|_{[0,i/k]}=r(\zeta)|_{[0,i/k]}$, $R(s,\zeta)|_{[s,1]}=\zeta|_{[s,1]}$ and $R(s,\zeta)|_{[i/k,s]}$ 
is the unique action minimizer (with respect to the Lagrangian $\Lagr$) with endpoints $\zeta(\frac ik)$ and $\zeta(s)$, see figure~\ref{f:shortening}(b). By proposition~\ref{p:estimate_sub}, the homotopy $R$ and the map $r$ are well defined and we have $\Act(R(s,\zeta))\leq \Act(\zeta)$ for every $(s,\zeta)\in[0,1]\times(\Act)_c$. Moreover $R$ is a strong  deformation retraction. In fact, for  each $\zeta\in(\Act)_c$ we have $R(0,\zeta)=\zeta$, $R(1,\zeta)=r(\zeta)$ and, if $\zeta$ already belongs to $(\Act_k)_c$, we further have $R(s,\zeta)=\zeta$ for every $s\in[0,1]$.

\frag{0}{$\zeta(0)$}
\frag{1/5}{$\zeta(\sfrac15)$}
\frag{2/5}{$\zeta(\sfrac25)$}
\frag{3/5}{$\zeta(\sfrac35)$}
\frag{4/5}{$\zeta(\sfrac45)$}
\frag{1}{$\zeta(1)$}
\frag[n]{g}{$\zeta$}
\frag[n]{r(g)}{$r(\zeta)$}
\frag[n]{s(g)}{$R(s,\zeta)$}
\frag{s}{$\zeta(s)$}
\frag{(a)}{\textbf{(a)}}
\frag{(b)}{\textbf{(b)}}

\begin{figure} 
\begin{center}
\includegraphics[scale=0.95]{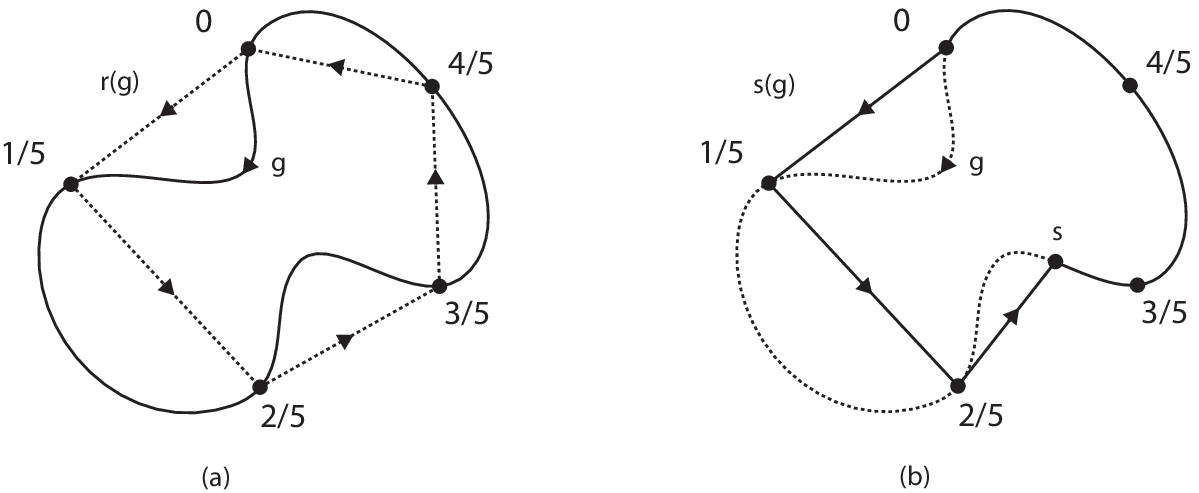}
\hcaption{\textbf{(a)} Example of a loop $\zeta$ (solid line) and the corresponding $r(\zeta)$ (dashed line),  for the case of the geodesics action functional on the flat $\R^2$, i.e. $\Lagr(t,q,v)= v_1^2+v_2^2$, and $k=5$.   \textbf{(b)} Homotoped loop $R(s,\zeta)$ (solid line).
}
\label{f:shortening}
\end{center}
\end{figure}

If $c_1<c_2\leq c$, the same homotopy $R$ can be used to show that the  pair $((\Act)_{c_2},(\Act)_{c_1})$ deformation retracts strongly  onto $((\Act_k)_{c_2},(\Act_k)_{c_1})$. Furthermore, if $\gamma\in\W(\T;M)$ is a critical point of $\Act$ with $\Act(\gamma)=c$, up to increasing $k$ we have that $\gamma$ belongs to $\Lambda_k$, and we can extend $R$ to a strong deformation retraction of the  pair $((\Act)_{c}\cup\gbra\gamma,(\Act)_{c})$ onto $((\Act_k)_{c}\cup\gbra\gamma,(\Act_k)_{c})$ such that $R(s,\gamma)=\gamma$ for every $s\in[0,1]$. Summing up, we have obtained the following.
   
\begin{lem}\label{l:HomotEquiv}
$ $
\begin{itemize}
\item[(i)] For each $c_1<c_2<\infty$  there exists $\bar k=\bar k(\Lagr,c_2)\in\N$ and, for every integer $k\geq\bar k$, a strong deformation retraction of the pair $\cbra{ (\Act)_ {c_2}   ,   (\Act)_ {c_1} }$ onto  $\cbra{ (\Act_k)_ {c_2},  (\Act_k)_ {c_1} }$.
\item[(ii)] For each critical point   $\gamma$ of $\Act$ with $\Act(\gamma)=c$,  there exists $\bar k=\bar k(\Lagr,c)\in\N$ and, for every integer $k\geq\bar k$, a strong deformation retraction of the pair $\cbra{ (\Act)_ c \cup \gbra\gamma ,   (\Act)_ c }$ onto  $\cbra{ (\Act_k)_ c \cup \gbra\gamma ,      (\Act_k)_ c }$.
\hfill$\qed$
\end{itemize}
\end{lem}

Let   $\gamma$ be a critical point of $\Act$ with critical value $c=\Act(\gamma)$. For every integer $k\geq\bar k(\Lagr,c)$ we have that $\gamma$ belongs to the $k$-broken Euler-Lagrange loop space $\Lambda_k$ and therefore it is a critical point of the discrete action $\Act_k$ as well.  We recall that the \textbf{local homology groups} of $\Act$ at $\gamma$  are  defined as
\begin{align*}
   \MC_*(\Act,\gamma)=\Hom_*\cbra{(\Act)_ c \cup \gbra\gamma ,   (\Act)_ c}, 
\end{align*}
where $\Hom_*$ denotes the singular homology functor with an arbitrary coefficient group (the  local homology groups  of the discrete action functional $\Act_k$ at $\gamma$ are defined analogously). The above lemma~\ref{l:HomotEquiv}(ii) has the following immediate consequence.
\begin{cor}\label{c:hom_equiv_loc_hom}
For each integer $k>\bar k(\Lagr,c)$ the inclusion 
\begin{align}\label{e:incl_pair}
\iota: \cbra{ (\Act_k)_ c \cup \gbra\gamma  ,      (\Act_k)_ c  }
\hookrightarrow
\cbra{ (\Act)_ c \cup \gbra\gamma ,   (\Act)_ c }
\end{align}
induces the   homology isomorphism
$\displaystyle \iota_*:  \MC_*(\Act_k,\gamma) \toup^\simeq  \MC_*(\Act,\gamma)$.
\hfill$\qed$
\end{cor}

It is well known that the local homology groups of a $C^2$ functional at a critical point are trivial in dimension that is smaller than the Morse index or bigger than the sum of the Morse index and the nullity (see \cite[corollary~5.1]{b:Ch}). By  lemma~\ref{l:HomotEquiv}(ii), we  recover this result for the $C^1$ action functional $\Act:\W(\T;M)\to\R$, at least for contractible critical points.

\begin{cor}\label{c:loc_A_interval}
Let $\gamma:\T\to M$ be a contractible loop that is a  critical point of the  action functional $\Act$ with Morse index $\iota(\gamma)$ and nullity $\nu(\gamma)$. Then, the  local homology groups $\MC_*(\Act,\gamma)$ are trivial if $*$ is less than $\iota(\gamma)$  or greater than $\iota(\gamma)+\nu(\gamma)$. 
\begin{proof}
For each sufficiently big $k\in\N$, $\gamma$ is also a  critical point of the discrete action $\Act_k$.  By proposition~\ref{p:SameMorseNullPair_total}, up to  increasing $k$ we have that $\Act$ and $\Act_k$ have the same Morse index and nullity pair $(\iota(\gamma),\nu(\gamma))$ at $\gamma$. By the above  corollary~\ref{c:hom_equiv_loc_hom}, up to further increasing $k$ we have $\MC_*(\Act_k,\gamma)\simeq\MC_*(\Act,\gamma)$. Since $\Act_k$ is smooth, the  local homology groups $\MC_*(\Act_k,\gamma)$ are trivial if $*$ is less than $\iota(\gamma)$  or greater than $\iota(\gamma)+\nu(\gamma)$, and the claim follows.
\end{proof}
\end{cor}

\section{Local Homology and Embeddings of Hilbert Spaces}\label{s:LocHomEmb}

In this section we will prove an abstract  Morse-theoretic result that might be of independent interest, and then we will discuss its application to the action functional of a convex quadratic-growth Lagrangian. The result will be an essential ingredient in the proof of the Lagrangian Conley conjecture.

Let us consider an open set  $\catU$ of a  Hilbert space  $\M$ and a $C^2$ functional   $\catF:\catU\to\R$ that satisfies the Palais-Smale condition. Let  $\m$ be a Hilbert subspace of $\M$ such that $\catu:=\catU\cap\m\neq\varnothing$ and  $\grad \catF(\bm y)\in\m$ for all $\bm y\in\catu$. This latter condition is equivalently expressed via the isometric inclusion  $J:\m\hookrightarrow\M$ as
\begin{align}
\label{e:commGradF}
(\grad \catF)\circ J=J\circ\grad(\catF\circ J).
\end{align}
Let  $\bm x\in \catU$ be an isolated critical point of $\catF$ that sits in the subspace $\m$, and let us further assume that the Morse index $\iota(\catF,\bm x)$  and the  nullity $\nu(\catF,\bm x)$ of $\catF$ at $\bm x$ are finite. We denote by $H=H(\bm x)$ the bounded self-adjoint  linear operator on $\M$ associated to the  Hessian of $\catF$ at $\bm x$, i.e. 
\begin{align}\label{e:HessF_H}
\Hess \catF(\bm x)[\bm v,\bm w]=  \abra{ H\bm v,\bm w}_\M,\s\s\forall \bm v,\bm w\in \M.
\end{align}
We require that $H$ is a Fredholm operator, so that the functional $\catF$ satisfies the hypotheses of the generalized Morse lemma (see \cite[page~44]{b:Ch}).

Throughout this section, for simplicity, all the homology groups are assumed to have coefficients in a field $\F$ (in this way we will avoid the torsion terms that appear in the  K\"unneth formula). We recall that the \textbf{local homology groups} of the functional $\catF$ at $\bm x$  are  defined as $\MC_*(\catF,\bm x)=\Hom_*\cbra{(\catF)_c \cup \gbra{\bm x}, (\catF)_c}$,  where $c=\catF(\bm x)$ and $(\catF)_c:=\catF^{-1}(-\infty,c)$. If we denote by  $\catf:\catu\to\R$ the restricted functional $\catF|_{\catu}$, then $\bm x$ is a critical point of $\catf$ as well and the local homology groups  $\MC_*(\catf,\bm x)$ are defined analogously as $\Hom_*((\catf)_c \cup \gbra{\bm x}, (\catf)_c)$. The  inclusion $J$ restricts to a continuous map of pairs 
\[
J:
((\catf)_c \cup \gbra{\bm x}, (\catf)_c)
\hookrightarrow 
((\catF)_c \cup \gbra{\bm x}, (\catF)_c).\]  
In this way, it induces the  homology homomorphism
\begin{align*} 
J_*:\MC_*(\catf,\bm x)\to \MC_*(\catF,\bm x). 
\end{align*}
The main result of this section is the following.

\begin{thm}\label{t:Hom*J}
If the Morse index and nullity pair of $\catF$ and $\catf$ at $\bm x$ coincide, i.e. 
\begin{align*}
(\iota(\catF,\bm x),\nu(\catF,\bm x))=(\iota(\catf,\bm x),\nu(\catf,\bm x)),
\end{align*}
then $J_*$ is an isomorphism of local homology groups.
\end{thm}
The proof of this theorem will be carried out in  subsection~\ref{s:IV4}, after several preliminaries. The reader might want to skip the remaining of section~\ref{s:LocHomEmb} (beside subsection~\ref{s:ApplMorseLagr}) on a first reading.

\begin{rem}
One might ask if theorem~\ref{t:Hom*J} still holds without the assumption~\eqref{e:commGradF}. This is   true in case $\bm x$ is a non-degenerate critical point: briefly, a relative cycle that represents a  generator of  $\MC_{\iota(\catF,\bm x)}(\catf,\bm x)$  also represents a  generator of  $\MC_{\iota(\catF,\bm x)}(\catF,\bm x)$, and all the other local homology groups $\MC_*(\catf,\bm x)$ and $\MC_*(\catF,\bm x)$, with $*\neq \iota(\catF,\bm x)=\iota(\catf,\bm x)$, are trivial. However, in the general case, assumption~\eqref{e:commGradF} is necessary, as it is   shown by the following simple example. Consider the functional  $\catF:\R^2\to\R$ given by 
\[\catF(x,y)=(y-x^2)(y-2x^2),\s\s\forall (x,y)\in\R^2.\] 
The origin $\bm 0$ is clearly an isolated critical point of $\catF$, and the corresponding Hessian is given in matrix form by 
\begin{align}\label{e:HessExAbbo}
\Hess \catF(0,0)=
  \qbra{\begin{array}{cc}
    0 & 0 \\ 
    0 & 2 \\ 
  \end{array}}.
\end{align} 
Now, let us consider the inclusion $J:\R\hookrightarrow\R^2$ given by $J(x)=(x,0)$, namely the inclusion of the $x$-axis  in $\R^2$.  The Morse index  of $\catF$ at the origin is~$0$ and   coincides with the Morse index of the restricted functional $\catF\circ J$. Analogously, the nullity  of $\catF$ and $\catF\circ J$ at the origin are both equal to~$1$. However the gradient of $\catF$ on the $x$-axis is given by
\[\grad \catF(x,0)=(8x^3,-3x^2),\s\s\forall x\in\R,\]
hence condition~\eqref{e:commGradF} is not satisfied, i.e.\ $(\grad \catF)\circ J\neq J\circ\grad (\catF\circ J)$. The local homology groups of $\catF$ and $\catF\circ J$ at the origin  are not isomorphic (and consequently $J_*$ is not an isomorphism). In fact, by examining the sublevel  $(\catF)_0$ (see figure~\ref{f:saddle}), it is clear that the origin is a saddle for $\catF$ and a  minimum for $\catF\circ J$. Therefore we have
\[\MC_*(\catF,\bm 0)=
  \left\{\begin{array}{ccc}
    \F && *=1, \\ 
    0 && *\neq 1, \\ 
  \end{array}
  \right.
  \s\s\s
 \MC_*(\catF\circ J,0)=
  \left\{\begin{array}{ccc}
    \F && *=0, \\ 
    0 && *\neq 0. \\ 
  \end{array}
  \right.  
\]  
\end{rem}

\frag{xx}{$x$}
\frag{yy}{$y$}
\frag{00}{$\bm0$}
\frag{(F)0}{ }

\begin{figure} 
\begin{center}
\includegraphics[scale=0.7]{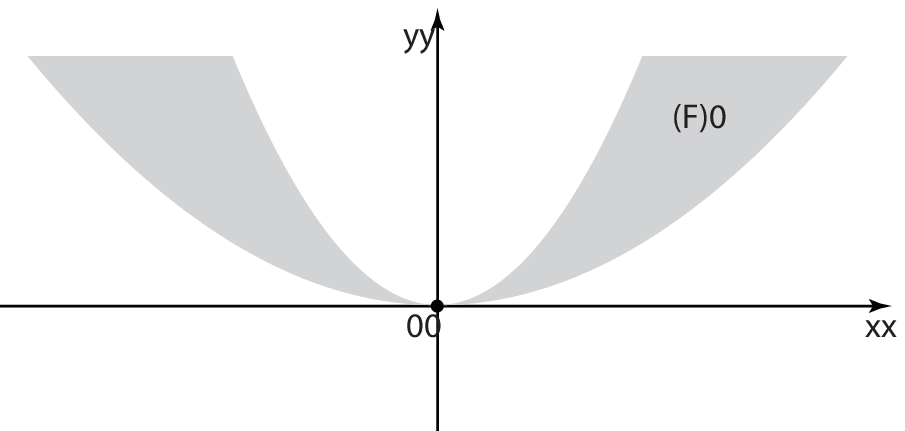}
\hcaption{Behavior of $\catF(x,y)=(y-x^2)(y-2x^2)$ around the critical point $\bm0$. The shaded region corresponds to the sublevel $(\catF)_0=\catF^{-1}(-\infty,0)$.
}
\label{f:saddle}
\end{center}
\end{figure}

\begin{rem}
Also the hypothesis of $C^2$ regularity of the involved functional is essential in order to obtain the assertion of theorem~\ref{t:Hom*J}. In fact, let us modify the functional $\catF$ of the previous remark in the following way
\[\catF(x,y)=(y-x^2)(y-2x^2)+3x^6\arctan\cbra{\frac y{x^4}},\s\s\forall (x,y)\in\R^2.\]
This functional is $C^1$ and twice Gateaux differentiable, but it is not $C^2$ at the origin, which is again a critical point of $\catF$. The Hessian of $\catF$ at the origin is still given by \eqref{e:HessExAbbo}, but the gradient of $\catF$ on the $x$-axis is now given by
\[ \grad \catF(x,0)=(8x^3,0),\s\s\forall x\in\R, \]
hence condition~\eqref{e:commGradF} is satisfied, i.e.\ $(\grad \catF)\circ J=J\circ\grad (\catF\circ J)$, where $J:\R\hookrightarrow\R^2$ is given by $J(x)=(x,0)$. The Morse index and nullity pair of $\catF$ at the origin is $(0,1)$ and coincides with  the Morse index and nullity pair of $\catF\circ J$ at $0$. Moreover, $0$ is a local minimum of $\catF\circ J$, which implies 
\[
 \MC_*(\catF\circ J,0)=
  \left\{\begin{array}{ccc}
    \F && *=0, \\ 
    0 && *\neq 0, \\ 
  \end{array}
  \right.  
\] 
However, the origin $\bm 0\in\R^2$ is not a local minimum of the functional $\catF$. In fact, a straightforward computation shows that $\bm0$ is a local maximum of the functional $\catF$ restricted to the parabola $y=\frac 32 x^2$, namely $0\in\R$ is a local maximum of the functional
\[ 
x\mapsto \catF\cbra{x,\sfrac 32 x^2}
=
-\frac14 
x^4 + 
3x^6 \arctan\cbra{\frac{3}{2x^2}}.
\]
This readily implies that $\MC_0(\catF,\bm0)=0\neq\MC_0(\catF\circ J,0)$, which contradicts the assertion of theorem~\ref{t:Hom*J}.
\end{rem}

\subsection{The generalized Morse lemma revisited}\label{s:IV2}

In order to prove theorem~\ref{t:Hom*J}, we need to give a more precise statement of the generalized Morse lemma. Everything that we will claim already follows from the classical proof (see  \cite[page 44]{b:Ch}). In order to simplify  the notation, from now on we will assume, without loss of generality, that $\bm x=\bm 0\in\M$ and hence $\catU\subset \M$ is an open neighborhood of $\bm 0$. According to the  operator $H$ associated to the Hessian of $\catF$ at the critical point $\bm 0$,  we have an orthogonal  splitting 
$\M= \M^+  \oplus \M^-  \oplus \M^0$, 
where $\M^+$ [resp.\ $\M^-$] is a closed subspace in which $H$ is positive definite [resp.\ negative definite], while  $\M^0$ is the finite-dimensional  kernel of $H$. We denote by $P^\pm:\M\to\M^\pm$  the linear projector onto $\M^\pm:=\M^+\oplus\M^-$. On  $\M^\pm\setminus\gbra{\bm0}$ we introduce the local flow $\Theta_H$ defined by $\Theta_H(s,\sigma(0))=\sigma(s)$, where $\sigma:(s_0,s_1)\to\M^\pm\setminus\gbra{\bm0}$ (with $s_0<0<s_1$) is a curve that satisfies  
\begin{align}\label{e:flow_line_sigma} 
\dot\sigma(s)=- \frac{H \sigma (s)}{\norm{H \sigma (s)}_\M},\s\s\forall s\in(s_0,s_1).
\end{align}
We also set $\Theta_H(0,\bm0):=\bm 0$. Then, the generalized Morse lemma may be restated as follows.

\begin{lem}[Generalized Morse Lemma revisited]\index{Morse!lemma!revisited}\index{lemma!generalized Morse --!revisited}\label{l:MorseRev}
With the above assumptions on $\catF$, there exists an open neighborhood $\catV\subseteq\catU$ of $\bm0$, a  homeomorphism onto its image
\begin{align*}
\phi: (\catV,\bm0) \to (\catU,\bm0)
\end{align*}
and a $C^1$ map
\begin{align*}
\psi: (\catV \cap \M^0,\bm0) \rightarrow (\M^\pm,\bm0),
\end{align*}
such that the following assertions hold.
\begin{itemize} 
\item[$(\mathrm{i})$] 
For each $\vv\in \catV$, if we write $\vv=\vv^0+\vv^\pm$ according to the splitting $\M=\M^0\oplus\M^\pm$, we have 
\[ \catF\circ\phi(\vv)= \underbrace{\catF\cbra{\vv^0 + \psi(\vv^0)  }}_{=:\catF^0(\vv^0)} + \underbrace{\sfrac 1 2 \abra{H \vv^\pm,\vv^\pm}_\M}_{=:\catF^\pm(\vv^\pm)}.\]

\item[$(\mathrm{ii})$] 
The origin $\bm0$ is a critical point of both $\catF^0$ and  $\catF^\pm$.

\item[$(\mathrm{iii})$]
The map $\psi$ is implicitly defined by
\begin{align*}
& P^\pm\cbra{ \grad \catF( \vv^0 + \psi(\vv^0)) }=\bm0,\s\s \forall\vv^0\in \catV \cap \M^0,    \\
& \psi(\bm0)=\bm0.
\end{align*}

\item[$(\mathrm{iv})$] The  homeomorphism $\phi$ is given by
\begin{align*}
&\phi^{-1}(\vv)= \vv^0 + \Theta_H(\tau(\vv-\psi(\vv^0)),\vv^\pm -\psi(\vv^0)) ,\s\s \forall\vv=\vv^0+\vv^\pm\in\phi(\catV),
\end{align*}
where  $\tau$ is a continuous function defined in the following way: for each $\vv=\vv^0+\vv^\pm$ that belongs to its domain, $\tau(\vv)$ is the only real number  satisfying
\begin{align*}
&|\tau(\vv)|< \|\vv^\pm\|_\M,\\
&\catF\cbra{\vv+\psi(\vv^0)} 
-  
\catF\cbra{\vv^0+\psi(\vv^0)}
   				 =
   				 \catF^\pm(\Theta_H(\tau(\vv),\vv^\pm)).
\end{align*}
\hfill$\qed$
\end{itemize}
\end{lem}

\subsection{Naturality of the Morse lemma}\label{s:IV3}

Let $\hhh$  be the bounded self-adjoint linear    operator on $\m\subset\M$  associated to the Hessian of the restricted functional $\catf=\catF|_{\m}$ at $\bm0$. Then 
\begin{align}\label{e:restr_H_H.} 
H|_{\m}=\hhh.
\end{align}
In fact, by condition~\eqref{e:commGradF}, we have
\begin{align*}
H\circ J=
\diff(\grad \catF)(\bm 0) \circ J= 
\diff(\underbrace{(\grad \catF)\circ J}_{  J\circ\grad(\catf)  })(\bm 0) =
J \circ\diff(\grad (\catf))(\bm 0)=
J\circ \hhh.
\end{align*}
In particular $\hhh$  is a   Fredholm operator on $\m$. If we denote by  $\m=\m^0\oplus\m^+\oplus\m^-$ the orthogonal splitting defined by the operator $\hhh$, equation~\eqref{e:restr_H_H.} readily implies that 
\begin{align}\label{e:incl_spl}
\m^0\subseteq\M^0,\s \m^+\subseteq\M^+,\s \m^-\subseteq\M^-,
\end{align}
and moreover, if we denote by $\p^\pm:\m\to\m^\pm$ the orthogonal projector onto $\m^\pm=\m^+\oplus\m^-$, this latter turns out to be the restriction of the projector $P^\pm:\M\to\M^\pm$ to $\m$, i.e.
\begin{align}\label{e:P_p} 
P^\pm|_{\m}=\p^\pm.
\end{align}

The hypotheses of the generalized Morse Lemma are fulfilled by both the  functional $\catF$ and its restriction $\catf$. The following  is the long list of the symbols involved in the statement of lemma~\ref{l:MorseRev}, and we write in the subsequent line the corresponding list of symbols involved in  the statement referred to the restricted functional $\catf$:
\begin{align*}
  \begin{array}{llllllllll}
    \M^\pm,& \M^0,& P^\pm,& \catV,& \Theta_H,&  \phi,& \psi,& \tau,&
    \catF^0,& \catF^\pm, \\ 
	\m^\pm,& \m^0,& \p^\pm,& \catv,& \Theta_\hhh,&  \phi_\bullet,&
    \psi_\bullet,& \tau_\bullet,& \catf^0,& \catf^\pm.
  \end{array}
\end{align*}
We want to show that, under the hypotheses of theorem~\ref{t:Hom*J}, the decomposition $\catF^\pm+\catF^0$ of $\catF$, given by the generalized Morse lemma, restricts  to the corresponding decomposition  $\catf^\pm+\catf^0$ of $\catf$. 

\begin{lem} \label{l:IsomLin}
$ $
\begin{itemize}
\item[$(\mathrm{i})$] 
If $\iota(\catF,\bm0)=\iota(\catf,\bm0)$,  then $\M^-=\m^-$. 
\item[$(\mathrm{ii})$] 
If $\nu(\catF,\bm0)=\nu(\catf,\bm0)$, then $\M^0=\m^0$. 
\end{itemize}
\begin{proof}
The claims follow at once from \eqref{e:incl_spl}, since 
\begin{align*}
\dim \m^-&=\iota(\catf,\bm0)=\iota(\catF,\bm0)=\dim \M^-,\\
\dim \m^0&=\nu(\catf,\bm0)=\nu(\catF,\bm0)=\dim \M^0.
\qedhere
\end{align*}
\end{proof}
\end{lem}

\begin{prop}\label{p:commut_hh_phi}
If $\nu(\catF,\bm0)=\nu(\catf,\bm0)$, the following equalities hold (on some neighborhood of the critical point $\bm 0$ where the involved maps are defined):
\begin{itemize}
  \item[$(\mathrm{i})$] $\psi=\psi_\bullet$,
  \item[$(\mathrm{ii})$] $\phi|_{\m}= \phi_\bullet$.
\end{itemize}
\begin{proof}
By lemma~\ref{l:IsomLin}(ii), the domains of the maps $\psi$ and $\psi_\bullet$ are open neighborhoods of $\bm 0$ in $\M^0=\m^0$. Up to shrinking these neighborhoods, we can assume that both $\psi$ and $\psi_\bullet$ have common domain $\catV^0\subset\M^0$. By lemma~\ref{l:MorseRev}(iii) we have $\psi(\bm0)=\psi_\bullet (\bm0)=\bm0$, and all we have to do in order to conclude the proof of (i) is to show that, for each $\bm v^0\in\catV^0\setminus\gbra{\bm0}$, the maps $\psi_\bullet(\bm v^0)$ and $\psi(\bm v^0)$ are implicitly defined by the same equation, that is
\[
P^\pm\cbra{ \grad \catF( \vv^0 + \psi(\vv^0)) }=
\bm0=
P^\pm\cbra{ \grad \catF( \vv^0 + \psi_\bullet(\vv^0)) }.
\] 
This is easily verified since, by \eqref{e:commGradF} and \eqref{e:P_p}, we have
\[
P^\pm( \grad \catF( \vv^0 + \psi_\bullet(\vv^0))  )=
\p^\pm( \grad \catf( \vv^0 + \psi_\bullet(\vv^0)) )=
\bm0.
\]

For  (ii), up to shrinking the domains of $\phi$ and $\phi_\bullet$, we can assume that they are maps of the form $\phi:\catV\to\catU$ and $\phi_\bullet:\catv\to\catu$, where $\catv=\catV\cap\m$. Being $\phi$ and $\phi_\bullet$  homeomorphisms onto their images, we can equivalently prove that $\phi^{-1}=\phi_\bullet^{-1}$ on  the open set $\phi_\bullet(\catv)\subset\catu$. To begin with, notice that~\eqref{e:restr_H_H.} readily implies that the flow $\Theta_{\hhh}$ is the  restriction of  the flow $\Theta_H$ to $\m^\pm\setminus\gbra{\bm0}$, i.e.
\begin{align*}
\Theta_H(\cdot,\vv^\pm)=\Theta_{\hhh}(\cdot,\vv^\pm),\s\s\forall \vv^\pm\in\m^\pm\setminus\gbra{\bm0}.
\end{align*}
By lemma~\ref{l:MorseRev}(iv) and since $\psi=\psi_\bullet$, for each $\vv=\vv^0+\vv^\pm\in\phi_\bullet(\catv)$   we have
\begin{align*}
\phi^{-1}(\vv)&= \vv^0 + \Theta_H(\tau(\vv -\psi(\vv^0)),\vv^\pm -\psi(\vv^0)),\\
\phi_\bullet^{-1}(\vv)&= \vv^0 + \Theta_H(\tau_\bullet(\vv -\psi(\vv^0)),\vv^\pm -\psi(\vv^0)).
\end{align*}
Hence, in order to conclude the proof of (ii) we just need to show that, for each $\vv$ in the domain of $\tau_\bullet$, we have $\tau(\vv)=\tau_\bullet(\vv)$. This is easily verified   
since, by lemma~\ref{l:MorseRev}(iv), $\tau(\vv)$ and $\tau_\bullet(\vv)$ are implicitly defined by the same  equation
\begin{align*}
				\frac 1 2 
				\abra{H\, \Theta_H(\tau(\vv),\vv^\pm), 
				\Theta_H(\tau(\vv),\vv^\pm)}_\M
&= \catF\cbra{\vv+\psi(\vv^0)} - \catF\cbra{\vv^0+\psi(\vv^0)}\\
&=\frac 1 2 
				\abra{H\, \Theta_H(\tau_\bullet(\vv),\vv^\pm), 
				\Theta_H(\tau_\bullet(\vv),\vv^\pm)}_\M.
\qedhere
\end{align*}
\end{proof}
\end{prop}

\begin{cor}\label{c:comm_Fpm_F0}
If $\nu(\catF,\bm0)=\nu(\catf,\bm0)$ then $\catF^\pm|_{\m}=\catf^\pm$ and $\catF^0=\catf^0$.
\hfill$\qed$
\end{cor}

\subsection{Local Homology} \label{s:IV4}
Before going to the proof of theorem~\ref{t:Hom*J}, we need to establish another   naturality property, this time for the isomorphism between the local homology of $\catF$ at $\bm0$ and the homology of the corresponding Gromoll-Meyer pairs. For the reader's convenience, let us briefly recall the needed definition applied to our setting. We denote by $\Phi_\catF$ the anti-gradient flow of $\catF$, i.e.
\[\pder {\Phi_\catF}{t} (t,\bm y)=-\grad \catF(\Phi_\catF(t,\bm y)),\s\s\Phi_\catF(0,\cdot)=\id_{\M}.\]
A pair of topological spaces $(\catW,\catW_-)$ is called a \textbf{Gromoll-Meyer pair} for $\catF$ at the critical point $\bm0$ when
\begin{itemize}
  \item[\textbf{(GM1)}] $\catW\subset\M$ is a closed neighborhood of $\bm0$ that does not contain other critical points of~$\catF$;
  
  \item[\textbf{(GM2)}] if $\catF(0)=c$, there exists $\epsilon>0$ such that  $[c-\epsilon,c)$ does not contain critical values of $\catF$, and $\catW\cap(\catF)_{c-\epsilon}=\varnothing$;

  \item[\textbf{(GM3)}] if $t_1<t_2$ are such that $\Phi_\catF(t_1,\bm y),\Phi_\catF(t_2,\bm y)\in\catW$ for some $\bm y\in\M$, then  $\Phi_\catF(t,\bm y)\in\catW$ for all $t\in[t_1,t_2]$;
  
  \item[\textbf{(GM4)}] $\catW_- =\gbra{\bm y\in\catW\,|\, \Phi_\catF((0,\infty)\times\gbra{\bm y})\subset\M\setminus\catW}$  is a piecewise submanifold of $\M$ trans\-ver\-sal to the flow $\Phi_\catF$.
\end{itemize}
It is always possible to build a Gromoll-Meyer pair  $(\catW,\catW_-)$   for $\catF$ at $\bm0$, and we have  $\Hom_*(\catW,\catW_-)\simeq\MC_*(\catF,\bm0)$, see  \cite[page~48]{b:Ch}.

\begin{lem}\label{l:nat_GM}
Let $(\catW,\catW_-)$ be a Gromoll-Meyer pair for $\catF$ at $\bm0$. Then, the following holds.
\begin{itemize}
  \item[$(\mathrm{i})$] The pair $(\catW_\bullet,\catW_{\bullet-}):=(\catW\cap\m,\catW_-\cap\m)=(J^{-1}(\catW),J^{-1}(\catW_-))$ is a Gromoll-Meyer pair for   $\catf=\catF|_{\m}$ at $\bm0$.
  \item[$(\mathrm{ii})$] Consider the restrictions of $J:\m\hookrightarrow\M$ given by
 \begin{align*}
 & J:
((\catf)_c \cup \gbra{\bm x}, (\catf)_c)
\hookrightarrow 
((\catF)_c \cup \gbra{\bm x}, (\catF)_c),\\
 & J: (\catW_\bullet,\catW_{\bullet-})\hookrightarrow(\catW,\catW_-).
 \end{align*}
These restrictions induce the homology homomorphisms  
 \begin{align*}
 & J_*: \MC_*(\catf,\bm0)\to\MC_*(\catF,\bm0),\\
 & J_*: \Hom_*(\catW_\bullet,\catW_{\bullet-})\to\Hom_*(\catW,\catW_-).
 \end{align*}
Then, there exist homology isomorphisms $\iota_{(\catW_\bullet,\catW_{\bullet-})}$ and $\iota_{(\catW,\catW_{-})}$ such that the following diagram  commutes.
\begin{align*}
\xymatrix{
\MC_*(\catf,\bm0)    
\ar[rr]^{J_*}  
\ar[dd]_{\iota_{(\catW_\bullet,\catW_{\bullet-})}}^{\simeq}    
 & & 
\MC_*(\catF,\bm0)
\ar[dd]^{\iota_{(\catW,\catW_{-})}}_{\simeq}    
\\\\
\Hom_*(\catW_\bullet,\catW_{\bullet-})
\ar[rr]^{J_*}  
 & & 
\Hom_*(\catW,\catW_-) 
}
\end{align*}
\end{itemize}
\begin{proof}
Part (i) just requires the straightforward  verification that the pair $(\catW_{\bullet},\catW_{\bullet-})$ satisfies conditions \textbf{(GM1)},...,\textbf{(GM4)}. Part (ii) requires to examine  the isomorphism between the homology of a Gromoll-Meyer pair and a corresponding   local homology group (we refer the reader to \cite[page~48]{b:Ch} for  more details on what we claim). The point, here, is to show that this isomorphism is given by the composition of homology isomorphisms induced by maps, so that the assertion follows from the functoriality of singular homology.

Notice that, by the assumption~\eqref{e:commGradF}, the anti-gradient flow $\Phi_{\catF}$ of $\catF$ restricts on $\m$ to the anti-gradient flow  $\Phi_{\catf}$ of the restricted functional $\catf$. 
We introduce the sets $\catY$ and $\catY_\bullet$ given by
\begin{align*}
\catY &:= \Phi_{\catF}\cbra{ [0,\infty)\times\catW }, \\
\catY_\bullet &:= \Phi_{\catf}\cbra{ [0,\infty)\times\catW_\bullet }=\Phi_{\catF}\cbra{ [0,\infty)\times\catW_\bullet }=\catY\cap\m,
\end{align*}
and we consider the following diagram. 
\begin{align*}
\small
\xymatrix{
\MC_*(\catf,\bm0)\ar[rrr]^{J_*}
&&&
\MC_*(\catF,\bm0)\\\\
\Hom_*\cbra{\catY_\bullet\cap(\catf)_c\cup\gbra{\bm0},
\catY_\bullet\cap(\catf)_c}
\ar[rrr]
\ar[uu]^{\simeq}
\ar[dd]_{\simeq}&&&
\Hom_*\cbra{\catY\cap(\catF)_c\cup\gbra{\bm0},
\catY\cap(\catF)_c}
\ar[uu]_{\simeq}
\ar[dd]^{\simeq}&&&\\\\
\Hom_*\cbra{\catY_\bullet,
\catY_\bullet\cap(\catf)_c}
\ar[rrr]&&&
\Hom_*\cbra{\catY,
\catY\cap(\catF)_c}\\\\
\Hom_*(\catW_\bullet,\catW_{\bullet-})
\ar[rrr]^{J_*}
\ar[uu]^{\simeq}&&&
\Hom_*(\catW,\catW_-)
\ar[uu]_{\simeq} 
}
\end{align*}
In this diagram, all the arrows are homology homomorphisms induced by inclusions. Moreover, all the vertical arrows are isomorphisms (this fact is proved by anti-gradient flow deformations and excisions), and we define the isomorphisms $\iota_{(\catW_\bullet,\catW_{\bullet-})}$ and $\iota_{(\catW,\catW_{-})}$ as the composition of the whole left vertical line and right vertical line respectively. By  the functoriality of singular homology, this diagram is commutative, and the claim of part (ii) follows.
\end{proof}
\end{lem}

After these preliminaries, let us go back to the proof of theorem~\ref{t:Hom*J}. First of all, if we assume $\nu(\catF,\bm0)=\nu(\catf,\bm0)$,  corollary~\ref{c:comm_Fpm_F0} implies that $\catF^0=\catf^0$. Hence the inclusion $J$ restricts to the identity map on the pair 
\[((\catf^0)_c\cup\gbra{\bm0},(\catf^0)_c)=((\catF^0)_c\cup\gbra{\bm0},(\catF^0)_c),\]
and therefore 
\begin{align}\label{e:Hom*J0}
\MC_*(\catf^0,\bm0)=\MC_*(\catF^0,\bm0).
\end{align}
For the Morse functionals $\catf^\pm$ and $\catF^\pm$ we have the following result.
\begin{lem}\label{l:Hom*Jpm}
If $(\iota(\catF,\bm0),\nu(\catF,\bm0))=(\iota(\catf,\bm0),\nu(\catf,\bm0))$ then the inclusion $J$, restricted as a map
\begin{align}\label{e:J_restr_pm}
J:((\catf^\pm)_c\cup\gbra{\bm0},(\catf^\pm)_c)
\hookrightarrow
((\catF^\pm)_c\cup\gbra{\bm0},(\catF^\pm)_c),
\end{align}
induces the  homology isomorphism $\displaystyle J_*:\MC_*(\catf^\pm,\bm0)\toup^\simeq\MC_*(\catF^\pm,\bm0)$.
\begin{proof}
The fact that $J$ restricts to a map of the form~\eqref{e:J_restr_pm} is guaranteed by corollary~\ref{c:comm_Fpm_F0}. Moreover,  lemma~\ref{l:IsomLin}(i) guarantees that $\m^-=\M^-$. Hence $J$ further restricts to a homeomorphism
\begin{align*}
\tilde J:(\M^-\cap(\catf^\pm)_c\cup\gbra{\bm0},\M^-\cap(\catf^\pm)_c)
\toup^\simeq
(\M^-\cap(\catF^\pm)_c\cup\gbra{\bm0},\M^-\cap(\catF^\pm)_c),
\end{align*}
and we obtain the following commutative diagram of inclusions.
\[
\xymatrix{
((\catf^\pm)_c\cup\gbra{\bm0},(\catf^\pm)_c)
\bigr.\ 
\ar@{^{(}->}[rr]^J &&
\bigl.\ ((\catF^\pm)_c\cup\gbra{\bm0},(\catF^\pm)_c)\\\\
(\M^-\cap(\catf^\pm)_c\cup\gbra{\bm0},\M^-\cap(\catf^\pm)_c)\Bigr.
\ar[rr]^{\tilde J}_\simeq 
\ar@{^{(}->}[uu]^{k_\bullet}_\sim &&
(\M^-\cap(\catF^\pm)_c\cup\gbra{\bm0},\M^-\cap(\catF^\pm)_c)\Bigr.
\ar@{^{(}->}[uu]^{k}_\sim
}
\]
It is well known that $k_\bullet$ and $k$ are homotopy equivalences. Therefore $J_*=k_*\circ \tilde J_*\circ (k_{\bullet*})^{-1}$ is a homology isomorphism.
\end{proof}
\end{lem}

\begin{proof}[Proof of theorem~\ref{t:Hom*J}.]
The homeomorphisms $\phi$ and $\phi_\bullet$ obtained by the Morse lemma induce local homology isomorphisms $\phi_*$ and $\phi_{\bullet*}$ such that the following diagram commutes.
\[
\xymatrix{
\MC_*(\catf,\bm0)    \ar[rr]^{J_*} &&
\MC_*(\catF,\bm0)  \\\\
\MC_*(\catf^0 + \catf^\pm,\bm0) 
\ar[rr]^{J_*} 
\ar[uu]^{\phi_{\bullet*}}_{\simeq} &&
\MC_*(\catF^0 + \catF^\pm,\bm0)    
\ar[uu]_{\phi_{*}}^{\simeq} 
}
\]
Hence, we only need to prove that the lower horizontal homomorphism $J_*$ is an isomorphism. We consider Gromoll-Meyer pairs $(\catW^\pm,\catW^\pm_-)$ and $(\catW^0,\catW^0_-)$ for $\catF^\pm$ and $\catF^0$ respectively at $\bm 0$, so  that the cross product of these pairs, which is
\[ 
(\catW,\catW_-):=
(\catW^\pm\times \catW^0,(\catW^\pm_-\times\catW^0)\cup(\catW^\pm\times\catW^0_-)), \]
is a Gromoll-Meyer pair for $\catF^0+\catF^\pm$ at $\bm 0$. Then, by lemma~\ref{l:nat_GM}, we obtain Gromoll-Meyer pairs for the functionals $\catf^\pm$, $\catf^0$ and $\catf^0+\catf^\pm$ at $\bm0$ respectively as 
\begin{align*}
(\catW_\bullet^\pm,\catW_{\bullet-}^\pm)&:=(\catW^\pm\cap\m,\catW_-^\pm\cap\m)=(J^{-1}(\catW^\pm),J^{-1}(\catW_-^\pm)),\\
(\catW_\bullet^0,\catW_{\bullet-}^0)&:=(\catW^0\cap\m,\catW_-^0\cap\m)=(J^{-1}(\catW^0),J^{-1}(\catW_-^0)),\\
(\catW_\bullet,\catW_{\bullet-})&:=(\catW\cap\m,\catW_-\cap\m)=(J^{-1}(\catW),J^{-1}(\catW_-))
\end{align*}
and, together with the K\"unneth formula\index{K\"unneth formula}, we obtain the following commutative diagram.
\[
\xymatrix{
\MC_*(\catf^0 + \catf^\pm,\bm0) 
\ar[rrr]^{J_*} 
\ar[dd]_{\iota_{(\catW_\bullet,\catW_{\bullet-})}}^\simeq 
&&&
\MC_*(\catF^0 + \catF^\pm,\bm0)
\ar[dd]^{\iota_{(\catW,\catW_-)}}_\simeq 
\\\\
\Hom_*(\catW_\bullet,\catW_{\bullet-})
\ar[rrr]^{J_*}
\ar[dd]_{\mathrm{K\ddot{u}nneth}}^\simeq
&&&    
\Hom_*(\catW,\catW_{-})
\ar[dd]^{\mathrm{K\ddot{u}nneth}}_\simeq
\\\\
{
\begin{array}{c}
\Hom_*(\catW_\bullet^\pm,\catW_{\bullet-}^\pm) \\
\otimes \\
\Hom_*(\catW_\bullet^0,\catW_{\bullet-}^0)
\end{array}
}
\ar[rrr]^{J_*\,\otimes\,J_*} 
&&&
{
\begin{array}{c}
\Hom_*(\catW^\pm,\catW_{-}^\pm) \\
\otimes \\
\Hom_*(\catW^0,\catW_{-}^0)
\end{array}
}\\\\
{
\begin{array}{c}
\Hom_*(\catf^\pm,\bm0) \\
\otimes \\
\Hom_*(\catf^0,\bm0)
\end{array}
}
\ar[rrr]^{J_*\,\otimes\,J_*}_{\simeq}
\ar[uu]^{\iota_{(\catW_\bullet^\pm,\catW_{\bullet-}^\pm)}\,\otimes\,\iota_{(\catW_\bullet^0,\catW_{\bullet-}^0)}}_{\simeq} 
&&&
{
\begin{array}{c}
\Hom_*(\catF^\pm,\bm0) \\
\otimes \\
\Hom_*(\catF^0,\bm0)
\end{array}
}
\ar[uu]_{\iota_{(\catW^\pm,\catW_{-}^\pm)}\,\otimes\,\iota_{(\catW^0,\catW_{-}^0)}}^{\simeq} 
}
\]
The commutativity of the upper and lower squares follows from lemma~\ref{l:nat_GM}, while the commutativity of the central square follows from the naturality of the K\"unneth formula (see for instance~\cite[page~275]{b:Hat}). By~\eqref{e:Hom*J0} and  lemma~\ref{l:Hom*Jpm}, the lower horizontal homomorphism $J_*\otimes J_*$ is an isomorphism, and so must be all the others horizontal homomorphisms.
\end{proof}

\subsection{Application to the Lagrangian action functional} \label{s:ApplMorseLagr}

We conclude this section showing that the abstract theorem~\ref{t:Hom*J} applies when  $J$ is the iteration map  and $\catF$ is the mean action functional associated to a convex quadratic-growth Lagrangian. To be more precise, the action functional does not fulfill the hypotheses of theorem~\ref{t:Hom*J}, since it is not $C^2$ in general (the $C^2$-regularity is needed by the Generalized Morse Lemma). However, we can still obtain the assertions of theorem~\ref{t:Hom*J} by means of the discretization  technique developed in section~\ref{s:discretizations}.

Let $\Lagr:\T\times\Tan M\to\R$ be a    convex quadratic-growth Lagrangian  with action $\act\tau$ and let $\gamma:\TT\tau\to M$ be a contractible $\tau$-periodic solution of the Euler-Lagrange system of $\Lagr$ (namely, a contractible critical point of $\act\tau$) with $\act\tau(\gamma)=c$. In order to simplify the notation, let us assume that $\tau=1$.

\begin{prop}\label{p:loc_hom_holds_true}
Let $n\in\N$ be such that $(\iota(\gamma),\nu(\gamma))=(\iota(\gamma\iter n),\nu(\gamma\iter n))$. Then the iteration map $\itmap n$, restricted as a map of pairs  of the form
\[ 
\itmap n : 
( (\Act)_c\cup\gbra\gamma,(\Act)_c )
\hookrightarrow
( (\act n)_c\cup\gbra\gamma, (\act n)_c ), 
\]
induces the homology isomorphism   
$\displaystyle\itmap n_*:
\MC_*(\Act,\gamma)
\toup^\simeq
\MC_*(\act{n},\gamma\iter n)$. 
\begin{proof} 
By corollary~\ref{c:hom_equiv_loc_hom}, for all $k\in\N$  sufficiently big,  the inclusion  
\[ \iota: 
( (\Act_k)_c \cup \gbra\gamma , (\Act_k)_c ) 
\hookrightarrow 
( (\Act)_c \cup \gbra\gamma , (\Act)_c ) \]
 induces an isomorphism in homology. Now, let  $\LL n_k$   be the $n$-periodic analogue of the $k$-broken Euler-Lagrange loop space $\Lambda_k$ (see section~\ref{s:discr_act_funct}). Namely, $\LL n_k$ is the subspace of $\W(\TT n;M)$ consisting of those loops $\zeta:\TT n\to M$ such that $\dist(\zeta(\frac ik),\zeta(\frac {i+1}k))<\rho_0$ and $\zeta|_{[i/k,(i+1)/k]}$ is an action  minimizer for each $i\in\gbra{0,...,nk-1}$ (here, $\rho_0$ is the constant given by proposition~\ref{p:fathi}). We denote by $\act n_k$ the restriction of $\act n$ to $\LL n_k$. Notice that the iteration map restricts as a continuous map of pairs of the form
\begin{align}\label{e:restr_iter}
\itmap n : 
( (\Act_k)_c\cup\gbra\gamma,(\Act_k)_c )
\hookrightarrow
( (\act n_k)_c\cup\gbra\gamma, (\act n_k)_c ). 
\end{align}
Moreover, as before, the inclusion  
\[ \iota\iter n: 
( (\act n_k)_c \cup \gbra\gamma , (\act n_k)_c ) 
\hookrightarrow 
( (\act n)_c \cup \gbra\gamma , (\act n)_c ) \]
 induces an isomorphism in homology, and   the following diagram commutes.
\begin{align*}
\xymatrix{
\MC_*(\Act,\gamma)  
\ar[rr]^{\itmap n_*}    
& & 
\MC_*(\act n,\gamma)\\\\
\MC_*(\Act_k,\gamma) 
\ar[rr]^{\itmap n_*} 
\ar[uu]^{\iota_*}_{\simeq}  
& &   
\MC_*(\act n_k,\gamma)   
\ar[uu]_{\iota\iter n_*}^{\simeq} 
}
\end{align*}
By proposition~\ref{p:SameMorseNullPair_total}, up to choosing a sufficiently big discretization pass $k\in\N$, the Morse index and nullity of the critical points of the action functional   do not change under discretization. Therefore, all we have to do in order to conclude the proof of the proposition is to establish the analogous claim for the restricted iteration map~\eqref{e:restr_iter}.

Applying the localization argument of section~\ref{s:discr_act_funct} around $\gamma$, we can assume that our convex quadratic-growth Lagrangian function has the form $\Lagr:\T\times U\times \R^N\to\R$, where $U$ is an open neighborhood of the origin in $\R^N$, and the corresponding action and mean action have the form $\Act:\W(\T;U)\to \R$ and $\act n:\W(\TT n;U)\to\R$. In this way we identify $\gamma$ with the point $\bm0\in\W(\T;U)$. Notice that the claim that we are proving is precisely the assertion of the abstract  theorem~\ref{t:Hom*J}   when $\catF$ is the  discrete mean action $\act n_k$ and $J$ is the iteration map $\itmap n$. All we have to do in order to conclude is to verify that  theorem~\ref{t:Hom*J} applies in our situation.

First of all, we recall that an open neighborhood $\catu\subset \Lambda_k$ of $\gamma\equiv\bm0$  can be identified with an open set of $\R^{Nk}$ by the diffeomorphism
\[ \zeta\mapsto \cbra{\textstyle\zeta(0), \zeta(\frac 1k) , ..., \zeta(\frac{k-1}k) }, \s\s\forall \zeta\in\catu. \]
Analogously, we can identify an open neighborhood $\catU\subset \LL n_k$ of $\gamma\iter n\equiv\bm0$   with an open set of $\R^{Nnk}$ by the diffeomorphism
\[ \sigma\mapsto \cbra{\textstyle\sigma(0), \sigma(\frac 1{nk}) , ..., \sigma(\frac{nk-1}{nk}) }, \s\s\forall \sigma\in\catU. \]
With these identifications, the iteration map $\itmap n$ is the restriction of an injective linear map $\R^{Nk}\hookrightarrow\R^{Nnk}$, that we still denote by $\itmap n$, given by
\[ \itmap n (w) = ( \underbrace{\bigl.  w,..., w \bigr.}_{n\mbox{ }\mathrm{times}}  ),\s\s\forall w\in\R^{Nk}.  \]
This map is an isometry with respect to the standard inner product  $\aabra{\cdot,\cdot}$  on $\R^{Nk}$ and 
the inner product   $\aabra{\cdot,\cdot}\iter n$ on $\R^{Nnk}$ obtained multiplying by $n^{-1}$ the standard one, i.e.
\begin{align*}
&\aabra{ w, z} = \sum_{j=0}^{k-1} \abra{w_j,z_j},
&\forall  w=(w_0,...,w_{k-1}), z=(z_0,...,z_{k-1})\in\R^{Nk},\\
&\aabra{ w', z'}\iter n = \frac1 n \sum_{j=0}^{nk-1} \abra{w_j',z_j'},
&\forall  w'=(w_0',...,w_{nk-1}'), z'=(z_0',...,z_{nk-1}')\in\R^{Nnk},
\end{align*}
where $\langle\cdot,\cdot\rangle$ denotes the standard inner product of $\R^N$.

Now, in order to conclude, the last hypothesis of theorem~\ref{t:Hom*J} that must be verified  is the condition expressed in~\eqref{e:commGradF}, that in our setting becomes 
\begin{align}\label{e:commGradIter}
 \grad \act n_k(\zeta\iter n)=\itmap n\circ \grad\Act_k(\zeta),\s\s\forall\zeta\in\catu,
 \end{align}
where the gradients of the action functionals $\Act_k$ and $\act n_k$ are computed with respect to the above inner products on $\R^{Nk}$ and $\R^{Nnk}$.

For each $\zeta\in\catu$ and $\xi\in\Tan_{\zeta}\Lambda_k$, we have
\begin{align*}
\diff\Act_k(\zeta)\,\xi=& 
\sum_{h=0}^{k-1} \int_{h/k}^{(h+1)/k} 
\bigl({  
	\langle
	\partial_v \Lagr(t,\zeta,\dot\zeta),\dot\xi
	\rangle
	+
	\langle
	\partial_q \Lagr(t,\zeta,\dot\zeta),\xi
	\rangle
}\bigr)
\diff t\\
=&
\sum_{h=0}^{k-1}  
	\abra{
	{\textstyle
	\partial_v \Lagr ({\frac hk,\zeta(\frac hk),\dot\zeta(\frac hk\sm-)})-
	\partial_v \Lagr ({\frac hk,\zeta(\frac hk),\dot\zeta(\frac hk\sm+)})
	,
	\xi(\frac hk)
	}
	}.
\end{align*}
This computation, together with lemma~\ref{l:TLambda}, shows that  $\grad\Act_k(\zeta):\T\to\R^N$ is the  element of  $\Tan_\zeta\Lambda_k$ given by
\[  
    \textstyle
    \grad\Act_k(\zeta)(\frac hk)=
	\partial_v \Lagr ({\frac hk,\zeta(\frac hk),\dot\zeta(\frac hk\sm-)}) -
	\partial_v \Lagr ({\frac hk,\zeta(\frac hk),\dot\zeta(\frac hk\sm+)}),\s
	\forall h\in\gbra{0,...,k-1}. 
\]
Analogously, for each   $\sigma\in\Tan_{\zeta\iter n}\LL n_k$ we have
\begin{align*}
\diff\act n_k(\zeta\iter n)\,\sigma=
\frac 1n
\sum_{l=0}^{n-1} \sum_{h=0}^{k-1}  
	\abra{
	{\textstyle
	\partial_v \Lagr ({\frac hk,\zeta(\frac hk),\dot\zeta(\frac hk\sm-)})-
	\partial_v \Lagr ({\frac hk,\zeta(\frac hk),\dot\zeta(\frac hk\sm+)})
	,
	\sigma(\frac {h}k + l)
	}
	},
\end{align*}
and, from this computation, equation~\eqref{e:commGradIter} readily follows. 
\end{proof}
\end{prop}

\section{Homological vanishing under iteration}\label{s:HomVanSec}

In this section we prove that the elements of the homotopy and homology groups of pairs of sublevels of the action functional (associated to a convex quadratic-growth Lagrangian)  are killed by the $n^{\mathrm{th}}$-iteration map, for $n$ that is a sufficiently big power of any given positive integer. This result was proved in the particular case of the geodesics action functional by Bangert and Klingenberg \cite[theorem~2]{b:BK}, and then extended by Long \cite[section~5]{b:Lo} to more general  Lagrangian systems. Long's proof relies on an ad hoc homology theory, which he calls \emph{Finite Energy Homology}, in order to deal with technical issues concerning the regularity of the involved singular simplices. Here, we provide a proof that makes use of the standard singular homology theory.

Consider a convex quadratic-growth Lagrangian $\Lagr:\T\times\Tan M\to\R$  with  associated action  $\act\tau$, $\tau\in\N$. As usual, let us put $\tau=1$ in order to simplify the notation. We denote by $\wc\subset\W(\T;M)$ and, more generally, by $\WC n\subset\W(\TT n;M)$, the connected com\-po\-nent of \textbf{contractible} loops. Notice that the iteration map restricts to a map \[\itmap n:\wc\hookrightarrow\WC n.\]
From now on we will  implicitly consider the action functionals $\Act$ and $\act n$ restricted to $\wc$ and $\WC n$ respectively. In particular, all the action sublevels will be contained in this latter sets.

\begin{thm}[Homological vanishing]\label{t:vanish}
Let $c_1<c_2\leq\infty$, where the  sublevel $(\Act)_{c_1}$ is not empty, and let $[\mu]\in\Hom_*((\Act)_{c_2}, (\Act)_{c_1})$. Then, for any integer $p\geq 2$, there exists  $\bar n=\bar n(\Lagr,[\mu],p)\in\N$ that is a power of $p$ such that $\itmap {\bar n}_*[\mu]=0$ in $\Hom_*( (\act{\bar n})_{c_2}, (\act{\bar n})_{c_1})$.
\end{thm}
Since $\itmap n\circ\itmap m=\itmap{nm}$ for each $n,m\in\N$, the assertion of this theorem can be rephrased as follows: for each $n\in\N$ that is a sufficiently big power of the given $p\in\N$, we have $\itmap {n}_*[\mu]=0$ in $\Hom_*( (\act{n})_{c_2}, (\act{n})_{c_1})$.

The proof of  theorem~\ref{t:vanish} is based on a homotopic technique  that is essentially  due to Bangert (see \cite[section 3]{b:Ba} or \cite[theorem 1]{b:BK}). We recall that a homotopy $F:[0,1]\times (X,U)\to (Y,W)$ is said  \textbf{relative} $U$ when $F(t,x)=F(0,x)$ for all $(t,x)\in[0,1]\times U$. For each $q\in\N$, we denote by $\Delta^q$ the standard $q$-simplex in $\R^q$.

\begin{lem}\label{l:vanish}
Let $c_1<c_2\leq\infty$ and $\sigma:(\Delta^q,\partial\Delta^q)\to((\Act)_{c_2}, (\Act)_{c_1})$ be a  singular simplex, i.e.\ $[\sigma]\in\pi_q((\Act)_{c_2}, (\Act)_{c_1})$. Then, there exists $\bar n=\bar n(\Lagr,\sigma)\in\N$  and, for every integer $n\geq\bar n$, a homotopy
\[\Ban\sigma n: [0,1] \times (\Delta^q,\partial\Delta^q)\to(\sub{n}{c_2}, \sub{n}{c_1})
\s\s\mbox{relative $\partial\Delta^q$},\]
which we call \textnormal{\textbf{Bangert homotopy}}, such that $\Ban \sigma  n(0,\cdot)=\sigma\iter n:=\itmap n\circ\sigma$ and $\Ban\sigma n(1,\Delta^q)\subset \sub{n}{ c_1}$. In particular, $\itmap n_*[\sigma]=0$ in $\pi_q(\sub{n}{c_2}, \sub{n}{c_1})$.
\begin{proof}
First of all, let us introduce some notation. For each path  $\alpha:[x_0,x_1]\to M$, we denote by  $\overline{\alpha}:[x_0,x_1]\to M$ the inverse path 
\[\overline\alpha(x)=\alpha(x_0+x_1-x),\s\s\forall x\in[x_0,x_1].\] 
If we consider a second path $\beta:[x_0',x_1']\to M$ with $\alpha(x_1)=\beta(x_0')$, we denote by $\alpha\bullet\beta:[x_0,x_1+x_1'-x_0']\to M$ the concatenation of   $\alpha$ and $\beta$, namely
\[
\alpha\bullet\beta(x)=
\left\{
  \begin{array}{lcl}
    \alpha(x)         & & x\in[x_0,x_1], \\ 
    \beta(x-x_1+x_0')  & & x\in[x_1,x_1+x_1'-x_0']. \\ 
  \end{array}
\right.
\]

Now, consider a continuous  map $\theta:[x_0,x_1]\to\W(\T;M)$, where $[x_0,x_1]\subset\R$. For each  $n\in\N$, we define $\theta\iter n :=\itmap n\circ\theta:[x_0,x_1]\to\W(\TT n;M)$. Now, we want to build another continuous map $\theta\siter n:[x_0,x_1]\to\W(\TT n;M)$ as explained in the following. To begin with, let us denote by $\ev:\W(\T;M)\to M$  the \textbf{evaluation map}, given by
\[ \ev(\zeta)=\zeta(0),\s\s\forall \zeta\in\W(\T;M). \]
This   map is smooth, which implies that the initial point curve $\ev\circ\theta:[x_0,x_1]\to M$ is  (uniformly) continuous. In particular, there exists a constant $\rho=\rho(\theta)>0$ such that, for each $x,x'\in[x_0,x_1]$ with $|x-x'|\leq\rho$, we have that $\dist(\ev\circ\theta(x),\ev\circ\theta(x'))$ is less than the injectivity radius of $M$. Here,   we have denoted by \virg{$\dist$} the Riemannian distance on $M$  (with respect to its fixed Riemannian metric). Now, for each $x,x'\in[x_0,x_1]$ with $0\leq x'-x\leq\rho$, we define the \textbf{horizontal geodesic}\index{geodesic!horizontal --} $\theta_x^{x'}:[x,x']\to M$ as the shortest geodesic that connects the points $\ev\circ\theta(x)$ and $\ev\circ\theta(x')$. Notice that, by proposition~\ref{p:fathi2}, this geodesic depends  smoothly on its endpoints. Then, let $J\in\N$ be such that
$x_0+J\rho\leq x_1\leq x_0+(J+1)\rho$. For each $x\in[x_0,x_1]$ we further choose $j\in\N$ such that $x_0+  j\rho\leq x\leq x_0+(j+1)\rho$, and we define the \textbf{horizontal broken geodesics}  $\theta_{x_0}^{x}:[x_0,x]\to M$ and $\theta_{x}^{x_1}:[x,x_1]\to M$ by
\begin{align*}
&\theta_{x_0}^x:= 
\theta_{x_0}^{x_0+\rho}
\bullet
\theta_{x_0+\rho}^{x_0+2\rho}
\bullet
...
\bullet
\theta_{x_0+j\rho}^{x},\\
&\theta_{x}^{x_1}:= 
\theta_{x}^{x_0+(j+1)\rho}
\bullet
\theta_{x_0+(j+1)\rho}^{x_0+(j+2)\rho}
\bullet
...
\bullet
\theta_{x_0+J\rho}^{x_1}.
\end{align*}
We define a preliminary map $\tilde\theta\siter n:[x_0,x_1]\to\W(\TT n;M)$ in the following way. For each   $j\in\gbra{1,..., n-2}$ and $y\in[0,\sfrac{x_1-x_0}{ n}]$ we put
\begin{align*}
    \tilde\theta\siter  n (x_0+y)  :=\,&
    	\theta\iter{ n -1}(x_0)  \bullet  
    	\theta_{x_0}^{x_0+ n y}    		\bullet 
    	\theta(x_0+ n y)  
    	\bullet
		\overline{ \theta_{x_0}^{x_0+ n y} },\\
    \tilde\theta\siter  n  \cbra{ x_0 + \sfrac j  n  		(x_1-x_0) + y } :=\,&
    	\theta\iter{ n -j-1}(x_0)  \bullet  
    	 \theta_{x_0}^{x_0+ n y}   \bullet 
    	\theta(x_0+ n y)\\   
			&\bullet \theta_{x_0+ n y}^{x_1}  
			\bullet 
	 	\theta\iter j (x_1)  \bullet  
    	\overline{ \theta_{x_0}^{x_1}}, \\
    \tilde\theta\siter  n  \cbra{ x_0 + 
    	\sfrac{ n-1} n (x_1 - x_0)  + y }  :=\,& 
    	\theta(x_0 + n y)  \bullet  
        \theta_{x_0+ n y}^{x_1}  	\bullet 
	 	\theta\iter { n -1}(x_1)  
		\bullet  
    	\overline{ \theta_{x_0+ n y}^{x_1} }  .
\end{align*}
For each $x\in[x_0,x_1]$, we reparametrize the loop $\tilde\theta\siter  n(x)$ as follows: in the above formulas, each fixed part $\theta(x_0)$ and $\theta(x_1)$ spends the original time $1$, while the moving parts $\theta(x_0+ n y)$ and the pieces of horizontal broken geodesics share the remaining time $1$ proportionally to their original parametrizations. We define  \[\theta\siter  n:[x_0,x_1]\to \W(\TT n;M)\] as the obtained continuous  path in the loop space (see the example in figure~\ref{f:bang}(a)). 

\frag[n]{a4(0)}{$\theta\siter 4(0)$}%
\frag[n]{a4(1/4+s)}{$\theta\siter 4(\sfrac14+y)$}%
\frag[n]{a4(1)}{$\theta\siter 4(1)$}%
\frag[n]{a4(s)}{$\theta\siter 4(y)$}%
\frag[n]{0<s<1/4}{$0\leq y\leq 1/4$}%
\frag[ss]{ev.a}{}%
\frag[ss]{ev.ab}{}%
\frag[ss]{a(0)4}{$\theta\iter4(0)$}%
\frag[ss]{a(0)3}{$\theta\iter3(0)$}%
\frag[ss]{a(0)2}{$\theta\iter2(0)$}%
\frag[ss]{a(4s)}{$\theta(4y)$}%
\frag[ss]{a(1)}{$\theta(1)$}%
\frag[ss]{a(1)4}{$\theta\iter4(1)$}%
\frag[n]{b(s)}{$\widehat\theta(y)$}%
\frag[n]{b(k/4+s)}{$\widehat\theta(\sfrac j4+y)$}%
\frag[n]{b(3/4+s)}{$\widehat\theta(\sfrac34+y)$}%
\frag[n]{0<s<1/4}{$0\leq y < 1/4$}%
\frag[n]{0<s<1/4kin12}{$0\leq y < 1/4$, $j\in\gbra{1,2}$}%
\begin{figure}[p]
\begin{center}
\includegraphics{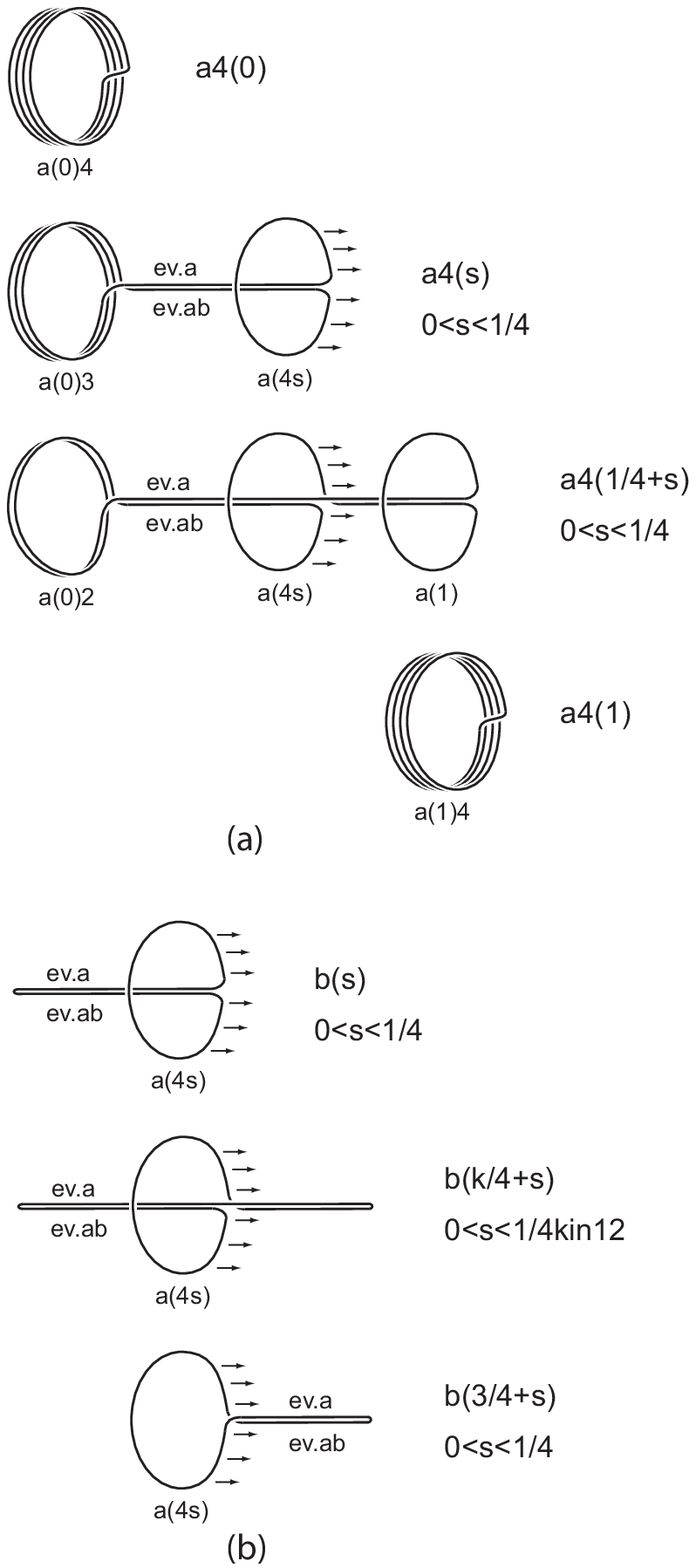}
\end{center}
\hcaption{\textbf{(a)} Description of  $\theta\siter 4:[0,1]\to\W(\T\iter{4};M)$, obtained from a map $\theta:[0,1]\to\W(\T;M)$. Here, for simplicity, we are assuming that the diameter of $\theta([x_0,x_1])$ is less than the injectivity radius of $M$, so that the horizontal geodesics are not broken. The arrows show the direction in which the loop $\theta(4y)$ is pulled as $y$ grows. \textbf{(b)} Description of the map of pulling loops $\widehat\theta:[0,1]\to\W(\T;M)$.
}
\label{f:bang}
\end{figure}

For each $x\in[x_0,x_1]$, we define the \textbf{pulling loop}  $\widehat\theta(x):\T\to M$  as the loop obtained erasing from the formula of $\tilde\theta\siter  n(x)$ the fixed parts  $\theta(x_0)$ and $\theta(x_1)$  and reparametrizing on $[0,1]$ (see the example in  figure~\ref{f:bang}(b)). Notice  that $\widehat\theta$ is independent of the integer $ n\in\N$ and, for each $x\in\N$, the action $\Act(\widehat\theta(x))$ is finite and depends continuously on $x$. In particular we obtain a finite constant
\[C(\theta):=\max_{x\in[x_0,x_1]} \gbra{\Act(\widehat\theta(x))}=
\max_{x\in[x_0,x_1]} \gbra{\int_0^1 \Lagr\!\cbra{t,\widehat\theta(x)(t),\der{}{t}{\widehat\theta}(x)(t)}\diff t} < \infty, \]
and, for each $ n\in\N$, the  estimate
\begin{equation}\label{e:est_act_siter}
\begin{split}
\s\s\act n(\theta\siter n(x))&\leq
\frac1 n \qbra{( n-1)\max\gbra{\Act(\theta(x_0)),\Act(\theta(x_1))} + \Act(\widehat\theta(x)) } \\
&\leq
\max\gbra{\Act(\theta(x_0)),\Act(\theta(x_1))} + \frac{C(\theta)} n.
\end{split}
\end{equation}

\frag[s]{Ds}{$s\Delta^q \subset \Delta^q$}%
\frag[s]{L}{$\Line$}%
\frag[s]{Lo}{$\Line^\bot$}%
\frag[ss]{x}{$\bm y$}%
\frag[ss]{a}{$x_0(\bm y, s)$}%
\frag[ss]{b}{$x_1(\bm y, s)$}%
\begin{figure}
\begin{center}
\includegraphics{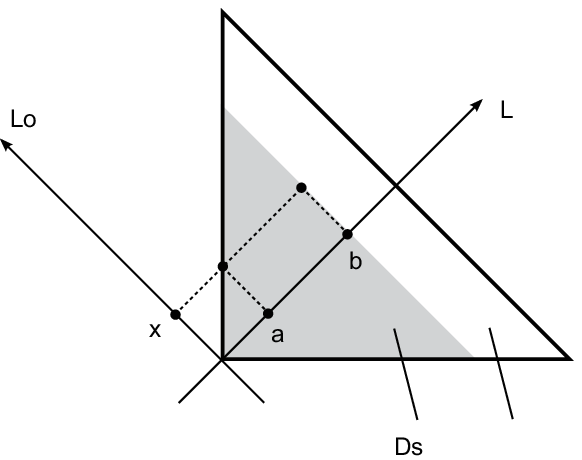}
\end{center}
\ccaption{}
\label{f:simplex}
\end{figure}

Now, let $\Line\subseteq\R^q$ be the straight line passing through the origin and the barycenter of the standard $q$-simplex $\Delta^q\subset\R^q$. According to the orthogonal decomposition $\R^q=\Line^\bot\oplus\Line$,  we can write the elements of $\Delta^q$ as $\bm z=(\bm y,x)\in\Line^\bot\oplus\Line$.
For each $s\in[0,1]$ we denote by $s\Delta^q$ the rescaled $q$-simplex, given by $\gbra{s\bm z\,|\,\bm z\in\Delta^q}$.
Varying $s$ from 1 to 0 we obtain a deformation retraction of $\Delta^q$ onto the origin of $\R^q$. For each $(\bm y,x)\in s\Delta^q$, we denote by $[x_0(\bm y, s),x_1(\bm y,s)]\subseteq\Line$ the maximum interval such that $(\bm y, x')$ belongs to $s\Delta^q$ for all $x'\in[x_0(\bm y, s),x_1(\bm y,s)]$ (see figure~\ref{f:simplex}).

Consider the $q$-singular simplex $\sigma$ of the statement. For each $ n\in\N$, we define the homotopy $\Ban\sigma n:[0,1]\times \Delta^q\to\W(\TT n;M)$ 
by
\begin{align*}
\Ban\sigma n(s,\bm z):=
\Biggl\{
  \begin{array}{lccl}
    \Bigl. \bigl(\sigma(\bm y,\cdot)|_{[x_0(\bm y, s),x_1(\bm y,s)]}\bigr)\siter n(x) &&&
    \bm z=(\bm y,x)\in s\Delta^q,\\ 
    \Bigl. \sigma\iter n(\bm z) & && \bm z\not\in s\Delta^q,\\  
  \end{array}
\end{align*}
for each $(s,\bm z)\in[0,1]\times\Delta^q$. This  homotopy $\Ban\sigma n$ is relative $\partial\Delta^q$, for
\[\Ban\sigma n(s,\bm z)=\sigma\iter n(\bm z),\s\s\forall(s,\bm z)\in[0,1]\times\partial\Delta^q,\]
and clearly $\Ban\sigma n(0,\cdot)=\sigma\iter n$. Take    $\epsilon >0$ such that
\[
\max_{\bm z\in\Delta^q} \Act(\sigma(\bm z))\leq c_2-\epsilon,\s\s
\max_{\bm z\in\partial\Delta^q} \Act(\sigma(\bm z))\leq c_1-\epsilon.
\]
For each $s\in[0,1]$, $ n\in\N$ and $\bm z= (\bm y,x)\in s\Delta^q$, by  the  estimate in~\eqref{e:est_act_siter} we have
\begin{equation*} 
\small
\act n(\Ban\sigma n(s,\bm z))  \leq 
\max\gbra{\Act(\sigma(x_0(\bm y,s))),\Act(\sigma(x_1(\bm y,s)))} + \frac{ C(\sigma(\bm y,\cdot)|_{[x_0(\bm y, s),x_1(\bm y,s)]}) } n,
\end{equation*}
while, for each $\bm z\in\Delta^q\setminus s\Delta^q$, we have
\begin{align*} 
\act n(\Ban\sigma n(s,\bm z))= \Act(\sigma(\bm z)).
\end{align*}
In particular, there exists a finite constant 
\[C(\sigma):= \max\gbra{ C(\sigma(\bm y,\cdot)|_{[x_0(\bm y, s),x_1(\bm y,s)]}) \,|\,s\in[0,1],\ (\bm y,x)\in s\Delta^q} \]
such that, for each $ n\in\N$ and $(s,\bm z)\in[0,1]\times\Delta^q$, we have
\begin{align*} 
\act n(\Ban\sigma n(s,\bm z))
\leq
\max_{\bm w\in\Delta^q}\gbra{\Act(\sigma(\bm w)} + \frac{C(\sigma)} n \leq c_2-\epsilon + \frac{C(\sigma)} n,\\
\act n(\Ban\sigma n(1,\bm z))
\leq
\max_{\bm w\in\partial\Delta^q}\gbra{\Act(\sigma(\bm w)} + \frac{C(\sigma)} n \leq c_1-\epsilon + \frac{C(\sigma)} n.
\end{align*}
These estimates prove that, for $n$ sufficiently big, the homotopy $\Ban \sigma n$ satisfies the properties stated in the lemma.
\end{proof}
\end{lem}

\begin{rem}\label{r:BangertBd}
In section~\ref{s:modif_homVan} we will need the following observation. Assume that the singular simplex $\sigma$ of lemma~\ref{l:vanish} has $W^{1,\infty}$-bounded image, i.e.\ there exists a real $\bar R'$ such that 
\[ \sup_{\bm z\in\Delta^q}\ \esssup_{t\in\T} \gbra{\abs{\der{}{t}\, \sigma(\bm z)(t) }_{\sigma(\bm z)(t)}  } \leq \bar R' .\]
Then the Bangert homotopies $\Ban{\sigma}n$ have $W^{1,\infty}$-bounded image as well, and this bound is uniform in $n\geq\bar n(\Lagr,\sigma)$. In other words,  there exists a  real $\bar R\geq\bar R'$ such that, for every integer $n\geq\bar n(\Lagr,\sigma)$,   we have
\[  
\sup_{(s,\bm z)\in[0,1]\times\Delta^q}
\esssup_{t\in\TT{n}} 
\gbra{ 
	\abs{
	\der{}{t}\, 
	\Ban{\sigma}n(s,\bm z)(t) 
	}_{\Ban{\sigma} n(s,\bm z)(t)}
} 
\leq 
\bar R .\]
\end{rem}

\begin{proof}[Proof of theorem~\ref{t:vanish}]
We denote by $\Sigma(\mu)$ the set of singular simplices in $\mu$ together with all their faces, and by $\K\subset\N$ the set of nonnegative integer powers of $p$, i.e. $\K=\gbra{p^n\,|\,n\in\N\cup\gbra0}$. The idea of the proof is to apply lemma~\ref{l:vanish} successively to all the elements of $\Sigma(\mu)$. More precisely, for each singular simplex $\sigma:\Delta^q\to(\Act)_{c_2}$ that belongs to $\Sigma(\mu)$, we will find $\bar n=\bar n(\Lagr,\sigma,p)\in\K$ and a homotopy
\[ P_{\sigma}\iter{\bar n}:[0,1]\times \Delta^{q}\to\sub{\bar n} {c_2}, \]
such that
\begin{itemize}
\item[(i)]  $P_{\sigma}\iter{\bar n}(0,\cdot)=\sigma\iter{\bar n}$,
\item[(ii)]  $P_{\sigma}\iter{\bar n}(1,\Delta^q) \subset (\act{\bar n})_{c_1}$,
\item[(iii)] if $\sigma(\Delta^q)\subset (\Act)_{c_1}$, then $P_{\sigma}\iter{\bar n}(s,\cdot)=\sigma\iter{\bar n}$ for each $s\in[0,1]$,
\item[(iv)] $P_{\sigma\circ F_i}\iter{\bar n}=P_{\sigma}\iter{\bar n}(\cdot,F_i(\cdot))$ for each $i=0,...,q$, where $F_i:\Delta^{q-1}\to\Delta^q$ is the standard affine map onto the $i^{\mathrm{th}}$ face of $\Delta^q$.
\end{itemize}
For each $n\in\K$ greater  than $\bar n$, we define a homotopy $P_{\sigma}\iter{n}:[0,1]\times \Delta^{q}\to\sub{n} {c_2}$ by $P_{\sigma}\iter{n}=\itmap{n/\bar n}\circ P_{\sigma}\iter{\bar n}$. This homotopy satisfies the analogous properties (i),...,(iv) in period $n$. Notice that property (iv) implicitly requires that $\bar n(\Lagr,\sigma,p)\geq\bar n(\Lagr,\sigma\circ F_i,p)$ for each $i=0,...,q$.

Now, assume that such homotopies exist and put 
\[\bar n=\bar n(\Lagr,[\mu],p):=\max\gbra{\bar n(\Lagr,\sigma,p)\,|\,\sigma\in\Sigma(\mu)}\in\K.\] 
Then,  we have a family of homotopies $\{P_{\sigma}\iter{\bar n}\,|\,\sigma\in\Sigma(\mu)\}$ satisfying the above properties. Therefore, a classical result in algebraic topology (basically, a variation of the homotopic invariance of singular homology, see \cite[lemma 1]{b:BK}) implies that $\itmap {\bar n}_*[\mu]=0$ in $\Hom_*(\sub{\bar n} {c_2},\sub{\bar n} {c_1})$.

In order to conclude the proof, we only need to build the above homotopies.  We do it  inductively on the degree of the relative cycle  $\mu$. If $\mu$ is a $0$-relative cycle, then  $\Sigma(\mu)$ is simply a finite  set of contractible loops that is contained in $(\Act)_{c_2}$. Let $\gamma\in\Sigma(\mu)$ be one of these loops. If $\gamma\in(\Act)_{c_1}$  we simply  set $\bar n=\bar n(\Lagr,\gamma,p):=1$ and $P_\gamma\iter{\bar n}(s):=\gamma$ for each $s\in[0,1]$. If $\gamma\not\in(\Act)_{c_1}$, since we are assuming that $(\Act)_{c_1}$ is non-empty, we can find a  continuous path  \[\Gamma:([0,1],\gbra{0,1})\to(\wc,(\Act)_{c_2})\] such that $\Gamma(0)=\gamma$ and $\Gamma(1)\in(\Act)_{c_1}$.  By  lemma~\ref{l:vanish}, for every integer $n\geq\bar n(\Lagr,\Gamma)$ there exists a   Bangert homotopy 
\[  B_{\Gamma}\iter n: [0,1]\times ([0,1],\gbra{0,1})\to(\WC n,(\act n)_{c_2})  \s\s\mbox{relative   $\gbra{0,1}$} \]
such that $B_{\Gamma}\iter n(0,\cdot)=\Gamma\iter n$ and $B_{\Gamma}\iter n(1,[0,1]) \subset (\act n)_{c_2}$. Then, we set 
\[\bar n=\bar n(\Lagr,\gamma,p):=\min\{n\in\K\,|\, n\geq \bar n(\Lagr,\Gamma) \}\] 
and we define the map $P_\gamma\iter{\bar n}:[0,1] \to (\act{\bar n})_{c_2}$ as $P_\gamma\iter{\bar n}:=\Ban\gamma{\bar n}(1,\cdot)$, see figure~\ref{f:BangDim0}. 

\frag[n]{Bgm}{$\Ban{\Gamma}{\bar n}$}%
\frag{D}{$\Gamma\iter{\bar n}$}%
\frag{B}{$P_{\gamma}\iter{\bar n}$}%
\frag{Z}{$\equiv \gamma\iter{\bar n}$}%
\frag{C}{$\equiv \Gamma\iter{\bar n}(1)$}%

\begin{figure}
\begin{center}
\includegraphics[scale=1]{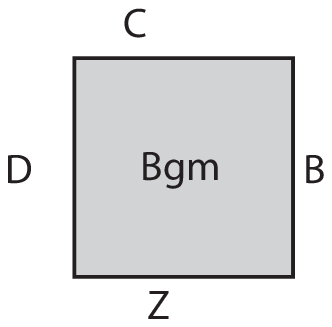}
\end{center}
\ccaption{}
\label{f:BangDim0}
\end{figure}

In case $\mu$ is a $q$-relative cycle, with $q\geq 1$, we can apply the inductive hypothesis: for every  nonnegative integer $j<q$ and for each $j$-singular simplex $\nu\in\Sigma(\mu)$ we obtain  $\bar n(\Lagr,\nu,p)\in\K$ and, for every  $n\in\K$ greater or equal than $\bar n(\Lagr,\nu,p)$, a homotopy $P_\nu\iter n$ satisfying the above properties (i),...,(iv). Now, consider a $q$-singular simplex $\sigma\in\Sigma(\mu)$. If $\sigma(\Delta^q)\subset(\Act)_{c_1}$ we simply  set $\bar n=\bar n(\Lagr,\sigma,p):=1$ and $P_\sigma\iter{\bar n}(s,\cdot):=\sigma$ for each $s\in[0,1]$. Hence, let us assume that $\sigma(\Delta^q)\not\subset(\Act)_{c_1}$.  We denote by $\bar n'=\bar n'(\Lagr,\sigma,p)$ the maximum of the $\bar n(\Lagr,\nu,p)$'s for all the proper faces $\nu$ of $\sigma$. For each   $n\in\K$ greater or equal than $\bar n'$,  every proper face $\nu$ of $\sigma$ has an associated  homotopy   
$P_{\nu}\iter n: [0,1]\times \Delta^{q-1} \to   \sub{n}{c_2}$.
For technical reasons, let us assume that $P_{\nu}\iter {n}(s,\cdot)=P_{\nu}\iter n(\sfrac12,\cdot)$ for each  $s\in[\sfrac12,1]$. Patching together the homotopies of the proper faces of $\sigma$, we obtain
\[  P_\sigma\iter n: ([0,\sfrac12]\times \partial\Delta^q) \cup (\gbra0\times \Delta^q) \to \sub{n}{c_2},\s\s\forall n\in\K,\ n\geq\bar n',\]
such that $P_\sigma\iter n(0,\cdot)=\sigma\iter n$ and $P_\sigma\iter n(\cdot,F_i(\cdot))=P_{\sigma\circ F_i}\iter n$ for each $i=0,...,q$. By retracting $[0,\sfrac12]\times \Delta^q$ onto $([0,\sfrac12]\times \partial\Delta^q) \cup (\gbra0\times \Delta^q)$ we can extend  the homotopy $P_\sigma\iter n$ to the whole $[0,\sfrac12]\times \Delta^q$, obtaining
\begin{align}\label{e:P_homot_1}  
P_\sigma\iter n: [0,\sfrac12]\times \Delta^q \to \sub{n}{c_2}, \s\s\forall n\in\K,\ n\geq\bar n'.
\end{align}
Notice that $P_\sigma\iter {\bar n'}(\sfrac12,\cdot)$ is a singular simplex of the form
\[
P_\sigma\iter {\bar n'}(\sfrac12,\cdot): (\Delta^q,\partial\Delta^q) \to (\sub{\bar n'}{c_2},\sub{\bar n'}{c_1} ).
\]
Let us briefly denote this singular simplex by $\tilde\sigma$. By  lemma~\ref{l:vanish}, there exists  $\bar n''\in\K$ greater or equal than $\bar n(\Lagr,\tilde\sigma)$  and a Bangert homotopy
\begin{align*} 
B_{\tilde\sigma}\iter {\bar n''}: [0,1]\times (\Delta^q, \partial\Delta^q )\to (\sub{\bar n'\bar n''}{c_2},\sub{\bar n'\bar n''}{c_1})\s\s\mbox{relative }\partial \Delta^q,
\end{align*}
such that $B_{\tilde\sigma}\iter {\bar n''}(0,\cdot)=\tilde\sigma\iter {\bar n''}=P_\sigma\iter {\bar n'\bar n''}(\sfrac12,\cdot)$ and  $B_{\tilde\sigma}\iter {\bar n''}(1,\Delta^q)\subset(\act {\bar n'\bar n''})_{c_1}$. 
Finally, we set $\bar n=\bar n(\Lagr,\sigma,p):=\bar n'\bar n''$ and we build the homotopy \[P_\sigma\iter{\bar n}: [0,1]\times \Delta^q \to \sub{\bar n}{c_2}\] 
extending the one in~\eqref{e:P_homot_1} by $P_\sigma\iter{\bar n}(s,\cdot):= B_{\tilde\sigma}\iter {\bar n''}(2s-1,\cdot)$ for each $s\in[\sfrac12,1]$. 
\end{proof}

\section{Convex quadratic modifications}\label{s:conv_quad_modifications}

Throughout this section, $\Lagr:\T\times\Tan M\to\R$ will be a Tonelli Lagrangian with associated action $\act\tau$, $\tau\in\N$. We want to show how several techniques from the $\W$ Morse Theory of the action functional  still apply in the Tonelli case. This will require some work, since the  Tonelli action functional $\act\tau$ is not even continuous on the $\W$ loop space.

First of all, notice that the Tonelli assumptions and the compactness of $M$ imply that $\Lagr$ is \textbf{uniformly} fiberwise superlinear (namely, the limit in \textbf{(T2)} is uniform in $(t,q)\in\T\times M$). In fact, if $\Ham$ is the   Hamiltonian that is Legendre dual to $\Lagr$ (see section~\ref{s:CZLetc}) and $k>0$, for each $(t,q,v)\in\T\times\Tan M$ we have
\begin{align*}
\Lagr(t,q,v) & \geq 
\max_{|p|_q\leq k} \gbra{ p(v) - \Ham(t,q,p) } 
\geq
\max_{|p|_q\leq k} \gbra{ p(v)} - \max_{|p|_q\leq k} \gbra{ \Ham(t,q,p) } \\
& \geq 
k\,|v|_q - \max \gbra{ \Ham(t',q',p')\,|\,(t',q',p')\in\T\times\Tan^*M,\ |p'|_{q'}\leq k }.
\end{align*}
In particular, if we put $C(\Lagr):=\max \gbra{ \Ham(t,q,p)\,|\,(t,q,p)\in\T\times\Tan^*M,\ |p|_{q}\leq 1 }$, we get
\[ \Lagr(t,q,v)\geq |v|_q - C(\Lagr),\s\s\forall (t,q,v)\in\T\times\Tan M. \]
Following Abbondandolo and Figalli \cite[section~5]{b:AbFig}, for each real $R>0$ we say that a  convex quadratic-growth Lagrangian  $\Lagr_R:\T\times\Tan M\to\R$ is a \textbf{convex quadratic $R$-modification}  (or simply an \textbf{$R$-modification}) of the Tonelli Lagrangian $\Lagr$  when:
\begin{itemize}
  \item[\textbf{(M1)}] $\Lagr_R(t,q,v)=\Lagr(t,q,v)$ for each $(t,q,v)\in\T\times\Tan M$ with $\abs v _q \leq R$,
  \item[\textbf{(M2)}] $\Lagr_R(t,q,v)\geq \abs v _q - C(\Lagr)$ for each $(t,q,v)\in\T\times\Tan M$.
\end{itemize}
It is always possible to build a convex quadratic  modification of a given Tonelli Lagrangian (see \cite[page~637]{b:AbFig}). For each $R>0$, we will denote by $\Lagr_R$ an arbitrary $R$-modification of the Tonelli Lagrangian $\Lagr$ with associated action $\act \tau_R:\W(\TT \tau;M)\to\R$ (for each $\tau\in\N$), i.e.
\begin{align*}
\act \tau_R(\zeta)&=\frac1 \tau\int_0^\tau \Lagr_R(t,\zeta(t),\dot\zeta(t))\,\diff t,&\forall\zeta\in\W(\TT \tau;M).
\end{align*}
As before, we will simply write $\Act_R$ for $\act 1_R$. Notice that, if $\gamma:\TT\tau\to M$ is a smooth $\tau$-periodic solution of the Euler-Lagrange system of $\Lagr$ and $R>\max\{|\dot\gamma(t)|_{\gamma(t)}\,|\,t\in\TT\tau\}$, then $\gamma$ is a critical point of $\act\tau_R$. Moreover, the Hessian of $\act\tau_R$ at $\gamma$ depends only on $\Lagr$, since $\Lagr$ and $\Lagr_R$ coincide along the lifted curve $(\gamma,\dot\gamma):\TT\tau\to \Tan M$. In particular, the Morse index and nullity pair $(\iota(\gamma),\nu(\gamma))$ of $\act\tau_R$ at $\gamma$ is independent of the chosen $R$, and in fact coincides with the Conley-Zehnder-Long index pair of $\gamma$.

One of the  important features  of convex quadratic modifications  is given by the following a priori estimate, that is due to Abbondandolo and Figalli (see \cite[lemma~5.2]{b:AbFig} for a proof).

\begin{lem}\label{l:apriori}
For each $\aaa>0$ and $\ttt\in\N$, there exists  $\rrr=\rrr(\aaa,\ttt)>0$ such that, for any $R$-modification $\Lagr_R$ of $\Lagr$ with $R>\rrr$ and for any  $\tau\in\gbra{1,...,\ttt}$, the following holds:  if $\gamma$ is a critical point of $\act \tau_R$  such that  $\act \tau_R(\gamma)\leq \aaa$, then  $\max\{|\dot\gamma(t)|_{\gamma(t)} \,|\, t\in\TT\tau \}\leq \rrr$. 
In particular, $\gamma$ is a $\tau$-periodic solution of the Euler-Lagrange system of $\Lagr$, and $\act \tau(\gamma)=\act \tau_R(\gamma)$.
\hfill$\qed$
\end{lem}

\subsection{Convex quadratic modifications and local homology}\label{s:modif&loc_hom}

Let $\gamma$ be a contractible integer  periodic solution of the Euler-Lagrange system associated to the Tonelli Lagrangian $\Lagr$. In order to simplify the notation, we  assume that the period of $\gamma$ is $1$, so that it is a map of the form $\gamma:\T\to M$. We fix a real constant $U>0$ such that
\begin{align}\label{e:choose_U} 
U > \max_{t\in\T} \gbra{ |\dot\gamma(t)|_{\gamma(t)} },
\end{align}
and we consider a $U$-modification $\Lagr_U$ of $\Lagr$. Notice that, by~\eqref{e:choose_U}, $\gamma$ is a critical point of $\Act_U$ with $\Act_U(\gamma)=\Act(\gamma)$.

Now, for each $k\in\N$ sufficiently big, we consider the $k$-broken Euler-Lagrange loop space $\Lambda_k=\Lambda_{k,\Lagr_U}$ associated to $\Lagr_U$ (see section~\ref{s:discr_act_funct}), and we denote by $\Act_{U,k}$ its discrete action, i.e. $\Act_{U,k}=\Act_U|_{\Lambda_k}$. We recall that $\Lambda_k$ is a finite dimensional submanifold of $\W(\T;M)$ and, in particular, its topology coincides with the one induced as a subspace of $W^{1,\infty}(\T;M)$. Therefore, we can define an open set $\catU_k\subset\Lambda_k$ by
\[ \catU_k:=\gbra{ \zeta\in\Lambda_k\ \biggl|\ \max_{t\in\T}\{|\dot\zeta(t)|_{\zeta(t)}\}<U   }. \]
Notice that, for $k\in\N$ sufficiently big, $\gamma$ belongs to $\catU_k$. Moreover, the action $\Act_U$ coincides with  the Tonelli action $\Act$ on the open set $\catU_k$. This allows us to define the \textbf{discrete Tonelli action} $\Act_k:\catU_k\to\R$ by
\[ \Act_k:=\Act|_{\catU_k}=\Act_{U}|_{\catU_k}=\Act_{U,k}|_{\catU_k}. \]
Since $\Act_{U,k}$ is smooth, we readily obtain that the discrete Tonelli action $\Act_k$ is smooth as well. Moreover, the germ of $\Act_k$ at $\gamma$ turns out to be independent of the chosen $U$ and of the chosen $U$-modification $\Lagr_U$ of $\Lagr$, as stated by the following.

\begin{lem}\label{l:germ}
Consider a real constant  $R\geq U$,  an $R$-modification   $\Lagr_R$ of $\Lagr$ with associated discrete action $\Act_{R,k}$ and a sufficiently big  $k\in\N$  so that both $\Act_{R,k}$ and $\Act_{U,k}$ are defined. Then,  there exists $\catV_k\subset \catU_k$ that is an open subset of both $\Lambda_{k,\Lagr_U}$ and $\Lambda_{k,\Lagr_R}$ and that contains $\gamma$. In particular $\Act_k|_{\catV_k} =\Act_{U,k}|_{\catV_k} =\Act_{R,k}|_{\catV_k}$.   
\begin{proof}
The Lagrangian functions $\Lagr_U$ and $\Lagr_R$ coincide on a neighborhood of the support of the lifted curve $(\gamma,\dot\gamma):\T\to\Tan M$. Therefore the $k$-broken Euler-Lagrange loops of $\Lagr_U$ and $\Lagr_R$ that are close to $\gamma$  are the same, and the claim follows.
\end{proof}
\end{lem}

Now, let $c=\Act_k(\gamma)=\Act(\gamma)$. By   corollary~\ref{c:hom_equiv_loc_hom} and the excision property, for each $k\geq\bar k(\Lagr_U,c)$, the inclusion
\[
\iota_k:
( (\Act_k)_c\cup\gbra\gamma , (\Act_k)_c )
\hookrightarrow
( (\Act_U)_c\cup\gbra\gamma , (\Act_U)_c )
\]
induces the local homology isomorphism
\begin{align}\label{e:incl_U*}
\iota_{k*}:
\MC_*(\Act_k,\gamma)
\toup^\simeq
\MC_*(\Act_U,\gamma).
\end{align}
For each $R$-modification $\Lagr_R$ of $\Lagr$, with $R>U$, the action $\Act_R$ coincides with $\Act_k$ on $\catU_k$. Hence, we also have an inclusion
\begin{align*} 
j_k:
( (\Act_k)_c\cup\gbra\gamma , (\Act_k)_c )
\hookrightarrow
( (\Act_R)_c\cup\gbra\gamma , (\Act_R)_c ).
\end{align*}
A priori, this inclusion might not induce an isomorphism in homology. However, we have the following statement.

\begin{lem}\label{l:loc_Ton_act}
The inclusion $j_k$  induces the homology isomorphism
\begin{align*}
j_{k*}:
\MC_*(\Act_k,\gamma)
\toup^\simeq
\MC_*(\Act_R,\gamma).
\end{align*}
\begin{proof}
Notice that,  for each $h\in\N$, we have $\catU_k\subset\catU_{hk}$. Since $h k\geq k\geq \bar k(\Lagr_U,c)$, by  corollary~\ref{c:hom_equiv_loc_hom} and the excision property, the inclusion 
\begin{align*} 
\iota_{hk}:
( (\Act_{hk})_c\cup\gbra\gamma , (\Act_{hk})_c )
\hookrightarrow
( (\Act_U)_c\cup\gbra\gamma , (\Act_U)_c )
\end{align*}
induces the  homology isomorphism
$\displaystyle \iota_{hk*}:
\MC_*(\Act_{hk},\gamma)\toup^\simeq
\MC_*(\Act_U,\gamma)$, analogously to $\iota_k$ in \eqref{e:incl_U*}. We define an inclusion $\lambda_h$ that factorizes $\iota_k$  as in  the following   diagram 
\[   
\xymatrix
{
 \dirty{((\Act_k)_c \cup \gbra{\gamma}, (\Act_k)_c)}\ 
\ar@{^{(}->}[rrr]^{\iota_k}
\ar@{^{(}->}[dd]_{\lambda_h}
&&&  
\dirty{( (\Act_U)_c \cup \gbra{\gamma}, (\Act_U)_c )}  \\\\
\dirty{((\Act_{hk})_c \cup \gbra{\gamma}, (\Act_{hk})_c)}
\ar@{^{(}->}[uurrr]_{\iota_{hk}}
}
\]
This inclusion induces the homology isomorphism 
\[\lambda_{h*}=(\iota_{hk*})^{-1}\circ\iota_{k*}:\MC_*(\Act_k,\gamma)\toup^\simeq\MC_*(\Act_{hk},\gamma).\]
Now, let us consider the $R$-modification $\Lagr_R$ of $\Lagr$. By lemma~\ref{l:germ}, for each $h\in\N$, we know that there exists an open neighborhood $\catV_{hk}\subset \catU_{hk}$ of $\gamma$ that is also an open subset of $\Lambda_{hk,\Lagr_R}$, and in particular  $\Act_{hk}|_{\catV_{hk}}  =\Act_{R,hk}|_{\catV_{hk}}  =\Act_R|_{\catV_{hk}}$. Applying once more corollary~\ref{c:hom_equiv_loc_hom} and the excision property, we obtain that the inclusion
\begin{align*} 
j_{hk}':
( (\Act_{hk}|_{\catV_{hk}})_c\cup\gbra\gamma , (\Act_{hk}|_{\catV_{hk}})_c )
\hookrightarrow
( (\Act_R)_c\cup\gbra\gamma , (\Act_R)_c ).
\end{align*}
induces an isomorphism in homology. Finally, consider the following diagram of inclusions.
\[   
\xymatrix
{
\dirty{((\Act_k)_c \cup \gbra{\gamma}, (\Act_k)_c)}
\ar@{^{(}->}[ddrrr]^{j_k}
\ar@{^{(}->}[dd]_{\lambda_h}^{(*)}
\\\\
\dirty{((\Act_{hk})_c \cup \gbra{\gamma}, (\Act_{hk})_c)}\ 
\ar@{^{(}->}[rrr]^{j_{hk}}
&&&  
\dirty{( (\Act_R)_c \cup \gbra{\gamma}, (\Act_R)_c )}  
\\\\
\dirty{( (\Act_{hk}|_{\catV_{hk}})_c\cup\gbra{\gamma},(\Act_{hk}|_{\catV_{hk}})_c )}
\ar@{^{(}->}[uurrr]_{j_{hk}'}^{(*)}
\ar@{_{(}->}[uu]^{\footnotesize
  \begin{tabular}{@{}c@{}}
    {excision} \\ 
    {inclusion} 
  \end{tabular}
}_{(*)}
}
\]
We already know that the inclusions marked with $(*)$ induce  isomorphisms in homology. Therefore,    $j_{hk}$  and $j_k$ also induce isomorphisms in homology.
\end{proof}
\end{lem}

\begin{rem}
As a consequence of the above lemma, we immediately obtain that the local homology groups $\MC_*(\Act_R,\gamma)$ do not depend (up to isomorphism) on the chosen real constant  $R\geq U$ and on the chosen $R$-modification $\Lagr_R$.
\end{rem}

\subsection{Convex quadratic modifications and homological vanishing}\label{s:modif_homVan} 
All the arguments of the previous section can be carried out word by word in an arbitrary period $n\in\N$. Briefly, we introduce the open set
\[ \catU_k\iter n:=\gbra{ \zeta\in\Lambda_k\iter n \biggl|\ \max_{t\in\TT n}\{|\dot\zeta(t)|_{\zeta(t)}\}<U   } \]
and we   define the \textbf{discrete mean Tonelli action} as   
\[\act n_k:=\act n|_{\catU_k\iter n}:\catU_k\iter n\to\R.\] 
Then lemmas~\ref{l:germ} and~\ref{l:loc_Ton_act} go through. Notice that the image of $\catU_k$ under the  $n^{\mathrm{th}}$-iteration map $\itmap n$ is contained in $\catU_k\iter n$. Now, consider $\infty\geq c_2>c_1=c=\Act(\gamma)$, and let us assume that $\gamma$ is not a local minimum of $\Act$. For each $R\geq U$, we have  the following  diagram of inclusions.
\begin{align*}
\xymatrix{
\dirty{( (\Act_k)_{c_1}\cup\gbra{\gamma},(\Act_k)_{c_1} )}\ 
\ar@{^{(}->}[rrr]^{\itmap n}
\ar@{^{(}->}[dd]_{\rho}
&&&
\ \dirty{( (\act n_k)_{c_1}\cup\{\gamma\iter n\},(\act n_k)_{c_1} )}
\ar@{^{(}->}[dd]^{\rho\iter n}
\\\\
\dirty{ ( (\Act_R)_{c_2},(\Act_R)_{c_1} ) }\ 
\ar@{^{(}->}[rrr]^{\itmap n}
&&&
\ \dirty{ ( (\act n_R)_{c_2},(\act n_R)_{c_1} ) }
}
\end{align*}
This latter, in turn, induces the following commutative diagram in homology.
\begin{align*}
\xymatrix{
\MC_*(\Act_k,\gamma) 
\ar[rr]^{\itmap n_*}
\ar[dd]_{\rho_{*}}
&&
\MC_*(\act n_k, \gamma\iter n) 
\ar[dd]^{\rho\iter n_{*}}
\\\\
\Hom_*( (\Act_R)_{c_2},(\Act_R)_{c_1} )  
\ar[rr]^{\itmap n_*}
&&
\Hom_* ( (\act n_R)_{c_2},(\act n_R)_{c_1} ) 
}
\end{align*}
Since $\gamma$ is not a local minimum of the action functional $\Act$, the sublevel $(\Act_R)_{c_1}$ is not empty. Hence, the homological vanishing (theorem~\ref{t:vanish}) guarantees  that for  each $[\mu]\in\MC_*(\Act_k,\gamma)$ and $p\in\N$, there is $\bar n=\bar n(\Lagr_R,[\mu],p)\in\N$ that is a power of $p$ such that  $\itmap{\bar n}_*\circ\rho_{*}[\mu]=0$. Here, we want to remark that we can choose $\bar n$  independent of $R$.

\begin{prop}\label{p:vanishIndep}
With the above assumptions, consider $[\mu]\in\MC_*(\Act_k,\gamma)$ and $p\in\N$. Then, there exist $\bar R=\bar R(\Lagr,[\mu],p)>0$ and $\bar n=\bar n(\Lagr,[\mu],p)\in\N$ that is a power of $p$ such that, for every real $R\geq\bar R$, we have $\itmap{\bar n}_*\circ\rho_{*}[\mu]=0$ in $\Hom_* ( (\act{\bar n}_R)_{c_2},(\act{\bar n}_R)_{c_1} )$. Equivalently, we have that $\itmap{\bar n}_*[\mu]\in\ker \rho\iter{\bar n}_*$.
\begin{proof}
We denote by $\Sigma(\mu)$ the set of singular simplices in $\mu$ together with all their faces, and by $\K\subset\N$ the set of nonnegative integer powers of $p$, i.e. $\K=\gbra{p^n\,|\,n\in\N\cup\gbra0}$. Notice that, for each $q$-singular simplex $\sigma\in\Sigma(\mu)$ we have $\sigma(\Delta^q)\subset\catU_k$, and in  particular 
\begin{align}\label{e:boundW1inf}
 \sup_{\bm z\in\Delta^q}\ 
 \esssup_{t\in\T} \gbra{ \abs{\der{}{t}\, \sigma(\bm z)(t) }_{\sigma(\bm z)(t)}}
  \leq U .
\end{align}
Hence, we can proceed along the line of the proof of the homological vanishing (theorem~\ref{t:vanish}): for each $q$-singular simplex $\sigma\in\Sigma(\mu)$ we get $\bar n=\bar n(\Lagr,\sigma,p)\in\K$ and, for every  $n\in\K$ greater or equal than $\bar n$, a homotopy
\[ P_{\sigma}\iter{n}:[0,1]\times \Delta^{q}\to \WC n, \]
such that
\begin{itemize}
\item[(i)]  $P_{\sigma}\iter{n}(0,\cdot)=\sigma\iter n$,
\item[(ii)] $\act n(P_{\sigma}\iter{n}(s,\bm z))<c_2$ for each $(s,\bm z)\in[0,1]\times\Delta^q$,
\item[(iii)]  $\act n(P_{\sigma}\iter{n}(1,\bm z))<c_1$ for each $\bm z \in \Delta^q$,
\item[(iv)] if $\Act(\sigma(\bm z))<c_1$ for each $\bm z\in\Delta^q$, then $P_{\sigma}\iter{n}(s,\cdot)=\sigma\iter n$ for each $s\in[0,1]$,
\item[(v)] $P_{\sigma\circ F_i}\iter{n}=P_{\sigma}\iter{n}(\cdot,F_i(\cdot))$ for each $i=0,...,q$, where $F_i:\Delta^{q-1}\to\Delta^q$ is the standard affine map onto the $i^{\mathrm{th}}$ face of $\Delta^q$.
\end{itemize}
Notice that, by \eqref{e:boundW1inf} and remark~\ref{r:BangertBd}, there exists a real constant $\bar R(\Lagr,\sigma,p)>0$ such that
\[ 
\sup_{(s,\bm z)\in[0,1]\times\Delta^q}
\esssup_{t\in\TT n} 
\gbra{ 
\abs{
\der{}{t} 
P_\sigma\iter n(s,\bm z)(t) 
}_{P_\sigma\iter n(s,\bm z)(t)}
}
\leq \bar R(\Lagr,\sigma,p).
\]
Now, we define 
\begin{align*}
\bar n&=\bar n(\Lagr,[\mu],p):=\max\gbra{\bar n(\Lagr,\sigma,p)\,|\,\sigma\in\Sigma(\mu)},\\
\bar R&=\bar R(\Lagr,[\mu],p):=\max\gbra{\bar R(\Lagr,\sigma,p)\,|\,\sigma\in\Sigma(\mu)},
\end{align*}
and we consider an $R$-modification $\Lagr_R$ of $R$, with $R\geq\bar R$. Then,  the family of homotopies $\{P_{\sigma}\iter{\bar n}\,|\,\sigma\in\Sigma(\mu)\}$ which we have built  satisfies the following properties: for each $q$-singular simplex  $\sigma\in\Sigma(\mu)$, the homotopy $P_\sigma\iter {\bar n}$ has the form
\[ P_{\sigma}\iter{\bar n}:[0,1]\times \Delta^{q}\to (\act {\bar n}_R)_{c_2}, \]
and moreover
\begin{itemize}
\item[(i)]  $P_{\sigma}\iter{\bar n}(0,\cdot)=\sigma\iter{\bar n}$,
\item[(ii)]  $P_{\sigma}\iter{\bar n}(1,\Delta^q) \subset (\act{\bar n}_R)_{c_1}$,
\item[(iii)] if $\sigma(\Delta^q)\subset (\Act_R)_{c_1}$, then $P_{\sigma}\iter{\bar n}(s,\cdot)=\sigma\iter{\bar n}$ for each $s\in[0,1]$,
\item[(iv)] $P_{\sigma\circ F_i}\iter{\bar n}=P_{\sigma}\iter{\bar n}(\cdot,F_i(\cdot))$ for each $i=0,...,q$.
\end{itemize}
As in the proof of theorem~\ref{t:vanish}, by \cite[lemma 1]{b:BK} we conclude that $\itmap{\bar n}_*\circ\rho_{*}[\mu]=0$ in $\Hom_* ( (\act{\bar n}_R)_{c_2},(\act{\bar n}_R)_{c_1} )$.
\end{proof}
\end{prop}

\section{Proof of theorem~\ref{t:Conley}}\label{s:proof}

We are now ready to prove theorem~\ref{t:Conley}. Throughout the proof, we will adopt the same notation of the previous section. In particular, we will implicitly assume that the mean action functionals $\act n_R$ of the $R$-modifications of $\Lagr$, for each $n\in\N$ and $R>0$, will be defined on the connected component of contractible loops $\WC n\subset\W(\TT n;M)$. Moreover, all the homology groups that will appear from now on are assumed to have coefficients in the field $\Z_2$.

Let us fix a prime $p\in\N$. We will denote by $\K\subset\N$ the set of non-negative integer powers of $p$, i.e. $\K=\gbra{p^n\,|\,n\in\N\cup\gbra0}$. 
We will proceed by contradiction, assuming that the only contractible periodic solutions of the Euler-Lagrange system of $\Lagr$ with period in $\K$ and mean action less than $\ACT$  (the constant chosen in the statement) are
\begin{align*}
\gamma_1,...,\gamma_r.
\end{align*}
Without loss of generality, we can assume that all these  orbits have period $1=p^0$. This can be easily seen in the following way. If $p^n$ is the maximum of their basic periods (and, in particular,  they are all $p^n$-periodic) we can build a Tonelli Lagrangian $\tilde\Lagr:\T\times\Tan M\to\R$ by time-rescaling of $\Lagr$ as
\begin{align*}
\tilde\Lagr(t,q,v):=\Lagr(p^n t,q, p^{-n} v),\s\s\forall (t,q,v)\in\T\times\Tan M.
\end{align*}
For each $j\in\N$, a curve $\tilde\gamma:\R\to M$ is a $j$-periodic solution  of the Euler-Lagrange system of $\tilde\Lagr$ if and only if the reparametrized curve $\gamma:\R\to M$, given by $\gamma(t):=\tilde\gamma(p^{-n} t)$ for each $t\in\R$, is a $p^n j$-periodic solution of the Euler-Lagrange system of $\Lagr$. Moreover, $\tilde\gamma$ and $\gamma$ have the same mean action (with respect to the Lagrangians $\tilde\Lagr$ and $\Lagr$ respectively).

We recall that $N$ is the dimension of the closed manifold $M$. For each $R>0$ and $n\in\N$, the homology of the sublevel $(\act n_R)_\ACT$ is non-trivial in degree $N$, i.e.
\begin{align}\label{e:Hom_sub_neq_0}
\Hom_N((\act n_R)_\ACT)\neq 0,\s\s\forall R>0,n\in\N.
\end{align}
This is easily seen as follows. First of all, notice that the quantity in~\eqref{e:ACTconstant} is finite (due to the compactness of $M$)  and may  be interpreted in the following way: for each integer $n\in\N$, if we denote by $\iota\iter n:M \hookrightarrow \WC n$ the embedding that maps a point to the constant loop at that point, the quantity in \eqref{e:ACTconstant}  is equal to the maximum of the functional $\act n\circ\iota\iter n: M\to \R$. By our choice of the constant $\ACT$,  we have 
\[ 
\act n_R\circ\iota\iter n(q)
=
\act n \circ \iota\iter n(q) <\ACT,
\s\s
\forall q\in M, 
\]
therefore $\iota\iter n$ can be seen as a map of the form $\iota\iter n:M\hookrightarrow(\act n_R)_\ACT$. We denote by $\ev:\W(\T;M)\to M$ the evaluation map, defined by $\ev(\zeta)=\zeta(0)$ for each $\zeta\in\W(\T;M)$. Since $M$ is an $N$-dimensional closed manifold and we consider homology groups with $\Z_2$ coefficients, we have that $\Hom_N(M)$ is non trivial. Therefore, the following commutative diagram readily implies that $\iota\iter n_*$ is a monomorphism, and the claim follows.
\begin{align*}
\xymatrix
{
 & \Hom_N((\act n_R)_\ACT) \ar[dr]^{\ev_*} & \\
0\neq\Hom_N(M) \ar[ur]^{\iota\iter n_*} \ar[rr]^{ \id_{M*}}_{\simeq}  & &  \Hom_N(M)\\
}
\end{align*}

Now, we want to show  that there exists $\gamma\in\gbra{\gamma_1,...,\gamma_r}$ having mean Conley-Zehnder index  $\aiota(\gamma)$ equal to zero (see section~\ref{s:CZLetc}).  In fact, assume by contradiction that $\aiota(\gamma_v)>0$ for each $v\in\gbra{1,...,r}$.  By the first iteration inequality in~\eqref{e:indexiteration}, there exists $n\in\K$ such that 
\begin{align}\label{e:CZLind_gammaJ}
\iota(\gamma_v\iter n)>N,\s\s \forall   v\in\gbra{1,...,r}.
\end{align}
By lemma~\ref{l:apriori}, if we choose a real constant $R>\tilde R(\ACT, n)$,  we know that the only critical points of $\act n_R$ in the open sublevel $(\act n_R)_\ACT$ are  $\gamma_1\iter n,...,\gamma_r\iter n$. By~\eqref{e:CZLind_gammaJ} and corollary~\ref{c:loc_A_interval}, the local homology of $\act n_R$ at the $\gamma_v\iter n$'s vanishes in degree $N$, i.e.
\[ \MC_N(\act n_R,\gamma_v\iter n)=0,\s\s \forall  v\in\gbra{1,...,r}. \]
Then, by the Morse inequality 
\begin{align*}
\dim \Hom_N((\act n_R)_\ACT)
\leq
\sum_{v=1}^r
\dim \MC_N(\act n_R,\gamma_v\iter n),
\end{align*}
we readily obtain that $\Hom_N((\act n_R)_\ACT)=0$, which contradicts~\eqref{e:Hom_sub_neq_0}.

Hence we can assume   that $\gamma_1,...,\gamma_s$, with $1\leq s\leq r$, are  periodic solutions with mean Conley-Zehnder  index equal to zero, while $\gamma_{s+1},...,\gamma_r$ (if $s<r$) are the ones with  strictly positive mean Conley-Zehnder index. By the second iteration inequality in \eqref{e:indexiteration}, we have
\[ \iota(\gamma_v\iter n)+\nu(\gamma_v\iter n) \leq N,\s\s \forall n\in\K,\ v\in\gbra{1,...,s}.\] 
In particular $\iota(\gamma_v\iter n), \nu(\gamma_v\iter n) \in\gbra{0,...,N}$ for each $n\in\K$ and $v\in\gbra{1,...,s}$, and therefore we can find an infinite subset $\K'\subseteq\K$ such that 
\begin{align*}
(\iota(\gamma_v\iter n), \nu(\gamma_v\iter n)) =
(\iota(\gamma_v\iter{m}), \nu(\gamma_v\iter{m})) , 
\s\s
\forall n,m\in\K',\ v\in\gbra{1,...,s}.
\end{align*} 
For each $n,m\in\K'$ with  $n<m$ and  for each real  $R>\max\{ |\dot\gamma_v(t)|_{\gamma_v(t)}\,|\, t\in\T,\  v\in\gbra{1,...,s} \}$, corollary~\ref{p:loc_hom_holds_true} guarantees that the iteration map $\itmap {m/n}$ induces the  homology isomorphism
\begin{align}\label{e:iso_iter_5}
\itmap{m/n}_*:
\MC_*(\act n_R,\gamma_v\iter n)
\toup^\simeq
\MC_*(\act{m}_R,\gamma_v\iter{m}),
\s\s\forall  v\in\gbra{1,...,s}.
\end{align}

If $s<r$, by the first  iteration inequality in~\eqref{e:indexiteration},  there exists  $n\in\K'$ big enough so that  the periodic solutions $\gamma_{s+1},...,\gamma_r$ with strictly positive mean Conley-Zehnder index satisfy
\[
\iota(\gamma_v\iter n)
\geq 
n\,\aiota(\gamma_v)-N>N,\s \s\forall v\in\gbra{s+1,...r}.
\]
By corollary~\ref{c:loc_A_interval}, for each real $R>\max\{|\dot\gamma_v(t)|_{\gamma_v(t)}\,|\, t\in\T,\ v\in\gbra{s+1,...,r}\}$, we have $\MC_N(\act n_R,\gamma_v\iter n)=0$ for each $v\in\gbra{s+1,...,r}$. If $s=r$ we   just set $n:=1$. If we further take $R> \tilde R(\ACT,n)$, where $\tilde R(\ACT,n)$ is the constant given by lemma~\ref{l:apriori},  the Morse inequality
\begin{align*}
0\neq 
\dim \, 
\Hom_N( (\act n_R)_\ACT) 
\leq 
\sum_{v=1}^r 
\dim\, \MC_N(\act n_R,\gamma_v\iter n)
=
\sum_{v=1}^s 
\dim\, \MC_N(\act n_R,\gamma_v\iter n)
\end{align*}
implies that there is a $\gamma\in\gbra{\gamma_1,...,\gamma_s}$ such that
\begin{align*}
\MC_N(\act n_R,\gamma\iter n) \neq 0.
\end{align*}
At this point, let us assume without loss of generality that $1\in\K'$ (this can be achieved by time-rescaling, as we discussed  at the beginning of the proof). Then, by \eqref{e:iso_iter_5}, we have 
\begin{align}\label{e:it_iso_Loc_gamma}
\itmap n_*:
\MC_*(\Act_R,\gamma)
\toup^\simeq
\MC_*(\act n_R,\gamma\iter{n})\neq 0,
\s\s\forall  n\in\K'.
\end{align}
Now, we apply our discretization technique as in section~\ref{s:modif&loc_hom}: we choose $U>0$ as in \eqref{e:choose_U} and we get the discrete mean Tonelli action functional $\act n_k:\catU_k\iter n\to\R$, for some $k\in\N$ sufficiently big and for every $n\in\N$.  For each $R\geq U$, the homology isomorphism induced by the iteration map in~\eqref{e:it_iso_Loc_gamma} fits into the following commutative diagram
\begin{align*}
\xymatrix{
\MC_*(\Act_k,\gamma)
\ar[rr]^{\itmap n_*}
\ar[dd]_\simeq
&&
\MC_*(\act n_k,\gamma\iter n)
\ar[dd]^\simeq
\\\\
\MC_*(\Act_R, \gamma)
\ar[rr]^{\itmap n_*}_\simeq
&&
\MC_*(\act n_R,\gamma\iter n)
}
\end{align*}
In this diagram, the vertical arrows are isomorphisms induced by inclusions (see lemma~\ref{l:loc_Ton_act}). Therefore, the iteration map induces a homology isomorphism
\begin{align*}
\itmap n_*:
\MC_*(\Act_k,\gamma)
\toup^\simeq
\MC_*(\act n_k,\gamma\iter{n})\neq 0,
\s\s\forall  n\in\K'.
\end{align*}

Let $c_1=\Act(\gamma)$ and let $\epsilon>0$ be small enough so that $c_2:=c_1+\epsilon<\ACT$ and there are no $\gamma_v\in\gbra{\gamma_1,...,\gamma_r}$   with $\Act(\gamma_v)\in (c_1,c_2)$. By  lemma~\ref{l:apriori}, for every $n\in\K'$ and $R>\tilde R(\ACT,n)$, the action functional $\act n_R$ does not have any critical point with critical value in $(c_1,c_2)$. This implies that the inclusion 
\begin{align*}
((\act n_R)_{c_1} \cup \gbra\gamma, (\act n_R)_{c_1})\hookrightarrow 
((\act n_R)_{c_2}, (\act n_R)_{c_1})
\end{align*}
induces a monomorphism in homology (see \cite[theorem~4.2]{b:Ch}), and therefore the inclusion
\begin{align*}
\rho\iter n:
((\act n_k)_{c_1} \cup \gbra{\gamma}, (\act n_k)_{c_1})\hookrightarrow 
((\act n_R)_{c_2}, (\act n_R)_{c_1}),
\end{align*}
induces a monomorphism in homology as well. Summing up, for each  $R >\tilde R(\ACT,n)$ and $n\in\K'$, we have obtained the following commutative diagram.
\begin{align*}
\xymatrix{
0\neq\dirty{\MC_N(\Act_k,\gamma)}\ 
\ar[rr]^{\itmap n_*}_\simeq
\ar@{^{(}->}[dd]_{\rho\iter1_*}
&&
\ \dirty{\MC_N(\act n_k,\gamma\iter n)}
\ar@{^{(}->}[dd]^{\rho\iter n_*}
\\\\
\Hom_N((\Act_R)_{c_2}, (\Act_R)_{c_1})\ 
\ar[rr]^{\itmap n_*}
&&
\ \Hom_N((\act n_R)_{c_2}, (\act n_R)_{c_1})
}
\end{align*}
This diagram contradicts the homological vanishing (proposition~\ref{p:vanishIndep}). In fact, since the local homology group $\MC_N(\Act_k,\gamma)$ is nontrivial and $N>0$, $\gamma$ is not a local minimum of $\Act_k$. For each  nonzero $[\mu]\in\MC_N(\Act_k,\gamma)$, there exist  $\bar R=\bar R(\Lagr,[\mu],p)\in\R$ and  $\bar n=\bar n(\Lagr,[\mu],p)\in\K$ such that, for each real  $R\geq\bar R$ and for each  $n\in\K$ greater or equal than $\bar n$, we have 
\[ 
\itmap n_*\circ\rho\iter1_*[\mu]= 
\itmap{n/\bar n}_*\circ \underbrace{\itmap{\bar n}_*\circ\rho\iter1_*[\mu]}_{=0}=
0,
\]
therefore $\itmap n_*[\mu]\in\ker \rho\iter n_*$.


\vspace{1cm}


\begin{thebibliography}{00000}

\footnotesize

\bibitem[Ab]{b:AbMaslov} \textsc{A. Abbondandolo}, \emph{On the Morse index of Lagrangian systems}, Nonlinear Anal., \textbf{53}, (2003),  551--566.

\bibitem[AF]{b:AbFig} \textsc{A. Abbondandolo,  A. Figalli}, \emph{High action orbits for Tonelli Lagrangians and superlinear Hamiltonians  on compact configuration spaces}, J. Diff. Eq. \textbf{234}  (2007), 626--653. 

\bibitem[AS1]{b:AbSc} \textsc{A. Abbondandolo, M. Schwarz}, \emph{On the Floer homology of cotangent bundles},  Comm. Pure Appl. Math.  \textbf{59}  (2006),  no. 2, 254--316.


\bibitem[AS2]{b:AS2} \textsc{A. Abbondandolo, M. Schwarz}, \emph{A smooth pseudo-gradient for the Lagrangian action functional}, Adv. Nonlinear Stud. \textbf{9} (2009) 597--623.


\bibitem[Ba]{b:Ba} \textsc{V. Bangert}, \emph{Closed geodesics on complete surfaces}, Math. Ann. \textbf{251}  (1980), 83--96. 

\bibitem[BK]{b:BK} \textsc{V. Bangert, W. Klingenberg}, \emph{Homology generated by iterated closed geodesics}, Topology  \textbf{22} (1983), 379--388.


\bibitem[Be]{b:Be} \textsc{V. Benci}, \emph{Periodic solutions of Lagrangian systems on a compact manifold}, J. Diff. Eq. \textbf{63} (1986), 135--161. 

\bibitem[BGH]{b:BGH} \textsc{G. Buttazzo, M. Giaquinta, S. Hildebrandt}, \emph{One-dimensional variational 
problems: an introduction}, Oxford Lecture Series in Mathematics and its Applications \textbf{15} (1998), 
Oxford Univ. Press.

\bibitem[CT]{b:CT} \textsc{J. Campos, M. Tarallo}, \emph{Large minimal period orbits of periodic autonomous systems}, Nonlinearity \textbf{17} (2004), no. 1, 357--370. 

\bibitem[Ch]{b:Ch} \textsc{K.-C. Chang}, \textbf{Infinite Dimensional Morse Theory and Multiple Solution Problems}, Birkh\"auser, 1993.

\bibitem[Co]{b:Co} \textsc{C.C. Conley}, \emph{Lecture at the University of Wisconsin}, April 6, 1984.

\bibitem[CZ]{b:CZ} \textsc{C.C. Conley, E. Zehnder}, \emph{Morse-type index theory for flows and periodic solutions for Hamiltonian equations},  Comm. Pure Appl. Math.  \textbf{37}  (1984),  no. 2, 207--253.

\bibitem[CZ2]{b:CZ2} \textsc{C.C. Conley,  E. Zehnder}, \emph{A global fixed point theorem for symplectic maps and subharmonic solutions of Hamiltonian equations on tori}, Proc. Sympos. Pure Math. \textbf{45},  Part 1, Amer. Math. Soc., Providence, RI, 1986. 


\bibitem[Fa]{b:Fa} \textsc{A. Fathi}, \textbf{Weak KAM theorem in Lagrangian dynamics}, 2008, preliminary version number 10.


\bibitem[FS]{b:FS} \textsc{U. Frauenfelder, F. Schlenk}, \emph{Hamiltonian dynamics on convex symplectic manifolds}, Israel J. Math. \textbf{159} (2007), 1--56.



\bibitem[Gi]{b:Gi} \textsc{V.L. Ginzburg}, \emph{The Conley conjecture}, Ann. of Math. \textbf{172} (2010), 1127--1180.

 
\bibitem[GG]{b:GG} \textsc{V.L. Ginzburg, B.Z. G\"urel}, \emph{Action and index spectra and periodic orbits in Hamiltonian dynamics}, Geom. Topol. \textbf{13} (2009), 2745--2805.


\bibitem[GM]{b:GM} \textsc{D. Gromoll, W. Meyer}, \emph{On differentiable functions with isolated critical points}, Topology \textbf{8} (1969), 361--369.

\bibitem[G\"u]{b:Gu} \textsc{B.Z. G\"urel}, \emph{Totally non-coisotropic displacement and its applications to Hamiltonian 
dynamics}, preprint (2007), to appear in Comm.\ Contemp.\ Math.

\bibitem[Ha]{b:Hat} \textsc{A. Hatcher}, \textbf{Algebraic topology}, Cambridge University Press, 2002.


\bibitem[Hi]{b:Hi} \textsc{N. Hingston}, \emph{Subharmonic solutions of Hamiltonian equations on tori}, Ann. of Math. \textbf{170} (2009), no.~2, 529--560.

\bibitem[HZ]{b:HZ} \textsc{H. Hofer, E. Zehnder}, \textbf{Symplectic invariants and
 Hamiltonian dynamics}, Birkh\"{a}user, Basel, 1994.

\bibitem[Kl]{b:Kl} \textsc{W. Klingenberg}, \textbf{Lectures on closed geodesics}, Springer-Verlag, Berlin, 1978.

\bibitem[LL1]{b:LL1} \textsc{C. Liu, Y. Long}, \emph{An optimal increasing estimate on the Maslov-type indices for 
iterations}, Chinese Sci. Bull. \textbf{42} (1997), 2275-2277 (Chinese edition), \textbf{43} (1998), 
1063-1066 (English edition). 

\bibitem[LL2]{b:LL2} \textsc{C. Liu, Y. Long}, \emph{Iteration inequalities of the Maslov-type index theory with applications}, J. Diff. Eq. \textbf{165}  (2000), no.\ 2, 355--376. 

\bibitem[Lo1]{b:LoMaslov} \textsc{Y. Long},  \emph{A Maslov-type index theory for symplectic paths}, Top. Methods Nonlinear Anal. \textbf{10} (1997), 47--78.

\bibitem[Lo2]{b:Lo} \textsc{Y. Long}, \emph{Multiple periodic points of the Poincar\'e map of Lagrangian systems on tori}, Math. Zeit. \textbf{233} (2000), 443--470.

\bibitem[LA]{b:LoAn} \textsc{Y. Long, T. An}, \emph{Indexing domains of instability for Hamiltonian systems}, NoDEA Nonlinear Differential 
Equations Appl. \textbf{5} (1998), 461--478. 

\bibitem[LL]{b:LL} \textsc{Y. Long, G. Lu}, \emph{Infinitely many periodic solution orbits of 
autonomous Lagrangian systems on tori}, J. Funct. Anal. \textbf{197} (2003), 301--322.

\bibitem[Lu]{b:Lu} \textsc{G. Lu}, \emph{The Conley conjecture for Hamiltonian systems on the cotangent bundle and its analogue for Lagrangian systems},  J. Funct. Anal. \textbf{259} (2009), 2967--3034.


\bibitem[LW]{b:LW} \textsc{G. Lu, M. Wang}, \emph{Existence of infinitely many brake orbits of Lagrangian autonomous systems on tori}, Int. Math. Forum \textbf{2}  (2007), no. 45-48, 2333--2338. 


\bibitem[Ma]{b:Ma} \textsc{J.N. Mather}, \emph{Action minimizing invariant measures for positive definite Lagrangian systems}, Math. Zeith. \textbf{207} (1991), 169--207.


\bibitem[Mi]{b:MiMo} \textsc{J. Milnor}, \textbf{Morse theory}, Ann. of Math. Studies \textbf{51} (1963), Princeton University Press.

\bibitem[Ra]{b:Ra} \textsc{H.-B. Rademacher}, \textbf{Morse-Theorie und geschlossene Geod\"atische},  habilitationsschrift, Bonner Mathematische 
Schriften \textbf{229} (1992).


\bibitem[Sc]{b:Sch2} \textsc{M. Schwarz}, \emph{On the action spectrum for closed symplectically aspherical manifolds}, Pacific J. Math. \textbf{193} (2000), 419--461. 

\bibitem[SZ]{b:SZ} \textsc{D. Salamon, E. Zehnder}, \emph{Morse theory for periodic solutions of Hamiltonian systems and the Maslov index}, Comm. Pure Appl. Math., \textbf{45} (1992), 1303--1360. 


\bibitem[Vi1]{b:Vi} \textsc{C. Viterbo}, \emph{A new obstruction to embedding Lagrangian tori}, Invent. Math. \textbf{100} (1990) 301--320.

\bibitem[Vi2]{b:Vi2} \textsc{C. Viterbo}, \emph{Symplectic topology as the geometry of generating functions}, Math. Ann. \textbf{292} (1992), 685--710. 


\end{thebibliography}
\end{document}